\documentclass[12pt]{article}
\usepackage{amsmath, amsfonts, amssymb,amsthm, amscd}

\usepackage{pictexwd,dcpic}
\usepackage{graphicx}
\usepackage{epsf}
\usepackage{epsfig}
\usepackage{rlepsf}
\usepackage{psfrag}
\usepackage{color}
\usepackage{hyperref}

\newtheorem{thm}{Theorem}[section]
\newtheorem{lem}[thm]{Lemma}

\newtheorem{prop}[thm]{Proposition}

\theoremstyle{definition}
\newtheorem{defn}[thm]{Definition}

\theoremstyle{definition}
\newtheorem{remark}[thm]{Remark}

 \newcommand{\el}{\text{L}}

\newcommand{\abs}[1]{\left\vert#1\right\vert}

\newcommand{\ssc}{\text{sc}}

\renewcommand{\epsilon}{\varepsilon}

\newcommand{\what}{\widehat}
\newcommand{\wh}{\widehat}

\newcommand{\ov}{\overline}

\newcommand{\Z}{{\mathbb Z}}

\newcommand{\id}{\operatorname{id}}

\newcommand{\de}{\text{deg}}

\newcommand{\supp}{\operatorname{supp}}

\newcommand{\R}{{\mathbb R}}

\providecommand{\ker}[1]{$\text{ker}\ {#1}$}

\newcommand{\N}{{\mathbb N}}
\newcommand{\Q}{{\mathbb Q}}

\def\abs#1{\mathopen|#1\mathclose|}

{\catcode`"=12 \gdef\hex{"}}

\mathchardef\laplace=\hex0001

\mathchardef\nabla=\hex0272

\makeatletter

\def\@@dalembert#1#2{\setbox0\hbox{$#1\mathrm I$}

  \vrule height\ht0 depth\z@ width.04\ht0

  \rlap{\vrule height\ht0 depth-.96\ht0 width.8\ht0}

  \vrule height.1\ht0 depth\z@ width.8\ht0

  \vrule height\ht0 depth\z@ width.1\ht0 }

\def\dalembert{\mathbin{\mathpalette\@@dalembert{}}\,}

\makeatother

\begin{document}

\title{{ A General Fredholm Theory {III}:\\
Fredholm Functors and Polyfolds}}

\author{ H. Hofer\footnote{Research partially supported
by NSF grant  DMS-0603957.}\\
New York University\\ USA\and K. Wysocki \footnote{Research
partially supported by  NSF grant DMS-0606588. }\\Penn State University\\
USA\and E. Zehnder
 \footnote{Research partially supported by  TH-project.}\\ETH-Zurich\\Switzerland}

\maketitle \tableofcontents
\section{Introduction and Survey of Results}
This paper is the third in a series of papers devoted to  a generalized
Fredholm theory. In Part I, \cite{HWZ2}, the
``splicing-based differential geometry'' is developed. In this theory  the familiar  local models for
spaces (open subsets of finite-dimensional spaces or
Banach spaces)  are replaced by  more general local models, namely by
open subsets of  so-called splicing cores. Further, the notion of
smoothness, as well as the notion of a smooth map,  is generalized
from the standard notion in finite dimensions to infinite dimensions
in a new way. The generalization  allows to extend the category of manifolds to a category of new
smooth objects which open up the possibility to identify new structures in
situations which could not be handled before. For example,  the moduli space in Symplectic Field Theory (SFT) can be
viewed as the zero-set of a Fredholm section (in a generalized
sense) defined on bundles on spaces belonging to our new category. Under the appropriate
transversality assumptions
the solution sets are  still
`old-fashioned' manifolds or  orbifolds and  in bad cases branched
manifolds or branched orbifolds. It still makes sense to talk about
orientations in all cases and the structures suffice to  establish  a
theory of integration in which Stokes' theorem holds true. This is all
one needs in order to define invariants.

{\bf Acknowledgement:} We  would like to thank Peter Albers, Yakov Eliashberg, Eleny Ionel,  Dusa McDuff and Tom Weinmann for very  useful  discussions.

\subsection{Background}

If one uses  the new local models, the recipe for defining
manifolds produces  so-called M-polyfolds which are the
starting point for further developments.  In the following we assume the reader to be familiar with the concepts and results  in \cite{HWZ2}. More
specifically,  we assume familiarity with  the notion of an
M-polyfold, its degeneration index, the notion of a strong bundle
splicing and that of a filled section. Moreover,  with the notion of an $\ssc^+$-section of a strong M-polyfold bundle, the notion of linearization, the notion of being
linearized Fredholm,  and the definition of the Fredholm index for
an sc-Fredholm operator.

In the first part of the present paper we shall
develop the theory of ep-groupoids and polyfolds, We shall modify the approach to
orbifolds based on  \'etale proper Lie groupoids, as
presented  in \cite{Mj,MM}  and  replace \'etale
proper Lie groupoids by ep-groupoids which are based on M-polyfolds.
In the second part of the paper we generalize  the Fredholm theory in
M-polyfolds  from \cite{HWZ3}  to a theory of Fredholm functors and polyfold Fredholm sections. Fredholm
functors can be viewed as Fredholm sections  compatible with local
symmetries  represented by the morphisms. In general,  it is not possible
to bring a Fredholm functor into a general position by a functorial
perturbation. However, if  multi-valued functors are admitted such
perturbations become possible. (Here we extend  ideas from
\cite{CRS} to the  functorial setting).  The price to pay is that the solution sets in case of
transversality are neither manifolds nor orbifolds, but merely
weighted branched submanifolds or suborbifolds. Neverthless these
objects have enough structure in order to establish a  well-defined integration
theory for so-called sc-differential forms as demonstrated in  \cite{HWZ7}. The integration on a branched  ep-subgroupoid in \cite{HWZ7} is  used to construct invariants. It is  related to the  recent work by McDuff in  \cite{Mc}. The results of the present paper form the basis for the
application to SFT given in \cite{HWZ4} and \cite{HWZ5}.

\subsection{Survey of the Main Results}
We describe now some of the main results postponing  the
precise definitions to  the later sections. A manifold is a second
countable paracompact space with an additional structure (the
overhead) of equivalence classes of atlases of smoothly compatible
charts.  One can formalize the equivalence of atlases in a
category-theoretic way as some kind of Morita-equivalence. From this
point of view the step to an orbifold structure on a topological space is
small.  Here  again we have an underlying topological space equipped with an the overhead of a Morita-equivalence class
of \'etale proper Lie groupoids as described in  \cite{Mj}. Polyfolds are second
countable paracompact spaces with an overhead of Morita-equivalence
classes of \'etale proper M-polyfold groupoids. In the polyfold theory the role of an atlas
for a manifold is played  by a pair $(Q,\beta)$ in which
$Q$ is an ep-groupoid (which generalizes the notion of an
\'etale proper Lie groupoid) and a homeomorphism
$\beta:\abs{Q}\to  Z$. An ep-groupoid is,  in particular,  a category
and $\abs{Q}$ denotes the orbit space in which two objects are identified
if there is a morphism between them. However,  one needs the whole
ep-groupoid to encode the smooth structure on the polyfold.

In  the following description of the main results we ignore the all
important overhead and just note that it defines a smooth structure
(in some new sense) on an otherwise second countable paracompact
space $Z$. The topological space $Z$ equipped with this for the moment
suppressed additional structure is called a polyfold. If one accepts
a polyfold as a generalization of a (possibly infinite-dimensional)
orbifold the results surveyed in the following appear familiar. In
fact our results can be viewed as  generalizations of known
results in Banach manifolds to much more general spaces.

It is a part of its structure that a polyfold $Z$ is equipped with a filtration
$$Z=Z_0\supset Z_1\supset Z_2\supset  \cdots  \supset Z_i\supset Z_{i+1}\supset \cdots  \supset  Z_{\infty}:=\bigcap_{i\geq 0}Z_i$$
into subsets $Z_i$ of $Z$ which have topologies such that the inclusions  $Z_{i+1}\to Z_i$ are continuous and enjoy some compactness properties. Moreover,  $Z_\infty$ is dense in every space $Z_i$.  In fact every $Z_m$ carries some smooth structure (again in the new sense) as well and smooth maps between polyfolds have to preserve
these levels of smoothness. We can define strong polyfold bundles
$p:W\rightarrow Z$. The space $W$ carries a double
integer-filtration $W_{n,k}$ for  $0\leq k\leq n+1$, where we may
view $k$ as the fiber regularity. A smooth section $f$ of the bundle
$p:W\rightarrow Z$ maps $Z_m$ to $W_{m,m}$. The collection of smooth
sections is denoted  by $\Gamma(p)$. There is an additional class of
so-called $\ssc^+$-sections. They  are smooth sections mapping $Z_m$
to $W_{m,m+1}$. Due to a compactness property of the fiber-wise
embeddings $W_{m,m+1}\rightarrow W_{m,m}$, the space of
$\ssc^+$-sections can be viewed as a well-defined universe of
compact perturbations. Of particular importance will be the set of
$\ssc^+$-multisections.
An  $\ssc^+$-multisection can be identified with a functor
$\lambda:W\rightarrow {\mathbb Q}^+$ such that near every point $z\in Z$
there exist a finite number of $\ssc^+$-sections $s_i$ for $i\in I$,
and associated  positive rational weights $\sigma_i$ satisfying for $w\in W$,
$$
\lambda(w)=\sum_{\{i\in I \vert \,  s_i(p(w))=w\}} \sigma_i.
$$
The sum over the empty set is defined to be zero. As it turns out two
such multisections $\lambda$ and $\tau$ can be added resulting in
$\lambda\oplus\tau$. Our notion of multisections  generalizes  ideas in \cite{CRS}
(where group actions were studied) to a functorial context. In
\cite{HWZ3} we have introduced  Fredholm sections which now will be generalized
to the polyfold context.   For a pair $(f,\lambda)$ in which
$f$ is a Fredholm section and $\lambda$ an
$\ssc^+$-multisection, we define the solution set $S(f,\lambda)$ to be the set
$$
S(f,\lambda)=\{z\in Z\ |\ \lambda(f(z))>0\}.
$$
By
$$S(f):=\{z\in Z\vert \, f(z)=0\}$$
 we shall denote the solution set of the Fredholm section $f:Z\to W$,  where $0$ is the zero section of the bundle $p$.  We should point out
 that the fiber of a strong polyfold bundle does not have a
linear structure but it has a preferred section $0$.
 We also define the
notion of an auxiliary norm $N:W_{0,1}\rightarrow {\mathbb R}^+$
in order  to measure the size of a $\ssc^+$-section or
$\ssc^+$-multisection. (It is not a real norm but related to a norm
in the overhead).
\begin{thm}[Compactness]
Let $p:W\rightarrow Z$ be a strong polyfold  bundle possibly with  boundary with
corners and  let $f$ be  a proper Fredholm section. Then there exists for a
given auxiliary norm  $N$ on the bundle $p$ an open neighborhood $U$ of the solution set  $S(f)$,  so that
for every $\ssc^+$-multisection $\lambda$ having its  support in $U$ and
satisfying  $N(\lambda)\leq 1$,  the solution set $ S(f,\lambda)$ is compact.
\end{thm}

We should point out that the boundary of
a polyfold has very little to do with the set theoretic boundary of
the underlying topological space (whatever it means in any given
context). For example,  the subspace $Z$ of ${\mathbb R}^2$ given by
$Z=((-\infty,0)\times {\mathbb R})\cup ([0,1]\times \{0\})$ with the
induced topology admits a smooth polyfold structure without boundary
 (One might think that $(1,0)$ is a boundary point!).

The next result shows that we can perturb the multisections  to obtain as a solution set a
branched suborbifold of $Z$. The advantage is
that we can integrate over these, once they are equipped with an orientation, as
 is demonstrated in \cite{HWZ7}.
\begin{thm}[Perturbation]
Let $p:W\rightarrow Z$ be a strong polyfold bundle possibly with
boundary with corners and  let $f$ be a proper Fredholm section of $p$. We assume that the sc-smoothness structure of the polyfold $Z$ is based on separable Hilbert spaces.  Let
 $N$ be  an auxiliary norm and $U$ an open neighborhood of  the solution set $S(f):=\{z\in Z\vert \,  f(z)=0\}$
so that the pair $(N,U)$ controls compactness. Then for every  $\ssc^+$-
multisection $\lambda$ having  support in $U$ and satisfying
$N(\lambda)<\frac{1}{2}$,  and  for every $\varepsilon\in (0,\frac{1}{2})$,  there
exists an $\ssc^+$-multisection $\tau$ having support in $U$ and satisfying
$N(\tau)<\varepsilon$,  so that  the solution set $S(f,\lambda\oplus\tau)$ is a
compact branched suborbifold with boundary with corners  equipped  with the natural weight function $w:S(f,\lambda\oplus\tau)\rightarrow (0,\infty)\cap {\mathbb Q}$ defined
by
$$
w(z)=(\lambda\oplus\tau)(f(z)).
$$
If  the Fredholm section $f$ is oriented, the  solution set $S(f,\lambda\oplus\tau)$ has a natural orientation.
\end{thm}
A branched suborbifold is essentially a pair $(S,w)$ in which  $S$ is  a subset of  $Z$ as the
above  $S(f,\lambda\oplus\tau)$ and $w$ is a map associating with a point in $S$ a positive
rational weight. If $f$ is a  proper Fredholm section and $\lambda_i$, $i=0,1$,
are generic $\ssc^+$-multisections,  then there exists a generic
family $\lambda_t$ interpolating between them so that  the solution set $\{(t,z)\in
[0,1]\times Z\ | \lambda_t(f(z))>0\}$ is a smooth branched
suborbifold with boundary with corners interpolating between the solution sets $S_0$
and $S_1$. Next assume
$\partial Z=\emptyset$.  Using  determinant bundles one can  introduce the notion of an orientation
$\mathfrak{o}$ for a Fredholm section $f$ leading to the notion of
an oriented Fredholm section $(f,\mathfrak{o})$. We summarize the
results needed   in the  appendix.  The construction of
determinant bundles are in principle well-known, see for example \cite{DK,FH},
however, in the polyfold context they are not entirely standard
since the linearized Fredholm sections do not depend continuously as
operators on the points at which  they are linearized. This requires
some extra work carried out in \cite{HWZ8}.  Transversal pairs $(f,\lambda)$,  pairs in good position and pairs in general position occuring in the following result are defined in Definition \ref{DER1} below.

\begin{thm}[Invariants]\label{invar}
Let $p:W\rightarrow Z$ be a strong polyfold bundle without boundary
and $f$ a proper oriented Fredholm section, where $Z$ is build on separable Hilbert spaces. Assume that $N$ is an
auxiliary norm and $U$ an open neighborhood of  the solution set $S(f)$ so that
the pair $(N,U)$ controls compactness. Then there is a well-defined map
$$
\Phi_f:H_{dR}^{\ast}(Z,{\mathbb R})\rightarrow {\mathbb R}
$$
defined on the deRham cohomology group $H_{dR}^{\ast}(Z,{\mathbb R})$ and having the following property. For every generic solution set $S=S(f,\lambda)$
(i.e. $(f,\lambda)$ is a transversal pair),  where $S$ is equipped
with the weight function $w(z)=\lambda(f(z))$ and  where  $\lambda$ has support in
$U$ and satisfies  $N(\lambda)<1$, the map $\Phi_f$ is represented by the formula
$$
\Phi_f([\omega]) = \int_{(S,w)}\omega = \mu_{\omega}^{(S,w)}(S).
$$
Moreover, if $t\rightarrow f_t$ is a proper homotopy of oriented
Fredholm sections then $\Phi_{f_0}=\Phi_{f_1}$.
\end{thm}

Here $\mu_\omega^{(S,w)}$ is a natural signed measure associated with
an sc-differential form $\omega$ on the polyfold $Z$ and the
weighted branched suborbifold $(S,w)$ (of course the measure can
only be nonzero provided the degree of the form and the dimension of
$S$ match). The underlying measure space is $(S,{\mathcal L}(S))$. The $\sigma$-algebra
 ${\mathcal L}(S)$ of subsets of $S$  is a natural generalisation of the Lebesgue $\sigma$-algebra on a smooth
manifold. We refer to  \cite{HWZ7} for more details.

For example,   as sketched in section \ref{gromovwitten}, the disjoint
union of the Gromov-compactified moduli spaces of pseudoholomorphic
curves with varying arithmetic genus and representing the various
second homology classes  in a compact symplectic manifold
$(M,\omega)$ can be viewed as the zero set of a Fredholm section $f$
of some strong polyfold bundle $p:W\rightarrow Z$, which in  every
connected component is proper. The evaluation map $ev_l:Z\rightarrow M$ at the
$l$-marked point  is smooth (in the new
sense) and pulls back every  differential form on $M$ to an sc-differential   form $\omega$
on $Z$. Also the map which associates with  a point in $Z$ the
underlying stable part of the domain defines a smooth map into the
Deligne-Mumford stack and pulls back differential forms on the stack to  sc-differential forms
$\omega$ on $Z$. Wedges of suitable forms can be integrated over a
(generic) branched suborbifold which is the solution set for a
suitable $\ssc^+$-multisection perturbation, and organizing the
data in the usual way we obtain the GW-potential as defined in  \cite{MS2}.
That the GW-theory fits into our framework will be shown in detail
\cite{HWZ4}, though the main point of \cite{HWZ4} is the
construction of the polyfold structures  in the presence of
nodes. The understanding of the nodes  presents already all the analytical difficulties
related to  the phenomenon  called breaking of trajectories  occurring,   for example,  in Floer Theory and in  SFT.

 Using the previous theorem  one  can associate with  a proper Fredholm section a
 ${\mathbb Q}$-valued degree as follows. If $f$ is an oriented proper Fredholm section of Fredholm index
$0$ we can integrate $[1]\in H^0_{dR}(Z,{\mathbb R})$, i.e. the
cohomology class of the constant $1$-function. Then we define
$$\de (f)=\Phi_f([1]).$$ This is a degree and takes, as one can
show, values in ${\mathbb Q}$ and has the usual properties of a degree.

There is also a version for the boundary case. Consider the
inclusion $j:\partial Z\rightarrow Z$ which  restricted to the local faces, is sc-smooth.   We can introduce the notion
of a $\ssc$-differential form on $\partial Z$ by taking
$\ssc$-differential forms on the local faces which are compatible at
their intersections. Let us denote by $\Omega^k(j)$ the collection
of pairs $(\omega,\tau)\in \Omega^k(Z)\oplus\Omega^{k-1}(\partial
Z)$.  The Cartan derivative $d$ is defined as $d(\omega,\tau)=(d\omega,j^\ast\omega-d\tau)$ and the  associated deRham cohomology group $H_{dR}^\ast(j)$ is denoted
 by $H^\ast_{dR}(Z,\partial Z)$.

\begin{thm}[Invariants in case of boundary] \label{poil}
Assume that $f$ is a proper, oriented Fredholm section of the strong
polyfold bundle $p:W\rightarrow Z$, where $Z$ is built on separable Hilbert spaces. Then there exists a well-defined
map $\Psi_f:H^\ast_{dR}(Z,\partial Z)\rightarrow {\mathbb R}$ so
that the following holds. If $(N,U)$ is a pair controlling
compactness, where $N$ is an auxiliary norm and $U$ a corresponding
open neighborhood of the solution set $S(f)$,  then for any $\ssc^+$-multisection
$\lambda$ having  support in $U$ and satisfying $N(\lambda(z))<1$ for
all $z$ so that $(f,\lambda)$ is in general position (i.e. in
particular,  the associated solution set is a compact, weighted smooth
branched suborbifold with boundary with corners),  the map
$\Psi_f$ has the representation

$$
\Psi_f([\omega,\tau]) := \int_{(S(f,\lambda),\lambda_f)}\omega-
\int_{(\partial S(f,\lambda),\lambda_f)}\tau.
$$
Moreover, if $t\mapsto f_t$ is an sc-smooth, oriented, proper
homotopy of sc-Fredholm sections, then
$$
\Psi_{f_0}=\Psi_{f_1}.
$$
\end{thm}

As already mentioned there are approaches to the structure of moduli spaces of stable maps which are different from ours. In some approaches so called Kuranishi structures are used, see \cite{FO} , \cite{FOOO} and \cite{LuT}.
We would like to point out that it is quite straight forward to construct a `forgetful' functor  from the polyfold Fredholm theory to a class of Kuranishi structures, see \cite{H2}.

\subsection{Sketch of the Application to Gromov-Witten}\label{gromovwitten}

As an illustration of the new concepts in this paper we sketch an application to the Gromow-Witten (GW)-invariants
referring to  \cite{HWZ4} for the details and the proofs. The GW-invariants are invariants of symplectic manifolds deduced from the structure of stable pseudoholomorphic maps from Riemann surfaces to the symplectic manifold.  The  construction of GW-invariants for general symplectic manifolds goes back to
Fukaya-Ono in \cite{FO} and  Li-Tian in  \cite{LiT}. Earlier work for special symplectic manifolds are due to Ruan in  \cite{R1} and \cite{R2}.
Recently,  Cieliebak and Mohnke proved the genus zero case  in \cite{CM}
using Donaldson's theory  \cite{Do} of codimension two symplectic hypersurfaces
in order to establish the crucial  transversality property. Our approach is  quite
different.  As a consequence of general principles, compact smooth moduli spaces are produced over which one can integrate.
In  the GW-case, the theory applies for arbitrary genus and since the theory also solves all the smoothnes problems automatically, many results, as for example  the composition formula for the GW-invariants, do not require extra work.

We start our sketch recalling some concepts from the theory of Riemann surfaces. Nodal Riemann surfaces show up in the compactification of the moduli space of compact Riemann surfaces $S$.  A nodal Riemann  surface is a multiplet
$$(S, j, M, D)$$
in which the pair $(S, j)$ is a closed, not necessarily connected Riemann surface $S$ equipped with the complex structure $j$. So, $S$ is the disjoint union of finitely many connected compact Riemann surfaces.  The set $M\subset S$ is an ordered finite subset of points, called  {\bf marked points}.  The unordered set $D$ consists of finitely many unordered pairs  $\{x, y\}$ of points in $S$ satisfying $x\neq y$ and called {\bf nodal pairs}. Moreover, $\{x, y\}\cap \{x',y'\}\neq \emptyset$ implies that the two sets are equal. The points of a nodal pair may belong to the same component or to different components of $S$. Denoting by $\abs{D}$ the collection of all the points of $S$ contained in the nodal pairs $D$, we assume, in addition, that $\abs{D}\cap M=\emptyset$. Points in $\abs{D}$ are called {\bf nodal points}. A {\bf special point} of the nodal Riemann surface is a point in $S$ which is either a nodal point or a marked point. The nodal Riemann surface $(S, j, M, D)$ is called {\bf connected} if the topological spaced obtained by identifying $x\equiv y$ for all nodal pairs $\{x, y\}\in D$ is a connected space. In this terminology it is possible that the nodal surface $(S, j, M, D)$ is connected but the Riemann surface $S$ has several connected components $C$, namely its domain components. The arithmetic genus $g_{\text{a}}$ agrees with the genus of the connected compact Riemann surface obtained by properly gluing the components $C$ of $S$ at all the nodes.
The arithmetic genus $g_{\text{a}}$ of the connected nodal Riemann  surface $(S, j, M, D)$ is the nonnegative integer $g_{\text{a}}$ defined  by
$$g_{\text{a}}=1+\sharp D +\sum_{C}(g(C)-1)$$
where $\sharp D$ is the number of nodal pairs in $D$ and where the sum is taken over the finitely many domain components $C$ of the Riemann surface $S$. Two connected nodal Riemann surfaces
$$(S, j, M, D)\quad \text{and}\quad (S', j', M', D')$$
are called {\bf isomorphic}, if there exists a biholomorphic map
$$\varphi:(S, j)\to (S', j'), $$
(i.e., the map $\varphi$ satisfies $T\varphi\circ j=j'\circ T\varphi$), mapping marked points onto marked points (preserving the ordering) and nodal pairs onto nodal pairs. If the two nodal Riemann surfaces  are identical,  the isomorphism above is called an {\bf automorphism} of the nodal surface $(S, j, M, D)$. In the following we denote by
$$[S,j, M, D]$$
the equivalence class of all connected nodal surfaces isomorphic to the connected nodal Riemann surface $(S, j, M, D)$.

A crucial role play the so called stable nodal surfaces.

\begin{defn} The connected nodal Riemann surface $(S, j, M, D)$ is called {\bf stable}, if the group of its automorphisms is finite.
\end{defn}

One knows that a connected nodal Riemann surface $(S,j, M, D)$ is stable if and only if every connected domain component $C$ of $S$ satisfies
$$2g(C)+\sharp M_C \geq 3,$$
where $g(C)$ is the  genus of $C$ and where $M_C=C\cap (M\cup \abs{D})$ are the special points lying on $C$.

After having recalled  the concepts from Riemann surface theory we start with the analytical set up of our approach to the GW-invariants. We consider a symplectic manifold
$(Q,\omega)$ and assume, for simplicity, that
$\partial Q=\emptyset$. We study  maps  $u:S\to Q$ defined on Riemann
surfaces  $S$  into the symplectic manifold having special regularity properties introduced below.  With
$$
u:{\mathcal O}(S,z)\rightarrow Q
$$
we shall denote a germ around  the point $z\in S$ defined on a piece of  the Riemann
surface $S$.
\begin{defn}
Let $m\geq 2$ be an integer and $\varepsilon>0$.  The  germ of
a continuous map $u:{\mathcal O}(S,z)\rightarrow Q$ is  called of {\bf class
$(m,\varepsilon)$ at the point $z$} if for a smooth chart
$\varphi:U(u(0))\rightarrow {\mathbb R}^{2n}$ of $Q$ mapping $u(0)$ to $0$ and for
holomorphic polar coordinates $ \sigma:[0,\infty)\times
S^1\rightarrow S\setminus\{z\}$ around $z$,  the map
$$
v(s,t)=\varphi\circ u\circ \sigma(s,t),
$$
which is defined for $s$ large,  has partial derivatives up to order
$m$, which if weighted by $e^{\varepsilon s}$, belong to
$L^2([R,\infty)\times S^1,{\mathbb R}^{2n})$ if $R$ is sufficiently
large. The germ is of  called of {\bf class $m$ around the point $z\in S$}, if $u$ belongs to the class $H^m_{\text{loc}}$ near $z$.
\end{defn}
The above definition does not depend on the choices involved, like
charts on $Q$ and holomorphic polar coordinates on $S$.

We next consider multiplets
$$\alpha:=(S, j, M, D, u)$$
where $(S, j, M, D)$ is a connected nodal Riemann surface, and where
$$u:S\to Q$$
is a continuous map.
\begin{defn}
The multiplet $\alpha=(S, j, M, D, u)$ is a {\bf stable map of class $(m,\delta )$} for a given integer $m\geq 2$ and a given real number $\delta>0$, if it satisfies the following properties.
\begin{itemize}
\item[(1)] The map $u$ is of class $(m, \delta)$ around the points in $D$ and of class $H^{m}_{\text{loc}}$ around all the other points of $S$.\
\item[(2)] $u(x)=u(y)$ for every nodal pair $\{x,y\}\in D$.
\item[(3)] If  a connected component $C$ of $S$ has genus  $g_C$ and $M_C$ special points, and satisfies $2\cdot g_C +M_C\leq 2$,   then
$$\int_C u^\ast\omega >0.$$
\end{itemize}
Two stable maps $\alpha=(S,j,M,D,u)$ and $\alpha'=(S',j',M',D',u')$ are called
 {\bf  equivalent}  if there exists an isomorphism
 $\varphi:(S,j,M,D)\rightarrow (S',j',M',D') $ between the connected nodal Riemann surfaces  satisfying
 $$u'\circ\varphi=u.$$
 An equivalence class  $[\alpha]=[(S, j, M, D, u)]$ is  called a {\bf stable curve of class
$(m,\delta)$.}
\end{defn}
The following space $Z$ will be equipped with
a  polyfold structure.
\begin{defn}\label{defspace}
Fix a $\delta_0\in (0,2\pi)$. The collection of all equivalence
classes $[\alpha]=[(S, j, M, D, u)]$ of class $(3,\delta_0)$ is called the
space of stable curves into $Q$ of class $(3,\delta_0)$,  and is denoted
by $Z$.
\end{defn}
In \cite{HWZ4} the following  result is proved for the set $Z$.
\begin{thm}\label{th-top} With $Z$ as defined above we have:
\begin{itemize}
\item[(1)] The space $Z$ has a natural second countable
paracompact topology.
\item[(2)] Given a strictly increasing sequence $(\delta_m)$,
starting at the previously chosen $\delta_0$ and staying below
$2\pi$, and given the special gluing profile $\varphi:(0, 1]\to (0, \infty )$ defined by $\varphi(r)=e^{\frac{1}{r}}-e$, the
space $Z$ has a natural polyfold structure where the $m$-th level
consists of equivalence classes of multiplets   $(S,j,M,D,u)$ in which $u$
is of class $(m+3,\delta_m)$.
\end{itemize}
\end{thm}
Formulated in the  technical terms introduced below, given the sequence $(\delta_i)$ and the gluing profile $\varphi$, there exists a natural collection of pairs $(X,\beta)$ in which  $X$ is an ep-groupoid and $\beta:\abs{X}\rightarrow
Z$ is a homeomorphism from the orbit space $\abs{X}$ of the ep-groupoid onto the topological space $Z$, so that for any two pairs $(X,\beta)$ and
$(X',\beta')$ there exists a third pair $(X'',\beta'')$ and
equivalences $F:(X'',\beta'')\rightarrow (X,\beta)$ and
$F':(X'',\beta'')\rightarrow (X',\beta')$ satisfying
$$
\beta'' =\beta\circ |F|\ \hbox{and}\ \ \beta'' = \beta'\circ |F'|.
$$
These pairs and their sc-smooth compatibility define the smooth
structure on the topological space $Z$. The construction of the natural polyfold structure
is  carried out \cite{HWZ4}.

 There are natural maps which play an
important role in the  GW-theory. Let us first note that $Z$ has many
connected components. If $g,k\geq 0$ are integers we denote by
$Z_{g,k}$ the subset of $Z$ consisting of of all classes $[\alpha]=[(S, j, M, D, u)]$
of arithmetic genus $g$ and  $k$ marked points. This subset is open
in $Z$ and therefore is equipped with the  induced polyfold structure. If $A\in
H_2(Q,{\mathbb Z})$ is a homology class, we can also consider the open subset $Z_{A,g,k}$  consisting
of  elements in $Z_{g,k}$ for which  the map $u$ represents $A$.
Now we consider for the  fixed pair $(g,k)$ the space $Z_{g,k}$. For every $i=1,\ldots ,k$, we
define  the evaluation map at the $i$-th marked point by
$$
ev_i:Z_{g,k}\rightarrow Q:[\alpha]=[S,j,M,D,u]\rightarrow u(m_i).
$$
If $2g+k\geq 3$, the forgetful map associates with
$[\alpha]$ the underlying stable part of the domain $S$. It is
obtained as follows.

We take a representative $(S,j,M,D,u)$ of our
class $[\alpha]$ and forget  first the map $u$.  Now we  take a component $C$
satisfying $2g(C)+\sharp (C\cap (M\cup \abs{D}))<3$. Since,  by definition,  our  nodal surface $(S, j, M, D)$  is connected, the
following cases arise. Firstly,  $C$ is a sphere with one node. In this case we  remove
the sphere and the node. Secondly,  $C$ is a sphere with two  nodal pairs $\{x,y\}$ and $\{x',y'\}$, where $x$ and
$x'$ lie on the sphere. In this case we  remove the sphere but shortcut the two
nodes by removing the two nodal pairs but adding the nodal
pair $\{y,y'\}$. Thirdly,  $C$ is a sphere with one node and one
marked point. In that case we remove the sphere but replace the
corresponding node on the other component by the marked point.
Continuing this way we end up with a stable noded marked Riemann
surface whose biholomorphic type does not depend on the order we
weeded out the unstable components. The procedure leads to the forgetful map
$$
\sigma:Z_{g,k}\rightarrow \overline{\mathcal
M}_{g,k}:[S,j,M,D,u]\rightarrow [(S,j,M,D)_{\text{stab}}]
$$
where  $\overline{\mathcal M}_{g,k}$ is the standard Deligne-Mumford
compactification  of the moduli space of marked stable Riemann surfaces.
\begin{figure}[htbp]
\mbox{}\\[2ex]
\centerline{\relabelbox \epsfxsize 3.8truein \epsfbox{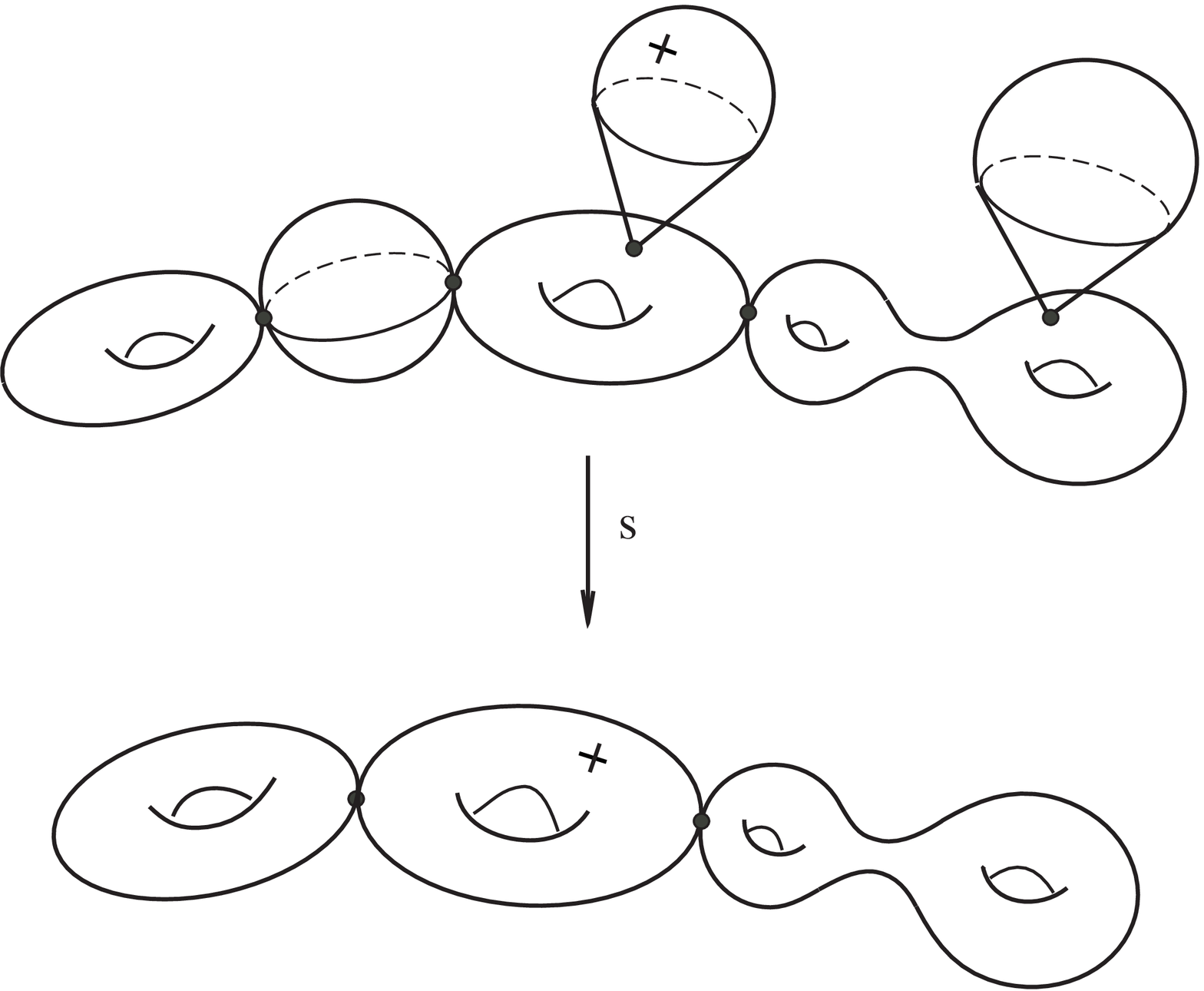}
\relabel {s}{$\sigma$}
\endrelabelbox}
\caption{The forgetful map $\sigma:Z_{g,k}\to \ov{M}_{g,k}$}
\label{Fig3.5.3}
\end{figure}
\begin{thm}
The evaluation maps $ev_i:Z_{g,k}\rightarrow Q$ and  the forgetful map $\sigma:Z_{g,k}\rightarrow
\overline{\mathcal M}_{g,k}$ are sc-smooth.
\end{thm}
As a consequence we can pull-back the differential forms on $Q$ and  on $
\overline{\mathcal M}_{g,k}$ to obtain sc-differential forms on the  polyfold
$Z_{g,k}$ which can be, suitably wedged together, integrated over
the smooth moduli spaces obtained from transversal Fredholm sections
of  strong bundles over $Z$, using the branched integration
theory from  \cite{HWZ7}.

Next we introduce a strong polyfold bundle $W$ over $Z$. The objects
of $W$ are defined as follows. We consider multiplets
$$\what{\alpha}=(S,j,M,D,u,\xi)$$
in which  the underlying stable map
$\alpha=(S,j,M,D,u)$ is a representative of an element in $Z$.
Moreover, $\xi$ is a continuous section along $u$, where for $z\in
S$, the mapping
 $$
 \xi(z):T_zS\rightarrow T_{u(z)}Q
 $$
 is complex anti-linear. The domain $S$ is equipped with  the complex structure $j$ and
 the target $Q$ is equipped with the almost complex structure  $J$. Moreover, on $S\setminus |D|$, the map
 $z\mapsto \xi(z)$ is of class $H^2_{loc}$. At the points in
 $\abs{D}$ we require that $\xi$ is of class $(2,\delta_0)$. This means,
 taking holomorphic polar coordinates $\sigma$ around a point $z\in \abs{D}$ and
 a chart around its image $u(z)\in Q$,  that  the map
 $$
 (s,t)\mapsto  pr_2\circ
 T\psi(u(\sigma(s,t)))\xi(\sigma(s,t))(\frac{\partial\sigma}{\partial
 s}(s,t))
 $$
and all its partial derivatives up to order $2$, if  weighted
 by $e^{\delta_0 \abs{s}}$,  belong to the space $L^2([s_0,\infty)\times S^1,{\mathbb
 R}^{2n})$ for $s_0$ large enough. The definition does not depend on the
 choices involved.

 We call two such tuples $(S, j, M, D, u, \xi)$ and $(S', j', M', D', u',\xi')$
 {\bf equivalent}  if  there exists an isomorphism
 $$
 \varphi:(S,j,M,D)\rightarrow (S',j',M',D')
 $$
 between the  nodal Riemann surfaces satisfying
$$u'\circ\phi=u\quad \text{and}\quad \xi'\circ\phi=\xi.$$
We denote an
{\bf equivalence class}  by $[\what{\alpha}]=[S,j, M, D,u,\xi]$. The collection of all such
equivalence classes constitutes  $W$.

We have defined what it means that an
element $\alpha$ represents an element on level $m$. Let us observe
that if $u$ has regularity $(m+3,\delta_m)$ it makes sense to talk
about elements $\xi$ along $u$ of regularity $(k+2,\delta_k)$ for
$0\leq k\leq m+1$. In the case $k=m+1$ the fiber
regularity is $(m+3,\delta_{m+1})$ and the underlying base
regularity is $(m+3,\delta_m)$.

The requirement of an
exponential decay in the fiber towards a nodal point  which is faster than the
exponential decay of the underlying base curve is well-defined and
independent of the charts picked to define it. Our conventions for defining the levels are governed by the overall
convention that sections should be horizontal in the sense that they
preserve the level structure, i.e. an element on level $m$ is mapped
by the section to an element on bi-level $(m,m)$. Hence if the
section comes from a first order differential operator we need
the  convention  we have just used. Therefore it makes
sense to say that an element
$$
\hat{\alpha}=(S,j,M,D,u,\xi)
$$
has (bi)-regularity $((m+3,\delta_m),(k+2,\delta_k))$ as long as $k$
satisfies the above restriction $0\leq k\leq m+1$. The equivalence class $[\hat{\alpha}]\in
W$ is said to be on level $(m,k)$ provided the pair $(u,\xi)$ has the above regularity.
One of the main consequences of \cite{HWZ4} is the following
theorem.
\begin{thm}\label{natural} With $Z$ as  just defined the following holds:
\begin{itemize}
\item[(1)] The set $W$ has a natural second countable paracompact
topology so that the natural projection  map
$$
\tau:W\rightarrow Z, \quad [\what{\alpha}]\mapsto  [\alpha]
$$
(forgetting the $\xi$-part) is continuous.
\item[(2)]  If $Z$ is  the previously introduced space with its
polyfold structure, then  the bundle $\tau:W\rightarrow Z$ has
the structure of a strong polyfold bundle  in a natural way.
\end{itemize}
\end{thm}

Finally,  we can introduce , for a compatible smooth almost complex
structure $J$ on  the symplectic manifold $(Q,\omega)$,  the section $\bar{\partial}_J$
of the strong polyfold bundle $W\rightarrow Z$ defined by
$$
\bar{\partial}_J([S,j,M,D,u])=[S,j,M,D,u,\bar{\partial}_{J,j}(u)]
$$
where $\bar{\partial }_{J, j}$ is the Cauchy-Riemann operator defined by
$$\bar{\partial }_{J, j} (u)=\dfrac{1}{2} ( Tu+J\circ Tu\circ j ).$$
Let us call a Fredholm section of a strong polyfold bundle
component-proper if  the restriction to every connected
component of the domain is proper. Then  the following
crucial  result which is a special  case of results proved in \cite{HWZ8}
and \cite{HWZ9} holds true.
\begin{thm}\label{compproper}
The section $\bar{\partial}_{J}$ is an sc-smooth component-proper
Fredholm section of the strong polyfold bundle $\tau:W\rightarrow
Z$ having a natural orientation.
\end{thm}

Applying Theorem \ref{invar} one derives  the following invariants. By
$(\mathfrak{M}, w)=(S(\bar{\partial}_J, \lambda ), \lambda_{\bar{\partial}_J})$ we abbreviate the oriented  and weighted solution set of the pair $(\bar{\partial}_J, \lambda )$, where $\lambda$ is an $\ssc^+$-multisection  and where  the solution set is given by
$(S(\bar{\partial}_J, \lambda )=\{z\in Z_{g,k}\vert \, \lambda (\bar{\partial}_Jz )>0\}$ and the weighting function $w$ is defined as
$w(z)=\lambda (\bar{\partial}_J z)$.  The set $\mathfrak{M}$ is an oriented  compact  branched suborbifold of $Z_{g,k}$ provided $\lambda$ is generic, i.e. $(\bar{\partial}_J,\lambda)$ is a transversal pair.

\begin{thm}\label{GW-thm}
Let $(Q,\omega)$ be a closed symplectic manifold. There exists for a
given  homology class $A\in H_2(Q)$ and for given  natural numbers $g,k\geq 0$ a multi-linear
map
$$
\Psi^Q_{A,g,k}: H^\ast(Q;{\mathbb R})^{\otimes k}\otimes
H_\ast(\overline{\mathcal M}_{g,k};{\mathbb R})\rightarrow {\mathbb
R}
$$
which on $H^*(Q; \R)^{\otimes k}$  is super-symmetric with respect to the grading by even and odd forms. This map
is uniquely characterized by the following formula. For a given compatible almost complex structure $J$ on $Q$  and a given
small generic perturbation by an $\ssc^+$-multisection $\lambda$ we have the presentation $$
\Psi^Q_{A,g,k}([\alpha_1],\ldots ,[\alpha_k];[\tau])=
\int_{(\mathfrak{M},w)} \hbox{ev}_1^\ast(\alpha_1)\wedge\ldots \wedge
\mbox{ev}_k^\ast(\alpha_k)\wedge \sigma^\ast(PD(\tau)).
$$
Here $PD$ denotes the Poincar\'e dual.
\end{thm}

The a priori real number
$\Psi_{A,g,k}^Q([\alpha_1],..,[\alpha_k];[\tau])$ is called a
GW-invariant. It is zero  unless the Fredholm index and the degree
of the differential form which is being integrated are the same.
One can  show that the numbers
$\Psi^Q_{A,g,k}([\alpha_1],\ldots ,[\alpha_k];[\tau])$ are  rational
if the (co)homology classes are integer. The defining integral
can be interpreted as a rational count of solutions of some non
linear problem. The integration theory used above is  the ``branched
integration'' introduced in  \cite{HWZ7}.

The various maps $\Psi^Q_{A,g,k}$ are interrelated by the so-called
composition law, which also forms the basis for the Witten-Dijkgraf-Verlinde-Verlinde-equation for which we refer to  \cite{Man,MS2,Tian} and we would like to mention that the  composition law follows readily
from  our  transversality theory as will be  shown in  \cite{HWZ4}.

 \section{Ep-Groupoids and Generalized Maps}\label{polyfold-groupoid-sec}

In this section we shall  introduce the concept  of an ep-groupoid. Ep-groupoids will serve later on as models for polyfolds which is the central topic of this paper.
Ep-groupoids are generalizations of proper \'etale Lie groupoids in which
 the local models for the object and morphism sets  are
M-polyfolds instead of finite-dimensional manifolds.
Different ep-groupoids can be models  for the same polyfold. This leads to the study of
equivalences between ep-groupoids. We shall  also  describe a construction for  inverting equivalences  and introduce the
concept of a generalized map. Our presentation of ep-groupoids follows the presentation of Lie groupoids in  \cite{Mj} and \cite{MM}.  The ideas go back to A. Haefliger and we refer to
\cite{Haefliger1, Haefliger1.5, Haefliger2, Haefliger3}.

\subsection{Ep-Groupoids}
We begin by recalling the notion  of a groupoid.
\begin{defn}
A {\bf groupoid} $\mathfrak{G}$ is a small  category whose
morphisms are all invertible.
\end{defn}
Recall that  the category $\mathfrak{G}$ consists of the set  of objects $G$, the set
${\bf G}$ of morphisms  (or arrows) and the   five  structure maps $(s, t, i, 1, m)$. Namely, the source and the target maps $s,t:{\bf G}\rightarrow G$ assign to every morphism , denoted by $g:x\to y$,  its source $s(g)=x$ and
its target $t(y)=y$, respectively.  The associative multiplication (or composition) map
$$m: {\bf G}{{_s}\times_t}{\bf G}\to {\bf G}, \quad m(h, g)=h\circ g$$
is defined on the fibered product
$$ {\bf G}{{_s}\times_t}{\bf G}=\{(h, g)\in  {\bf G}\times {\bf G}\, \vert \, s(h)=t(g)\}.$$
For every object $x\in G$, there exists the unit morphism $1_x:x\mapsto  x$ in $\bf{G}$ which is a $2$-sided unit for the composition, that is, $g\circ 1_x=g$ and $1_x\circ h=h$ for all morphisms $g, h\in \bf{G}$ satisfying $s(g)=x=t(h)$. These unit morphisms together define the unit map $u:G\to \bf{G}$ by $u(x)=1_x$. Finally, for every morphism $g:x\mapsto  y$ in $\bf{G}$, there exists the inverse morphism $g^{-1}:y\mapsto  x$ which is a $2$-sided inverse for the composition, that is, $g\circ g^{-1}=1_y$ and $g^{-1}\circ g=1_x$. These inverses together define the inverse map $i:\bf{G}\to \bf{G}$ by $i (g)=g^{-1}.$
The {\bf orbit space} of a groupoid $\mathfrak{G}$,
$$\abs{\mathfrak{G}}=G/\sim,$$
is  the quotient of the set of objects $G$ by the equivalence relation $\sim$ defined by
$x\sim y$ if and only if  there exists a morphism $g:x\mapsto  y$.  The equivalence class $\{y\in G\vert \, y\sim x\}$ will be denoted by
$$\abs{x}=\{y\in G\vert \, y\sim x \}.$$
If $x, y\in G$ are two objects, then ${\bf G}(x, y)$ denotes the set of all morphisms $g:x\mapsto y$. In particular, for $x\in G$ fixed, we denote by ${\bf G}(x)={\bf G}(x, x)$ the {\bf stabilizer} (or {\bf isotropy}) group of $x$,
$${\bf G}(x)=\{\text{morphisms}\ g:x\mapsto x\}.$$
For the sake of notational economy
we shall denote in the following a groupoid, as well as its object set by the same letter $G$ and its morphism set  by the bold letter ${\bf G}$.

\newenvironment{Myitemize}{%
\renewcommand{\labelitemi}{$\bullet$}%
\begin{itemize}}{\end{itemize}}

A {\bf homomorphism} $F:G\to G'$ between two groupoids is, by definition, a functor. In particular, the two induced maps $F:G\to G$ and $F:{\mathbf  G}\to {\mathbf  G}'$ between the object sets and the morphism sets commute with the structure maps:
\begin{align*}
s'\circ F=F\circ s&\qquad  \quad t'\circ F=F\circ t\\
i'\circ F=F\circ i& \qquad  \quad u'\circ F=F\circ u
\end{align*}$$m'\circ (F\times F)=F\circ m.$$
Ep-groupoids, as defined next, can be viewed as M-polyfold versions of  \'etale and proper Lie-groupoids discussed e.g. in \cite{Mj} and \cite{MM}.

\begin{defn}
{\em An  {\bf ep-groupoid}  $X$ is a groupoid $X$ together with M-polyfold
structures on the object set $X$ as well as on the morphism set ${\bf X}$
so that all the structure maps  $(s, t, m, u, i)$ are sc-smooth maps and the following
holds true.
\begin{Myitemize}
\item {\em ({\bf \'etale})} The source and target maps
$s$ and $t$ are surjective local sc-diffeomorphisms.
\item {\em ({\bf proper})} For every point $x\in X$,    there exists an
open neighborhood $V(x)$ so that the map
$t:s^{-1}(\overline{V(x)})\rightarrow X$ is a proper mapping.
\end{Myitemize}
}
\end{defn}

We  point out that if  $X$  is a groupoid  equipped  with
M-polyfold structures on the object set $X$ as well as on the
morphism set ${\bf X},$ and  $X$  is \'etale,  then  the
fibered product ${\bf X}{{_s}\times_t}{\bf X}$ has a natural
M-polyfold structure so that the multiplication map $m$ is defined
on an  M-polyfold.  Hence its
sc-smoothness is well-defined. This is proved in Lemma \ref{ert} below.

In an ep-groupoid every morphism $g:x\to y$ can be extended to
a unique local diffeomorphism $t\circ s^{-1}$ satisfying $s(g)=x$
and $t(g)=y$. The properness assumption implies that  the isotropy groups
${\bf G}(x)$ are finite groups.

The local structure of the morphism set of an ep-groupoid in a neighborhood of an isotropy group is described in the following theorem whose proof can be found in Appendix 5.

\begin{thm}\label{localstructure}
Let $x$  be an object of an ep-groupoid $X$.  Then every open neighborhood $V\subset X$ of $x$ contains an open neighborhood $U\subset V$ of $x$,  a group homomorphism
$$
\varphi: {\bf G}(x)\to  \text{Diff}_{\ssc}(U), \quad g\mapsto
\varphi_g,
$$
of the isotropy group into the group of sc-diffeomorphims
of $U$,  and an sc-smooth map
$$
\Gamma: {\bf G}(x)\times U\to {\bf X}
$$
having the following properties.
\begin{Myitemize}
\item $\Gamma(g,x)=g$.
\item $s(\Gamma(g,y))=y$ and $t(\Gamma(g,y))=\varphi_g(y)$.
\item  If $h:y\rightarrow z$  is a morphism between two points in  $U$,  then there exists a
unique element $g\in  {\bf G}(x)$ satisfying  $\Gamma(g,y)=h$.
\end{Myitemize}
\end{thm}
The group homomorphism $\varphi: {\bf G}(x)\to \text{Diff}_{\ssc}(U)$ is
called a {\bf natural representation of the isotropy group $ {\bf G}(x)$} of
the element $x\in X$. The diffeomorphism $\varphi_g$ is given by
$t\circ s^{-1}$ where $s(g)=t(g)=x$. We see that every morphism
between points in $U$ belongs to the image of the map $\Gamma$ and
so has an extension to precisely one of the finitely many
diffeomorphisms $\varphi_g$ of $U$, where $g\in  {\bf G}(x)$.

 An M-polyfold $X$ is, in view of its definition in \cite{HWZ2},  equipped with a filtration
$$X=X_0\supset X_1\supset X_2\supset \cdots \supset X_{\infty}:=
\bigcap_{k\geq 0}X_k$$
into subsets such that the injection maps $X_{k+1}\to X_k$ are continuous and $X_{\infty}$ is dense in  all the sets $X_{k}$.  Points or  subsets contained in $X_{\infty}$ are called {\bf smooth points} or {\bf smooth subsets}. On the space $X_k$ there is the induced filtration
$X_k\supset X_{k+1}\supset \cdots \supset X_{\infty}$.  This M-polyfold structure of $X_k$ will be denoted by $X^{k}$.

 If $x\in X$, we denote by $\text{ml}(x)\in
\N_0\cup\{\infty\}$ the largest nonnegative  integer $m$ or $\infty$ so that
$x\in X_m$. We call $m(x)$ the  {\bf maximal level} of $x$.

In an ep-groupoid the object set $X$ as well as the morphism set ${\bf X}$ are M-polyfolds and hence are both equipped with filtrations.
Since the source and the target maps are local diffeomorphisms and therefore preserve by definition the  levels, we conclude that
$$\text{ml}(x)=\text{ml}(y)=\text{ml}(g)$$
if $g:x\to y$ is a morphism.  Consequently, the filtration of the object set $X$ induces the filtration
$$\abs{X}=\abs{X_0}\supset  \abs{X_1}\supset \cdots \supset \abs{X_{\infty}}=\bigcap_{k\geq 0}\abs{X_k}$$
of the orbit space $\abs{X}=X/\sim.$

\subsection{Functors and Equivalences}
In this section we introduce  the notion of an equivalence which  plays an important role in the construction of polyfolds later on.
\begin{defn}
A functor $F:X\rightarrow Y$ between two ep-groupoids is called
{\bf sc-smooth} provided the induced maps between object- and
morphism- spaces are sc-smooth.
\end{defn}
In view of the  functoriality of $F$,  the two induced maps commute with the structure maps and therefore an  sc-smooth functor $F:X\rightarrow Y$ induces an $\ssc^0$-map
$\abs{F}:\abs{X}\to \abs{Y}$ between the orbit  spaces. A continuous map is called $\ssc^0$ if it preserves the natural filtrations.   An important
class of sc-smooth functors are the equivalences defined next.
\begin{defn}
{\em
An sc-smooth functor
$F:X\rightarrow Y$ between  two ep-groupoids  is called an {\bf equivalence} provided it has
the following properties.
\begin{Myitemize}
\item $F$ is a local sc-diffeomorphism on objects as well as
morphisms.
\item The induced map $\abs{F}:\abs{X}\to \abs{Y}$ between the
orbit spaces is an $\ssc$-homeomorphism.
\item  For every $x\in X$,  the map $F$ induces a bijection
${\bf X}(x)\rightarrow {\bf Y}(F(x)) $ between the isotropy groups.
\end{Myitemize}
}
\end{defn}
If $F:X\rightarrow Y$ and $G:Y\rightarrow Z$ are
equivalences, then the composition $G\circ F:X\rightarrow Z$ is also an equivalence.

In general,  an equivalence is not invertible as a
functor.  Later on, we will discuss a general procedure of inverting
arrows in a category by only changing the morphism set but
keeping the object set. This ``inverting of arrows'' for a given
class of arrows is a standard procedure in category theory and we refer to
\cite{GZ}.
\begin{defn} \label{ntrans}
{\em
Two sc-smooth functors
$F:X\to Y$ and $G:X\to Y$ between the same ep-groupoids
are called {\bf naturally equivalent}, if there exists an sc-smooth
map $\tau:X\to {\bf Y}$ which associates with every object $x\in X$
an arrow $\tau (x):F(x)\to G(x)$ in ${\bf Y}$  and which is ``natural in
$x$'' in the sense that for every arrow $h:x\to x'$ in ${\bf X}$ the
identity
$$\tau (x')\circ F(h)=G(h)\circ \tau (x)$$
holds. The map $\tau$ is called a {\bf natural transformation}.
\mbox{}\\
\begin{equation*}
\begin{CD}
F(x)@>\tau (x)>>G(x)\\
@VVF(h)V @VVG(h)V \\
F(x')@>\tau (x')>>G(x').\\ \\
\end{CD}
\end{equation*}
 }
\end{defn}
Two naturally equivalent functors $F:X\to Y$ and $G:X\to Y$ induce the same
map between the orbit  spaces,
$$
\abs{F}=\abs{G}:\abs{X}\to \abs{Y}.
$$
From the definitions one deduces  easily the following result.
\begin{prop}
{\em Assume that the sc-smooth functors $F:X\to Y$ and $ G:X\to  Y$ between
ep-groupoids are naturally equivalent.  Then if one of the functors  is an
equivalence so is the other.}
\end{prop}
\begin{proof}
Assume that $F:X\to Y$ is an equivalence. As remarked above the functors  $F, G$ induce the same mappings between  the orbit spaces so that  $\abs{F}=\abs{G}:\abs{X}\to \abs{Y}$.  Hence $\abs{G}$ is a homeomorphism.  If  $x\in X$, we have to show that $G$ induces a bijection between the isotropy groups ${\bf X}(x)$ and ${\bf Y}(G(x))$. It suffices to show that given $g\in {\bf Y}(G(x))$, there exists a unique $h\in {\bf X}(x)$ such that $G(h)=g$.  Let $\tau:X\to {\bf Y}$ be a natural transformation and
define the  morphism $f:=\tau(x)^{-1}\circ g\circ  \tau (x):F(x)\to F(x)$. Since  there is a bijection between ${\bf X}(x)$ and ${\bf Y}(F(x))$, there exists a unique morphism $h:x\mapsto  x$ satisfying $F(h)=f$. From  $\tau (x)\circ F(h)=G(h)\circ \tau (x)$  one concludes that $g=\tau (x)\circ F(h)\circ \tau (x)^{-1}=G(h)$ which proves our claim.
Next we show that $G$ is a local sc-diffeomorphism on objects.
Fix $x_0\in X$.  Let ${\bf U}(\tau(x_0))$ be an open neighborhood  in ${\bf Y}$ of  the morphism $\tau (x_0):F(x_0)\to G(x_0)$ such that the source and the target maps
$s:{\bf U}(\tau (x_0))\to U(F(x_0))$ and $t:{\bf U}(\tau (x_0))\to U(G(x_0))$ are sc-diffeomorphisms. Since, by assumption, $F$ is a local sc-diffeomorphism on objects, there exists a an open neighborhood $U(x_0)$ of $x_0$ in $X$ such that $F:U(x_0)\to F(U(x_0))$ is an sc-diffeomorphism. We take $U(x_0)$  so  small that $\tau (U(x_0))\subset{\bf  U}(\tau (x_0))$. Moreover, we may assume that $s:U(\tau (x_0))\to F(U(x_0))$ is an sc-diffeomorphism.  Then  $G(x)=t\circ s^{-1}\circ F(x)$  for $x\in U(x_0)$ and  $G(U(x_0))=U(G(x_0))$. Because  $t, s^{-1}$ and $F$ are sc-diffeomorphisms also the composition $G$ is an sc- diffeomorphism on the neighborhood  $U(x_0)$.

It remains to show that $G:X\to Y$ is a local sc-diffeomorphism on the  morphisms set. To do this we fix a morphism $g:x_0\to y_0$ between the  objects $x_0$ and $y_0$ in $X$.  By assumption there exists an open neighborhood ${\bf U}(g)$ of $g$ in $ {\bf X}$ such that $F:{\bf U}(g)\to F({\bf U}(g))$ is an sc-diffeomorphism.  Since $F$ and $G$ are local sc-diffeomorphism, we find open neighborhoods $U(x_0)$ and $U(y_0)$ of the objects $x_0$ and $y_0$ such that the maps $F$ and $G$ are sc-diffeomorphism on $U(x_0)$ and $U(y_0)$. We may choose  the neighborhoods  $U(x_0)$, $U(y_0)$ and ${\bf U}(g)$ so that the source and the target maps $s:{\bf U}(g)\to U(x_0)$ and $t:{\bf U}(g)\to U(y_0)$ are sc-diffeomorphisms. Then the set $G({\bf U}(g))$ is an open subset of ${\bf Y}$ and the sc-smooth  map $G: {\bf U}(g)\to G({\bf U}(g))$  has an inverse given by
$G^{-1}(h)=F^{-1}(\tau (x)\circ h\circ (\tau (y))^{-1})$ for a morphism $h:G(x)\to G(y)$ belonging to $G({\bf U}(g))$.  Since this inverse is an sc-diffeomorphism,  we have proved that also $G$ is a local sc-diffeomorphism.

\end{proof}

Next we shall define the {\bf weak fibered product}  $\el =X\times_YZ$ in which  $X, Y$,  and $Z$ are ep-groupoids. Consider  two sc-smooth functors $F:X\to Y$ and $G:Z\to Y$ between ep-groupoids having  the same target ep-groupoid $Y$.
We first define the  fibered product $\el=X\times_Y Z$
as a groupoid  and then, under an additional assumption, as
an ep-groupoid.

The object set of $\el$ consists of triples $(x,
\varphi ,z)\in X\times {\bf Y}\times Z$ where $\varphi:F(x)\to G(z)$
is a morphism in ${\bf Y}$.  Hence
$$
\el=\{ (x, \varphi, z)\in X\times {\bf Y}\times Z\vert \,
\text{$s(\varphi )=F(x)$\, and \, $ t(\varphi )=G(z)$}\}.
$$
In short notation, $\el=X {{_F}\times_s} {\bf Y}
{{_t}\times_G} Z$. A morphism $l:(x, \varphi , z)\to (x', \varphi'
,z'), $ between two objects in $\el$ is a triple $l=(h,
\varphi, k)\in {\bf X}\times {\bf Y}\times {\bf Z}$
with  morphisms $h:x\to x'$ and $k:z\to z'$  satisfying
$$
\varphi' =G(k)\circ \varphi \circ F(h)^{-1}.
$$
Hence the set of morphisms is equal to ${\bf L}={\bf X}{{_{s\circ
F}}\times_s}{\bf Y}{{_t}\times_{s\circ G}}{\bf Z}$ or, explicitly,
 \begin{equation*}
{\bf L}=\{(h, \varphi ,k)\in {\bf X}\times {\bf Y}\times {\bf Z}\vert \,
s\circ F(h)=s(\varphi) \,\, \text{and}\,\, t(\varphi )=s\circ
G(k)\}
\end{equation*}
and the source and the target maps $s, t:{\bf L}\to \el$
 are defined by
\begin{equation}\label{newequation8.2}
\begin{aligned}
s(h, \varphi ,k)&=(s(h), \varphi, s(k))\ \hbox{and}\ \ t(h, \varphi
,k)&=(t(h), \varphi', t(k))
\end{aligned}
\end{equation}
where $\varphi':=G(k)\circ \varphi\circ F(h)^{-1}$. The
multiplication  of two morphisms   is  defined as
\begin{equation*}
(h,\varphi ,k)\circ (h', \psi, k')=(h\circ h', \psi, k\circ k').
\end{equation*}
The identity  morphisms $1_{(x, \varphi, z)}\in {\bf L}$ at the
object $(x, \varphi ,z)\in L$ is  the triple $1_{(x, \varphi,
z)}=(1_x, \varphi , 1_z)$. The inversion map is defined by
$i(h,\varphi,k)=(h^{-1},G(k)\circ\varphi \circ F(h)^{-1}, k^{-1})$.
With the above definitions the fibered product  $L=X\times_YZ$ becomes  a groupoid called the {\bf weak fibered product}.

In the next step we  describe  conditions which guarantee that $\el$  is an ep-groupoid. We need a lemma which is a special case of a
more general result in \cite{HWZ-polyfolds1}.
\begin{lem}\label{ert}
Let $X$, $Y$ and $Z$ be M-polyfolds. Assume that $f:X\rightarrow Y$ is
a local sc-diffeomorphism and $g:Z\rightarrow Y$ is an $\ssc$-smooth
map. Then the fibered product
$$X{{_f}\times_g}Z=\{(x, z)\in X\times  Z\vert f(x)=g(z)\}$$
 has a natural
M-polyfold structure and the projection map
 $\pi_2:X{{_f}\times_g}Z\rightarrow Z$ is a  local
$\ssc$-diffeomorphisms.
\end{lem}
\begin{proof} Let $\triangle=\{(y, y)\vert \, y\in Y\}$ be the diagonal in the space $Y\oplus Y$. Then
$X{{_f}\times_g}Z=(f, g)^{-1}(\triangle )\subset X\oplus Z$.   Hence $X{{_f}\times_g}Z$ is a closed subset  of
 $X\oplus Z$ and consequently carries a  paracompact
second countable topology. Fix a point $(x,z)\in X{{_f}\times_g}Z$. Since $f:X\to Y$ is a local sc-diffemorphism, we find two open neighborhoods $U(x)\subset X$ and $V(g(z))\subset Z$ of the points $x$ and $g(z)$ such that $f:U(x)\to V(g(z))$ is an sc-diffeomorphism.  Next we find an  open neighborhood $W(z)\subset Z$ of $z$ so that
there exists a chart $\varphi:W\rightarrow O$ onto an open subset $O$
of a splicing core $K$ and,  in addition,   $g(W(z))\subset V(g(z))$. Then
 $N:= (U (x)\oplus W (z))\cap X{{_f}\times_g}Z$ is an open neighborhood of $(x,z) $ in $X{{_f}\times_g}Z$ and the map
$$
\Phi:N\rightarrow O, \quad  (x,z)\mapsto  \varphi (z)
$$
is a bijection and of class  $\ssc^0$. The inverse map  $\Phi^{-1}:O\to N$ is equal to
$$
\Phi^{-1}(k)=(f^{-1}\circ g \circ \varphi^{-1}(k),\varphi^{-1}(k)))
$$
and is an $\ssc^0$-map. If $\Phi_1:N_1\to O_1$ is defined by $(x, z)\mapsto \varphi_1(z)$, where $\varphi_1:W_1\to O_1\subset K_1$, is another such chart  map,  then $\Phi_1\circ \Phi^{-1}: \Phi (N\cap N_1)\to \Phi_1(N\cap N_1)$ is  equal to  $\Phi_1\circ \Phi^{-1}(k)=\varphi_1\circ \varphi^{-1}(k)$ for $k\in  \Phi (N\cap N_1)$. Hence $\Phi_1\circ \Phi^{-1}$ is sc-smooth. Consequently, the above charts $(\Phi, N)$ define an M-polyfold structure on $X{{_f}\times_g}Z$. The construction also shows that  the projection $\pi_2$  is a local sc--diffeomorphism.
\end{proof}

We shall use the previous lemma in the construction of the weak fibered product of ep-groupoids.

\begin{thm}\label{major}
Let $X$, $Y$ and $Z$ be ep-groupoids. Assume that  the functor $F:X\rightarrow Y$
is an equivalence and $G:Z\rightarrow Y$ an $\ssc$-smooth functor.
Then the  fibered product $X\times_Y Z$ has in a natural way the
structure of an ep-groupoid. Moreover, the projection functor $p:X\times_Y
Z\rightarrow Z$ is an equivalence.
\end{thm}
\begin{proof}
In order to apply Lemma \ref{ert}, we view  the object set of the fibered product $X\times_Y Z$ as
\begin{equation}\label{job-x}
X{{_F}\times_{s\circ\pi_1}}\left[{\bf Y}{{_t}\times_G}Z\right].
\end{equation}
Since the target map $t$ is a local sc-diffeomorphism and $G:Z\to Y$ is sc-smooth, Lemma \ref{ert} implies that ${\bf
Y}{{_t}\times_G}Z$ is in a natural way an  M-polyfold. The map $s\circ\pi_1:{\bf Y}{{_t}\times_G}Z\rightarrow Y$
is sc-smooth and $F:X\rightarrow Y$ is a local sc-diffeomorphism. Applying
 Lemma \ref{ert}  again, we conclude that the set  of  objects $X\times_Y Z$ has a
natural M-polyfold structure.  We could have also used instead of
(\ref{job-x}) the different bracketing
\begin{equation}\label{second}
\left[X{{_F}\times_s}{\bf Y}\right]{_{t\circ\pi_2}\times_G}Z,
\end{equation}
 which would lead to the same M-polyfold structure.

 Using the fact that $s\circ F$, $s$, and $t$ are local
sc-diffeomorphisms one also  shows that   the set of morphisms
${\bf L}={\bf X}{{_{s\circ F}}\times_s}{\bf Y}{{_t}\times_{s\circ G}}{\bf Z}$  has a
natural M-polyfold structure. With these  M-polyfold structures on $L$ and ${\bf L}$, the source map
$$
s:{\bf L}\rightarrow \el :s(h,\varphi,k)=(s(h),\varphi,s(k))
$$
is sc-smooth and we show that $s$ is a local diffeomorphism.  Fix a morphism  $(h,\varphi,k)\in {\bf L}$. Since the target map $s:{\bf Z}\to Z$ is a local  sc-diffeomorphism, we find open neighborhoods ${\bf V}(k)\subset {\bf Z}$ of the morphism $k$ and an open neighborhood $V(s(k))\subset Z$ of the point $s(k)$ such that $s:{\bf V}(k)\to V(s(k))$ is an  sc-diffeomorphism. Using  Lemma \ref{ert} and  shrinking these neighborhoods if necessary,  we find an open  neighborhood  ${\bf U}$ of the morphism $(h,\varphi,k)$ in ${\bf L}$ such that the projection ${\bf p}: {\bf U}\subset {\bf L}\to {\bf V}(k)\subset {\bf Z}$ is an sc-diffeomorphism.   Using the second bracketing (\ref{second})  and Lemma \ref{ert} again, we find an open neighborhood $W\subset \el $ of  the object $(s(h), \varphi, s(k))$ such that the projection $p:L\to Z$ is an sc-diffeomorphism from $W$ onto $V(s(k))$. Now observe that
$$
s(h',\varphi',k') = {p}^{-1}\circ s\circ {\bf p}(h',\varphi',k')
$$
for every $(h', \varphi', k')\in {\bf U}$.  Since  the right-hand side is a composition of local
sc-diffeomorphisms, the source map $s:{\bf L}\to \el$ on the left-hand side is a local sc-diffeomorphism.
Also the  inversion map $i:{\bf L}\to {\bf L}$ is an
sc-diffeomorphism  and since the target map  is the composition $t=s\circ i$ of the source map with the inverse, we conclude that  the target map $t:{\bf L}\to \el$ is a local  sc--diffeomorphism as well. Consequently,  the
multiplication map
$$
m:{\bf L}{{_s}\times_t}{\bf L}\rightarrow {\bf
L}, \quad m((h,\varphi,k),(h',\varphi',k'))=(h\circ h',\varphi',k\circ k')
$$
is well-defined and  clearly  sc-smooth. Next  we show that $\el$ is
proper. Pick a point $a=(x,\varphi ,z)\in \el$.  By the properness of the groupoids $X$ and $Z$,  there are open
neighborhoods $U(x)$ and $V(z)$ so that the maps
$$
t:s^{-1}(\overline{U(x)})\rightarrow X\ \text{and}\quad
t:s^{-1}(\overline{V(z)})\rightarrow Z
$$
are proper.  Define  the open $W(a)\subset \el$  by
$$
W(a)=\{\text{$(x',\varphi', z')\in L\vert$ \, $x'\in U(x),$\  $z'\in V(z)$ and $\varphi':F(x')\to G(z')$}\}.
$$
We claim  that $t:s^{-1}(\overline{W(a)})\rightarrow \el $ is
proper. To see this we  take   a sequence $(h_j,\varphi_j,k_j)\in{\bf
L}$ such $s(h_j,\varphi_j,k_j)\in \ov{W(a)}$. We may assume that after taking a subsequence $t(h_j,\varphi_j,k_j)\to (x', \varphi'  z')=:b$. By the definition of the target map $t:{\bf L}\to \el$,  $t(h_j, \varphi_j, k_j)=(t(h_j), \varphi'_j, t(k_j))$ where
$\varphi_j'=G(k_j)\circ \varphi_j\circ F(h_j)^{-1}$. Hence the convergence of $t(h_j, \varphi_j, k_j)$ to $(x', \varphi', z')$ implies that $t(h_j)\to x'$, $t(k_j)\to z'$,  and $\varphi_j'=G(k_j)\circ \varphi_j\circ F(h_j)^{-1}\to \varphi'$.
Since $s(h_j)\in \ov{U(x)}$, the properness of the map $t:s^{-1}(\overline{U(x)})\rightarrow X$, implies that after taking a subsequence, $h_j\to h$.  The same argument shows that after taking a further subsequence we have  $k_j\to k$.
In particular, we conclude that $F(h_j)\to F(h)$ and $G(k_j)\to G(k)$ and since the inversion $i$ is sc-smooth,  we also have the convergence  $G(k_j)^{-1}\to G(k)^{-1}$.
From  $\varphi_j=G(k_j)^{-1}\circ \varphi_j'\circ F(h_j)$ and $\varphi_j'\to \varphi'$, we deduce that $\varphi_j\to
G(k)^{-1}\circ \varphi'\circ F(h)$. Hence the sequence   $(h_j,\varphi_j,k_j)$ has a convergent subsequence which proves that $\el$ is proper.

Next we shall show that the projection functor
\begin{equation*}
p:X\times_Y Z\rightarrow Z
\end{equation*}
is an equivalence. We already know from Lemma \ref{ert} that the projection $p$ is an  sc-smooth functor which is
a local $\ssc$-diffeomorphism on objects and morphisms.  To see  that $p$  induces a bijection ${\bf L}(a)\rightarrow {\bf
Z}(z)$ between the  isotropy groups ${\bf L}(a)$ and ${\bf Z}(z)$, we take $a=(x,\varphi,z)\in L$ and $z=p(a)$.  The isotropy group of $a$ is  equal to
$$
{\bf L}(a)=\{(h,\varphi,k)\ |\ h:x\rightarrow x,\ k:z\rightarrow z,\
\varphi\circ F(h)=G(k)\circ\varphi\}.
$$
Given $k\in {\bf Z}(z)$, the morphism $\varphi^{-1}\circ G(k)\circ \varphi$ belongs to the isotropy group ${\bf Y}(F(x))$.
Since $F$ is an equivalence, there is unique morphism $h$ belonging to the isotropy group ${\bf X}(x)$ such that
$F(h)=\varphi^{-1}\circ G(k)\circ \varphi$.  Hence the map ${\bf L}(a)\to {\bf Z}(z)$ defined by $(h, \varphi, k)\mapsto k$ is a bijection.

It remains to prove  that $\abs{p}:\abs{\el }\rightarrow \abs{Z}$ is an
sc-homeomorphism. As  $p$ is an sc-smooth functor,   the induced map $\abs{p}$  is of class
$\ssc^0$.  If $|p(x,\varphi,z)|=|p(x',\varphi',z')|$, then  $\varphi:F(x)\to G(z)$ and  $\varphi:F(x')\to G(z')$. Moreover,   there exists a
morphism $k:z\rightarrow z'$. Then $(\varphi')^{-1}\circ G(k)\circ \varphi$ is a morphism between $F(x)$ and $F(x')$. Because  $\abs{F}:\abs{X}\to \abs{Y}$ is  a bijection, there exists a unique morphism
$h:x\to x'$ such that $F(h)= (\varphi')^{-1}\circ G(k)\circ \varphi$. This implies that $(h, \varphi, k)\in {\bf L}$  is a morphism between $(x, \varphi, z)$ and $(x', \varphi', z')$ showing that both of these triples belong to the same equivalence class in the orbit space $\abs{L}$. So, $\abs{p}$ is an injection.  If $|z|\in \abs{Z}$, then $\abs{F}|x|=\abs{G}|z|$ for some $|x|\in \abs{X}$, that is, $|G(z)|=|F(x)|$.  This means that  there is a morphism $\varphi:F(x)\to G(z)$ so that  the morphism $\varphi^{-1}\circ 1_{G(z)}\circ \varphi$ belongs to the isotropy group ${\bf Y}(F(x))$. Hence there exists an element $h\in {\bf X}(x)$ so that $F(h)=\varphi^{-1}\circ 1_{G(z)}\circ \varphi$. This implies that $(x, h, y)\in \el$ and  $\abs{p}|(x, h, z)|=|z|$ proving that $\abs{p}$ is  a surjection.
Consider $|a|\in \abs{\el}$ with the representative $a=(x, \varphi, z)\in L$.
Since $p:\el \to Z$ is a local sc-diffeomorphism, there exists an open neighborhood $U(a)$ of the point $a$ in $\el$ and an open neighborhood $V(z)$ of the point $z$ in $Z$ so that $p: U(a)\to V(z)$ is an sc-diffeomorphism.  The quotient maps $\pi_1:\el\to \abs{\el}$ and $\pi_2:Z\to \abs{Z}$ are open.  Hence $\abs{U(a)}$ and $\abs{V(z)}$ are open neighborhoods of the equivalence class $[a]$ and $[z]$ in $\abs{\el}$ and $\abs{Z}$, respectively. From
$\abs{p}(\abs{U(a)})=\pi_2\circ p\circ \pi_1 (\abs{U(a)})=\abs{V(z)}$ it follows that $\abs{p}$ maps open sets in $\abs{\el}$ onto open sets in $\abs{Z}$. Therefore, $\abs{p}^{-1}=\abs{p^{-1}}:\abs{Z}\to \abs{\el}$ is continuous. Because  also $\abs{p}:\abs{\el}\to \abs{Z}$ is continuos and  a bijection, the map  $\abs{p}$ is a homeomorphism.  This  finishes the proof of Theorem \ref{major}.
\end{proof}
The next result is  important for our constructions later on.
\begin{prop}
{\em Assume that the functors $F:X\rightarrow Y$ and $G:Z\rightarrow Y$
between ep-groupoids are equivalences. Then there exists a third
ep-groupoid $\el $ and equivalences
$$
\Phi:\el \rightarrow X\ \ \hbox{and}\ \  \Psi:\el \rightarrow Z
$$
so that the compositions $F\circ\Phi$ and $G\circ \Psi:L\to Y$ are
naturally equivalent.}
\end{prop}
\begin{proof}
Set  $\el =X\times_Y Z$.  Since $F$ and $G$
are equivalences and therefore local sc-diffeomorphisms,  Theorem \ref{major} implies  that
the projections
$$
\pi_1:\el \rightarrow X\ \ \hbox{and}\ \  \pi_2:\el \rightarrow Z
$$
are equivalences. A natural
transformation  between $F\circ\pi_1$ and $G\circ\pi_2$ is  given  by the sc-smooth map
$\tau$ defined  by $ \tau(x,\varphi,z)=\varphi.
$
\end{proof}

\subsection{Inversion of Equivalences and Generalized Maps}\label{sect2.3}
As we shall see later on the ep-groupoids can be viewed  as
models for  polyfolds.  If two ep-groupoids have an equivalence
between them they will turn out to be  models for the same polyfold.
In this subsection we introduce the notion of a generalized map
between ep-groupoids. Later on generalized maps will descend to maps between
polyfolds.

So far we have constructed a category whose  objects are  the polyfold
groupoids and whose morphisms between them are the sc-smooth
functors.  There is a distinguished family of morphisms, namely the
equivalences. Equivalences are usually not invertible (as functors).
However, there is a category-theoretic procedure for inverting a
class of prescribed arrows in a given category while at the same
time keeping the objects and only minimally changing the morphisms.
The general procedure is described in \cite{GZ}. In \cite{Mj},
Moerdijk describes the procedure in the case of Lie groupoids, which
are used to give a definition of an orbifold. Modulo the
modifications necessitated by the fact that we work in the splicing
world we follow Moerdijk's description.
  We define a new category whose objects are
the ep-groupoids and whose morphisms ``$X\Longrightarrow Y$'' are
equivalence classes of diagrams of the form
$$X\xleftarrow{F}A\xrightarrow{\Phi} Y$$
where $X, Y$, and $A$  are ep-groupoids and where $F$ is an equivalence and
$\Phi$  is an sc-smooth functor. A second such diagram
$$X\xleftarrow{F'}A'\xrightarrow{\Phi'} Y$$
is called a {\bf  refinement } of the first diagram, if there exists
an equivalence $H:A'\to A$ so that the functors  $F\circ H$ and $F':A'\to X$ are naturally
equivalent as well as $\Phi\circ H$ and $\Phi':A'\to Y$ as illustrated in  the diagram
below. It is clear that given three diagrams which connect $X$ with
$Y$  so that $d'$ refines $d$ and $d''$ refines $d'$, then $d''$
refines $d$.
\begin{equation*}
\begin{CD}
X@<F<<A@>\Phi>>Y\\
@. @AAHA @. \\
X@<F'<<A'@>\Phi'>>Y.\\ \\
\end{CD}
\end{equation*}
\begin{defn}
{\em Two diagrams $X\xleftarrow{F}A\xrightarrow{\Phi}Y$
and $X\xleftarrow{F'}A'\xrightarrow{\Phi'}Y$ as above are called
{\bf equivalent} \index{equivalent diagrams}  if they have a common
refinement.}
\end{defn}
The notion of having a common refinement is clearly reflexive and
symmetric on diagrams of the above form. Let us show that it is also
transitive, so that it indeed  defines an equivalence relation.
\begin{prop}
{\em  Assume that $d, d'$ and $d''$ are three diagrams connecting the ep-groupoids  $X$ with $Y$ and assume that
the diagrams $(d, d')$  and $(d', d'')$ have common refinements. Then also the two diagrams $d,d''$ have  a common refinement.}
\end{prop}
\begin{proof}
Assume that $b$ and $b'$ are the common refinements of the diagrams  $(d, d')$  and $(d', d'')$ , respectively. In particular,  the diagrams $b$ and $b'$ are common refinements of  the diagram $d'$. In view of the remarks proceeding the proposition, it suffices to prove that  $b$ and $b'$ have a common refinement.  In order to prove this,
 assume that  the diagrams $b: X\xleftarrow{G}B\xrightarrow{\Psi} Y$ and $b':X\xleftarrow{G'}B'\xrightarrow{\Psi'} Y$  are refinements of the diagram
$X\xleftarrow{F}A\xrightarrow{\Phi} Y$.  This situation is illustrated by the following  two diagrams
\begin{equation*}
\begin{CD}
X@<F<<A@>\Phi>>Y\\
@. @AAHA @. \\
X@<G<<B@>\Psi>>Y\\.\\
\end{CD}
\qquad  \qquad
\begin{CD}
X@<F<<A@>\Phi>>Y\\
@. @AAH'A @. \\
X@<G'<<B'@>\Psi'>>Y\\.\\
\end{CD}
\end{equation*}
where  the functors $H:B\rightarrow A$ and $H':B'\rightarrow A$ are
equivalences.   We take the fibered product $L=B\times_{A} B'$ and
define the diagram $c$ by
$$
c:\, X\xleftarrow{G\circ\pi_1}L\xrightarrow{\Psi\circ\pi_1}Y.
$$
We claim that $c$  is a common refinement of the diagrams $b$ and $b'$.  In view of Theorem \ref{major}, the projection
$\pi_1:\el \rightarrow B$ is an equivalence. The
projection $\pi_2:\el \rightarrow B'$  has the same properties as
$\pi_1$.  We have to show that $\Psi\circ\pi_1$ and $\Psi'\circ\pi_2:\el \to Y$
are naturally equivalent as well as $G\circ\pi_1$ and
$G'\circ\pi_2:\el \to X$.  Note  that if $(x, \varphi, x')\in\el$, then
$$
\Psi'\circ\pi_2(x, \varphi ,x')=\Psi'(x'), \quad  \Psi \circ\pi_1(x, \varphi ,x')=\Psi(x).
$$
and $\varphi :H(x)\to H'(x').$
Since  the functors  $\Phi \circ H$ and $\Psi$ as well as $\Phi \circ H'$ and $\Psi'$ are naturally equivalent, there exist two sc-smooth maps $\tau_1:B\to {\bf Y}$ and $ \tau_2:B'\to {\bf Y}$ such that $\tau_1(x)$ is a morphism $\Phi \circ H(x)\to \Psi (x)$ and $\tau_2(x')$   is a morphism $\Phi \circ H'(x')\to \Psi' (x')$.  The sc-smooth map $\tau:\el \to {\bf Y}$,  given by
$$\tau (x, \varphi, x')=\tau_2(x')\circ \Phi (\varphi )\circ \tau_1 (x)^{-1},$$
defines the  natural equivalence between  the functors $\Psi' \circ\pi_2$ and  $\Psi\circ\pi_1:\el \to Y$.
Similar arguments prove that  the two functors $G\circ\pi_1$  and $G'\circ\pi_2:\el \to X$ are naturally equivalent. This  completes the proof  of the proposition.
\end{proof}
\begin{defn}
 Let $X$ and $Y$ be ep-groupoids. A {\bf generalized map}
$\mathfrak{a}:X\Rightarrow Y$ is by definition   the equivalence
class of a diagram
$$
d : X\xleftarrow{F}A\xrightarrow{\Phi}Y
$$
where $A$ is an ep-groupoid and where $F$ is an equivalence and $\Phi$ is an sc-smooth functor. We shall use  the notation
$\mathfrak{a}=[d]=[X\xleftarrow{F}A\xrightarrow{\Phi}Y]$.
\end{defn}

 In order to define the composition $ \mathfrak{b}\circ \mathfrak{a}$ of  the generalized maps $\mathfrak{a}:X\Rightarrow Y$ and $\mathfrak{b}:Y\Rightarrow Z$ we
choose a representatives $X\xleftarrow{F}A\xrightarrow{\Phi} Y$ and $Y\xleftarrow{G}B\xrightarrow{\Psi} Z$  of the equivalence classes $\mathfrak{a}$ and $\mathfrak{b}$, and consider the diagram
$$X\xleftarrow{F}A\xrightarrow{\Phi} Y\xleftarrow{G}B\xrightarrow{\Psi} Z.$$
We  replace the middle portion $A\xrightarrow{\Phi} Y\xleftarrow{G}B$ of the diagram by the
ep-groupoid $A\times_YB$. In view of Theorem \ref{major}, the projection $\pi_1:A\times_YB\to A$ is
an equivalence and we can build the diagram
\begin{equation}\label{groupoideq5}
X\xleftarrow{F\circ \pi_1}A\times_YB\xrightarrow{\Psi\circ \pi_2} Z.
\end{equation}
The map  $F\circ\pi_1$ is an equivalence and the map $\Psi\circ \pi_2$ an sc-smooth functor. If
$X\xleftarrow{F'}A'\xrightarrow{\Phi'}Y$ is equivalent to the diagram $d$ and
$Y\xleftarrow{G'}B'\xrightarrow{\Psi'}Z$ is equivalent to the diagram $d_1$, then  the two diagrams
$$\text{$X\xleftarrow{F\circ \pi_1}A\times_YB\xrightarrow{\Psi\circ \pi_2} Z$\  and \
$X\xleftarrow{F'\circ \pi_1}A'\times_YB'\xrightarrow{\Psi'\circ \pi_2} Z$}$$
are equivalent.
Therefore, the composition of the equivalence classes having the
representatives  $X\xleftarrow{F}A\xrightarrow{\Phi} Y$ and
$Y\xleftarrow{G}B\xrightarrow{\Psi} Z$ can be defined as the
equivalence class of the  diagram \eqref{groupoideq5}. The identity
morphism $1_Y$ of the groupoid $Y$ can be identified with the equivalence class of the diagram
\begin{equation}\label{groupoideq6}
Y\xleftarrow{1_Y}Y\xrightarrow{1_Y}Y.
\end{equation}
If the functor $F:X\to Y$ is an equivalence, then the diagram
$Y\xleftarrow{F}X\xrightarrow{F} Y$ is a refinement of
\eqref{groupoideq6} and hence belongs to the class $1_Y$. A smooth
functor $\Phi:X\to Y$ can be identified with the equivalence class
of the diagram
$$
X\xleftarrow{1_X}X\xrightarrow{\Phi}Y,
$$
denoted by $[\Phi]$.
As a special case, an equivalence functor $F:X\to Y$ is identified with the
equivalence class of the diagram
$$
X\xleftarrow{1_X}X\xrightarrow{F}Y.
$$
The diagram defines  the equivalence class  $[F]$. Its inverse, denoted by
$[F]^{-1}$, is readily identified with the equivalence class of the
diagram
$$
Y\xleftarrow{F}X\xrightarrow{\text{Id}_X}X
$$
so that  $[F]\circ [F]^{-1}=1_Y$ and $[F]^{-1}\circ [F]=1_X$.

We see
that  those functors which originally are equivalences become
invertible in the new category.

The following lemma explains the relationship between generalized
maps and induced maps between the underlying orbit space. The proof  is
straightforward.
\begin{lem}
Let $d:\ X\xleftarrow{F}A\xrightarrow{\Phi}Y$ be a diagram  between
ep-groupoids in which $F$ is an equivalence and $\Phi$ an sc-smooth
functor. Then $d$ induces an $\ssc^0$-map $\abs{d}:\abs{X}\rightarrow \abs{Y}$ between the orbit spaces
defined by
$$
\abs{d}= \abs{\Phi}\circ \abs{F}^{-1}.
$$
If a second diagram $d'$ refines $d$ then $\abs{d'}=\abs{d}$. In particular,
equivalent diagrams induce the same map $\abs{X}\rightarrow \abs{Y}$.

Hence a generalized map $X\Rightarrow Y$  between ep-groupoids which is an equivalence class $[d]$ of diagrams
induces the  canonical  map $\abs{[d]}:=\abs{d}:\abs{X}\to \abs{Y}$ between the orbit spaces.
\end{lem}
One might ask if one can characterize invertible generalized maps,
i.e. generalized maps $\mathfrak{a}$ which have an inverse in
$\mathfrak{Ep}$. An obvious conjecture is that an invertible
generalized map $\mathfrak{a}:X\Longrightarrow Y$ can be represented
by a diagram $X\xleftarrow{F}A\xrightarrow{G}Y$, where $F$ and $G$
are equivalences. This is,  in fact,  true if one deals with the special
case of \'etale proper Lie groupoids (of finite dimension). The
standard proof of the above fact uses the inverse function theorem
and a dimension argument and this part cannot be generalized. In
fact,  it seems doubtful to be true in our more general framework.
Therefore,  we introduce the following two notions.
\begin{defn}\label{twonot}
A generalized map $\mathfrak{a}:X\Rightarrow Y$ is called {\bf
invertible} if there exists a generalized map
$\mathfrak{b}:Y\Longrightarrow X$ satisfying  $\mathfrak{b}\circ
\mathfrak{a}=1_X$ and $\mathfrak{a}\circ\mathfrak{b}=1_Y$. We call a generalized map
$\mathfrak{a}$ {\bf strongly invertible},  or {\bf s-invertible} for
short, if it is the equivalence class of a diagram $X\xleftarrow{F}A\xrightarrow{G}Y$  in which both functors $F$ and $G$ are
equivalences. An
s-invertible element is called an {\bf s-isomorphism}.
\end{defn}
If $\mathfrak{a}=[X\xleftarrow{F}A\xrightarrow{G}Y]$  is s-invertible, then  its  inverse $\mathfrak{a}^{-1}$  can be represented by
$Y\xleftarrow{G}A\xrightarrow{F}X$. Indeed,  the composition
$$
[Y\xleftarrow{G}A\xrightarrow{F}X]
\circ [X\xleftarrow{F}A\xrightarrow{G}Y]$$
is represented by the diagram
$X\xleftarrow{F\circ\pi_1}A\times_YA\xrightarrow{F\circ\pi_2}X$ which is refined by the diagram  $X\xleftarrow{F}A\xrightarrow{F}X$ in view of the equivalence
 $A\rightarrow A\times_YA$  defined as $a\rightarrow (a,a)$.
 The diagram   $X\xleftarrow{F}A\xrightarrow{F}X$ also refines the identity diagram
$X\xleftarrow{1_X}X\xrightarrow{1_X}X$ via the
equivalence $F:A\rightarrow X$. Hence our candidate for
$\mathfrak{a}^{-1}$ is indeed the inverse of $\mathfrak{a}$.

This shows that the
inverse of an s-invertible element is s-invertible. It is easily
verified that the composition of two s-invertible maps is
s-invertible.

In the following we denote by $\mathfrak{Ep}$ the
category whose objects are  ep-groupoids and whose morphisms  are  generalized maps.

\subsection{Strong Bundles over Ep-Groupoids}\label{ssect1.4}
The section is devoted to strong bundles  over ep-groupoids and extends the  previously
 developed  ideas of equivalences and generalized maps to
this context.

We consider an ep-groupoid $X=(X, {\bf X})$ and a strong M-polyfold
bundle
$$p:E\to X$$
over the object space $X$ of the ep-groupoid.  Strong M-polyfold bundles are defined in Definition 4.9 of \cite{HWZ2}. In particular, $E$ is an M-polyfold of type $1$ and $X$ is an M-polyfold of type $0$. Moreover, $p:E\to X$ is a surjective sc-smooth map  and the fibers
$$p^{-1}(x)=E_x$$
over $x\in X$ carry the structure  of a Banach space. Since the source map $s:{\bf
X}\to X$ is by definition a local $\ssc$-diffeomorphism, the fibered product
$$
{\bf X}{{_s}\times_{p}}E=\{(g, e)\in {\bf X}\times E\vert \, s(g)=p(e)\}$$
is an M-polyfold in view of Lemma \ref{ert}. Moreover, the bundle
$${\bf E}={\bf X}{{_s}\times_{p}}E\xrightarrow{\pi_1} {\bf X}$$
is, as  the pull-back of a strong M- polyfold bundle, also a strong M-polyfold bundle in view of Proposition 4.11 in \cite{HWZ2}.

 Now we assume that there exists a  {\bf strong bundle
map}  $\mu:{\bf E}\to E$  which covers the target map $t:{\bf X}\to X$ of the ep-groupoid so that
$$t\circ \pi_1(g, e)=p\circ \mu (g, e)$$
for all $(g, e)\in {\bf X}{{_s}\times_{p}}E$,

\begin{equation*}
\begin{CD}
 {\bf X}{{_{s}}\times_{p}}E@>\mu>>E\\
@V\pi_1VV     @VVp V \\
{\bf X}@>t>> X.\\
\end{CD}
\end{equation*}
\mbox{}\\[4pt]
In addition, we assume that the postulated bundle map $\mu$ satisfies the following
properties.
\begin{Myitemize}
\item $\mu$ is a surjective local $\ssc$-diffeomorphism
and linear on the fibers $E_x$.
\item $\mu(1_x, e_x)=e_x$ for  all $x\in X$ and $e_x\in E_x$.
\item  $\mu (g\circ h,  e)=\mu (g, \mu (h, e))$\\
 for  all $g, h\in {\bf X}$ and $e\in E$ satisfying  $s(h)=p(e)$ and $t(h)=s(g)=p(\mu (h,e))$.
\end{Myitemize}
It follows, in particular, that
$$\mu (g, \cdot ):E_x\to E_y$$
is a linear isomorphism if $g:x\mapsto y$ is a morphism in ${\bf X}$.

With the above data we shall define the ep-groupoid
$$E=(E, {\bf E})$$
in which the object set $E$ is the M-polyfold $E$ we started with and ${\bf E}={\bf X}{{_{s}}\times_{p}}E$ is the above fiber product.  The source and target maps $s, t:{\bf E}\to E$ are defined as follows,
\begin{align*}
s(g, e)&=e\\
t(g, e)&=\mu (g, e),
\end{align*}
if $(g, e) \in {\bf X}{{_{s}}\times_{p}}E.$ These maps $s$ and $t$ are fiberwise linear surjective  local $\ssc_\triangleleft$-diffeomorphisms covering the source and target maps ${\bf X}\to X$. Indeed,
\begin{align*}
p\circ s(g, e)&=p(e)=s(g)\\
p\circ t(g, e)&=p\circ \mu(g, e)=t\circ \pi_1(g, e)=t(g).
\end{align*}
The identity morphism at $e_x\in E_x$ is the pair $(1_x, e_x)\in {\bf X}$ and if $g:x\mapsto y$ is a morphism in ${\bf X}$, the inverse of $(g, e_x)\in {\bf E}$ is the pair $i (g, e_x)=(g^{-1}, \mu (g, e_x))$. The multiplication map in $E$ is defined by
$$(h, f)\circ (g, e):=(h\circ g, e)$$
whenever  $f=\mu (g, e)\in E$.

The two sc-smooth projection maps $p:E\to X$ and $\pi_1:{\bf E}\to {\bf X}$ together define an  sc-smooth functor denoted  by
$$P:E\to X$$
between the two ep-groupoids $E=(E, {\bf E})$ and $X=(X, {\bf X})$.

We shall refer to this functor $P:E\to X$ as to a  {\bf strong bundle over the  ep-groupoid $X$}.

 Such a strong bundle is, in particular, an ep-groupoid together with a functor $P$ onto the base ep-groupoid and  we shall use the same letter for the two induced maps on the object and the morphism sets, namely,
$$E\xrightarrow{P}X\quad \text{and}\quad {\bf E}\xrightarrow{P}{\bf X}.$$

\begin{defn}
{\em A {\bf linear strong bundle morphism} $\Phi:P\to P'$ between the two strong bundles $P:E\to X$ and $P':E'\to X'$ over the ep-groupoids  $X$ and $X'$ consists of a functor $\Phi:E=(E, {\bf E})\to E'=(E', {\bf E}')$ between ep-groupoids which is linear on the fibers and which covers an sc-functor $\varphi:X\to X'$ between the bases. Moreover, the functor $\Phi$ induces strong bundle maps $\Phi:E\to E'$ and $\Phi:{\bf E}\to {\bf E}'$ between the object sets and morphism sets,

\begin{equation*}
\begin{CD}
E@>\Phi>>E'\\
@VPVV     @VVP' V \\
X@>\varphi>> X' .\\
\end{CD}
\end{equation*}
\mbox{}\\[1ex]
}
\end{defn}
There is a distinguished class of linear strong bundle morphisms which
generalizes equivalent functors between ep-groupoids.
\begin{defn}
{\em A {\bf linear strong bundle equivalence} $\Phi:P\to P'$ between the two strong bundles $P:E\to X$ and $P':E'\to X'$ over ep-groupoids is a linear strong bundle morphism $\Phi:P\to P'$ satisfying the following properties.
\begin{itemize}
\item[(1)] The functor $\Phi:E\to E'$ is an equivalence of ep-groupoids, covering the equivalence $\varphi:X\to X'$ between the underlying  ep-groupoids.
\item[(2)]  The induced maps $\Phi:E\to E'$ and $\Phi:{\bf E}\to {\bf E}'$ between the object sets and the morphism sets preserve the strong bundle structures and are locally strong bundle isomorphisms.
\end{itemize}
}
\end{defn}
For notational convenience we shall abbreviate these notions as follows.
\begin{itemize}
\item[ ] {\bf bundle map\phantom{valence}}\quad:=\quad{\bf linear strong bundle morphism}
\item[ ]{\bf bundle equivalence}\quad:=\quad{\bf linear strong bundle
equivalence}.
\end{itemize}
In order to generalize the notion of natural equivalence we consider two strong bundles $P:E\to X$ and $P':E'\to X'$ over ep-groupoids and two bundle maps
$$
\Phi, \Psi:P\to P'.
$$

\begin{defn}
{\em The  bundle maps $\Phi$ and $\Psi$  are called {\bf naturally equivalent}, if there exists a natural transformation
$$T:E\to {\bf E}'$$
for the two functors $\Phi, \Psi:E\to E'$ between ep-groupoids (in the sense of Definition \ref{ntrans}) and a natural transformation
$$\tau:X\to {\bf X'}$$
between the underlying functors $\varphi, \psi:X\to X'$ which commute with the functors $P:E\to X$ and $P':E'\to X'$ so that
$$\tau\circ P(e)=P'\circ T(e),\quad e\in E.$$}
\end{defn}

In order to reformulate the above definition we derive a simple description of  the  natural transformation $T:\Phi\rightarrow\Psi$ as a canonical lift of the underlying natural transformation $\tau:\varphi\rightarrow \psi$. By definition, the map
$$
T:E\rightarrow {\bf E}':= {\bf X}'{{_s}\times_{P'}}E',
$$
which covers the natural transfomation $\tau:X\rightarrow {\bf X}'$ has the form
$$
T(e)=(\tau(P(e)),A(e)).
$$
Since $\Phi(e)=s(T(e))=A(e)$ and $\Psi(e)=t(T(e))=\mu(\tau(P(e)),A(e))$
it follows that
$$
T(e)=(\tau(P(e)),\Phi(e))\ \hbox{and}\ \ \Psi(e)=\mu(\tau(P(e)),\Phi(e)).
$$
This implies that $T$ is a strong bundle map between  the bundles $E\rightarrow X$
and ${\bf X}'{{_s}\times_{P'}}E'\rightarrow {\bf X}'$ covering $\tau$.
In view of this discussion we obtain an equivalent definition as follows.
\begin{defn}{\em
A {\bf natural transformation} $T$ between bundle maps $\Phi:P\to P'$ and $\Psi:P\to P'$
(covering the underlying functors $\varphi:X\to X'$  and $\psi:X\to X'$)
is a strong bundle map
$$
T:E\rightarrow {\bf X}'{{_s}\times_{P'}}E'
$$
covering a natural transformation $\tau:X\rightarrow {\bf X}'$ between  the functors $\varphi$ and $\psi$ of the form
$$
T(e)=(\tau(P(e)),\Phi(e))
$$
and satisfying
$$
\mu(\tau(P(e)),\Phi(e))=\Psi(e).
$$}
\end{defn}

Let us denote by $\mathfrak{SEp}$ the category whose objects are
strong bundles over ep-groupoids. As in the case of  the category $\mathfrak{Ep}$
we take as morphisms a class of generalized maps introduced next.   We consider two strong bundles
$P:E\to X$ and $P':E'\to X'$ over ep-groupoids and study the
diagrams
$$
E\xleftarrow{\Phi} E''\xrightarrow{\Psi}E',
$$
where $\Phi'':E''\to X''$ is a  third strong bundle over an ep-groupoid, and where $\Phi$ is a bundle equivalence and $\Psi$ a bundle map.  Following the earlier construction of a generalized map between ep-groupoids, one introduces the notion of a refinement of a diagram and then calls two diagrams equivalent, if they possess a common refinement. The equivalence class of the  diagram $E\xleftarrow{\Phi} E''\xrightarrow{\Psi}E',$ is then, by definition, a {\bf generalized strong bundle map}, for simplicity called a {\bf generalized bundle map} and denoted by
$$\mathfrak{A}:E\Rightarrow E'.$$

By
$$\mathfrak{a}:X\Rightarrow X'$$
we shall denote the associated underlying generalized map between the base ep-groupoids. One has to keep in mind that in the fibers the induced maps are linear and preserve the strong bundle  structure. The notion of invertible and s-invertible are defined as before (Definition \ref{twonot}), namely as follows.
\begin{defn}
{\em  A {\bf generalized bundle map} $\mathfrak{A}:E\Rightarrow E'$ is called {\bf invertible} if there exists a generalized bundle map $\mathfrak{B}:E'\Rightarrow E$ satisfying $\mathfrak{B}\circ \mathfrak{A}=1_E$ and  $\mathfrak{A}\circ \mathfrak{B}=1_{E'}.$ The {\bf generalized bundle map} $\mathfrak{A}:E\Rightarrow E'$ is called {\bf s-invertible}, if it can be represented by a diagram
 $$E\xleftarrow{\Phi} E''\xrightarrow{\Psi}E'$$
 in which both bundle maps are equivalences.}
 \end{defn}
An s-invertible generalized bundle map $\mathfrak{A}:E\Rightarrow E'$ will in the following be called an {\bf s-bundle isomorphism}.

We point out that an s-bundle isomorphism $\mathfrak{A}:E\Rightarrow E'$ covers automatically an s-isomorphism $\mathfrak{a}:X\Rightarrow X'$
between the underlying base ep-groupoids.

\begin{defn}
{\em The two strong bundles $P:E\to X$ and $P':E'\to X'$ over ep-groupoids  are called {\bf strong bundle equivalent}, if there exists an s-bundle isomorphism $\mathfrak{A}:E\Rightarrow E'$.}
\end{defn}

\begin{defn}
{\em A (sc-smooth) {\bf section} of the strong bundle  $P:E\rightarrow
X$ over the ep-groupoid $X$  is an  sc-smooth functor $F:X\rightarrow
E$ satisfying $P\circ F=Id_X$. An {\bf $\ssc^+$-section}  is an
sc-smooth section $F:X\rightarrow E$ inducing an sc-smooth functor
$X\rightarrow E^{0,1}$, where $E^{0,1}$ has the grading
$(E^{0,1})_m=E_{m,m+1}$ for all  $m\geq 0.$

A {\bf Fredholm
section} $F$ of the strong bundle $P:E\rightarrow X$ is an sc-smooth
functor which,  as a section on the object sets,  is an  M-polyfold
Fredholm section as defined in Definition 3.6 in  \cite{HWZ3}. }
\end{defn}
The functoriality of the section $F:X\to E$ requires  that
$$\mu (\varphi, F(x))=F(y)$$
 whenever there is a morphism $\varphi:x\to y$ in ${\bf X}$.

The space of sc-smooth sections  of the bundle $P:E\rightarrow X$ is
denoted by $\Gamma(P)$,  the  space of sc$^+$-sections is
denoted by $\Gamma^+(P)$, and the space of Fredholm sections by ${\mathcal F}(P)$. The following observation will be useful later on.
\begin{prop}\label{newprop2.20}
 Consider the two strong bundles $P:E\to X$ and $P':E'\to Y$ over the ep-groupoids $X$ and $Y$ and let $\Phi, \Psi:P'\to P$ be two bundle equivalences which are naturally equivalent. If $F:X\to E$ is an sc-smooth section of the strong bundle $P$, then the pull-back sections $\Phi^* (F)$ and $\Psi^* (F):Y\to E'$ of the bundle $P'$ agree,
$$\Phi^* (F)(y)=\Psi^*(F)(y),\quad y\in Y.$$

\end{prop}
\begin{proof}
Let  $\phi$ resp. $\psi:Y\rightarrow X$ be the functors on the base ep-groupoids associated with the bundle equivalences $\Phi$ resp. $\Psi:E'\to E$. Abbreviating $G=\Phi^* (F)$ and $H=\Psi^*(F)$ we know, by definition of the pull-back of a section,
$$\Phi\circ G (y)=F \circ \varphi (y)\quad \text{and}\quad \Psi \circ H(y)=F\circ \psi (y).$$
Denoting both natural transformations by $\tau$ we have for $y\in Y$ the morphisms
$\tau (y):\varphi (y)\to \psi (y)$ in ${\bf X}$ and $\tau (e):\Phi (e)\to \Psi (e)$ in ${\bf E}$, where $e\in E'$ is in the fiber over $y$, $P'(e)=y$. Hence, by definition of the target map $t: {\bf E}\to E$,
$$\mu (\tau (y), \Phi (e))=\Psi (e)\quad \text{and}\quad P'(e)=y.$$
Consequently,  since the section $F$ is a functor and $\tau (y)$ is a morphism $:\varphi (y)\to \psi (y)$, we find
\begin{equation*}
\begin{split}
\Psi (G(y))&=\mu(\tau(y),\Phi(G(y)))\\
            &=\mu(\tau(y),F(\varphi(y)))=F(\psi(y))\\
            &=\Psi(H(y)).
\end{split}
\end{equation*}
By assumption, $\Psi$ is a bundle equivalence and $H$ and $G$ are sections. Therefore we conclude $G(y)=H(y)$  for all $y\in Y$ what we wanted to prove.
\end{proof}
In order to apply Proposition \ref{newprop2.20}, we consider the diagram
\begin{equation*}
\begin{CD}
E@<\Phi<<E''@>\Psi>>E'\\
@. @AA\Sigma A @. \\
E@<\Theta<<A@>\Xi>>E'\\.
\end{CD}
\end{equation*}
representing a bundle diagram refinement  in which all the functors  are
bundle equivalences. By definition there are  the natural equivalences $\Phi\circ \Sigma \simeq \Theta$ and $\Psi\circ\Sigma\simeq\Xi$ and we conclude from Proposition \ref{newprop2.20} for a section $F:X\to E$ of the strong bundle $P:E\to X$ over the ep-groupoid $X$ that
$$(\Phi\circ\Sigma)^\ast(F)=\Theta^\ast(F).$$
Similarly, if $G$ is a section of the strong bundle $P':E'\to X'$ over the ep-groupoid $X'$, then
$\Xi^* (G)=(\Psi \circ \Sigma)^* (G)$ and we can formulate the following definition.

\begin{defn}
{\em Let $E$ and $E'$ be two strong bundles over ep-groupoids, and let $F$ be a
section of $E$ and $G$ a section of $E'$.  The sections are
{\bf equivalent}  if there exists a diagram
$$E\xleftarrow{\Phi}E''\xrightarrow{\Psi}E'$$
of strong bundle equivalences such that for the pull-back sections, we have  $\Phi^\ast(F)=\Psi^\ast(G)$.}
\end{defn}
This is well-defined in the following sense. If we have two
equivalent diagrams connecting $E$ with $E'$ and the above equality holds for
one diagram, then it holds for the other diagram  as well. Indeed,  by the preceding
discussion,  it holds for a common refinement and therefore the
sections obtained by the pull-backs of the second diagram must be the
same on the image of the refinement. But since the latter is a
bundle equivalence this implies the assertion.

Also observe that
given a bundle equivalence $\Phi:E\rightarrow E'$ and a section $F$
of $E$,  there is a well-defined push forward section $\Phi_\ast(F)$ of the strong bundle $P':E'\to X'$.
Consequently, we obtain the following result.
\begin{prop}\label{rty}
Let $P:E\to X$ and $P':E'\to X'$ be two strong bundles over ep-groupoids.  Then an
 s-invertible generalized bundle map  $\mathfrak{A}:E\Rightarrow E'$  induces a well defined push-forward map $\mathfrak{A}_{\ast}:\Gamma (P)\to \Gamma (P')$ between sc-smooth sections which is a bijection. Its inverse is the pull-back map  $\mathfrak{A}^\ast:\Gamma (P')\to \Gamma (P)$. The same holds rue for the $\ssc^+$-sections $\Gamma^+(P)\to \Gamma^+(P')$ and the Fredholm sections ${\mathcal F}(P)\to {\mathcal F}(P')$.
\end{prop}

\subsection{Auxiliary Norms}
In \cite{HWZ3},  we introduced  the notion of an  auxiliary norm  on an M-polyfold. In this section we generalize this concept  to ep-groupoids by incorporating  morphisms.

We recall from Definition 5.5 in  \cite{HWZ3} that an {\bf auxiliary norm} $N$ for the M-polyfold bundle $p:Y\to X$ consists of a continuous map $N:Y_{0,1}\to [0,\infty )$ having the following properties.
\begin{Myitemize}
\item For every $x\in X$, the induced map $N\vert (Y_{0,1})_x\to [0,\infty )$ on the fiber
$ (Y_{0,1})_x$ is a complete norm.
\item If $y_k\xrightarrow{m}y$, then
$$N(y)\leq \liminf_{k\to \infty} N(y_k).$$
\item If $N(y_k)$ is a bounded sequence and the underlying sequence $x_k=p(y_k)$ converges to $x\in X$, then $y_k$ has an $m$-convergent subsequence.
\end{Myitemize}
For the so called $m$-convergence $y_k\xrightarrow{m}y$ we refer to Definition 5.4 in \cite{HWZ3}. Roughly it means for the sequence $y_k=(p(y_k), e_k)\in Y_{0,1}$ that the base points $p(y_k)=x_k\in X$ converge in the M-polyfold to $x\in X$ and the fiber components $e_k\rightharpoonup e$ converge in a weak sense to $e\in Y_1$ satisfying $p(e)=x$.

\begin{defn}
{\em
Let $P:E\rightarrow X$ be a strong bundle over the ep-groupoid $X$.
An auxiliary norm for the strong bundle $P$ is a map
$N^\ast:E_{0,1}\rightarrow [0,\infty)$ having the following
properties.
\begin{Myitemize}
\item  As a map on the object set, the map $N^\ast$ is an auxiliary norm as defined in Definition 5.5 in \cite{HWZ3}.
\item If $\varphi:e\to e'$ is a morphism in ${\bf E}$, then
$N^\ast(e)=N^\ast(e')$.
\end{Myitemize}
}
\end{defn}

The existence of an auxiliary norm is guaranteed by the following
proposition which uses an  existence result in \cite{HWZ3}. {\bf We
assume that the local models for $P$ have reflexive fibers}, for example Hilbert spaces.
\begin{prop}
Every strong  bundle $P:E\rightarrow X$  over an ep-groupoid admits
an auxiliary norm.
\end{prop}
\begin{proof}
We begin with a strong local bundle $K\rightarrow O$. The existence of an auxiliary norm in this case has been proved in Proposition 5.6 in  \cite{HWZ3}.
Next consider a strong bundle $P:E\to X$ over the ep-groupoid $X$. By the discussion in Section \ref{ssect1.4}, we have the  strong M-polyfold bundle $P:E\to X$ over the ep-groupoid $X$ and the strong bundle map $\mu: {\bf E}\to E$ having  the property that if $g:x\to y$ is a morphism  in  ${\bf X}$, then $\mu (g, \cdot ):E_x\to E_y$ is a linear isomorphism.

Around  every point $x\in X$,  we choose a strong bundle chart (as in Definition 4.8 in \cite{HWZ2})
$$\Phi:E\vert U_{x}\to  K^{\mathcal R}\vert O$$
covering the $\ssc$-diffeomorphism $\varphi: U_{x} \to O$ between an open neighborhood $U_{x}\subset X$ of the point $x$ and the open subset $O$ of the splicing core $K^{\mathcal S}$ associated with the splicing ${\mathcal S}$.
The open neighborhood $U_{x}$ can be taken so small  that  the isotropy group ${\bf X}(x)$ acts  on it by  the natural representation $\varphi_g:U_{x}\to U_{x}$ for $g\in {\bf X}(x)$ in view of Theorem  \ref{localstructure}.   Since  an auxiliary norm exists on $K^{\mathcal R}\vert O$, we  pull it back to $E\vert U_{x}$ via the chart $\Phi$ to obtain an auxiliary norm $N_{x}'$  on $E\vert U_{x}$.
We use the following notation. If $y\in X$ and $g:y\to y'$ is a morphism, then we write $y'=g\cdot y$. In particular, $y'=g\cdot y$ if $y'= \varphi_g (y)$. Also,  we write $e'=g\cdot e$ if  $e\in E_x$ and $e'=\mu (g, e)\in E_y$.

For  $\wh{e}=(y, e)\in U_{x}\oplus  E_y$, we define
\begin{equation*}
\begin{split}
\wh{N}'_{x}(\wh{e})=\wh{N}'_{x}(y, e):&=\dfrac{1}{\sharp {\bf X}(x)}\sum_{g\in {\bf X}(x)}N_{x}'(\varphi_g (y), \mu (g, e))\\
&=\dfrac{1}{\sharp {\bf X}(x)}\sum_{g\in {\bf X}(x)}N_{x}'(g\cdot y, g\cdot e).
\end{split}
\end{equation*}
Then $\wh{N}'_{x}$ is an auxiliary norm on $E\vert U_{x}$ having  the additional property that if there is a morphism in ${\bf E}$ between  two objects $\wh{e}$ and $\wh{e}'$ belonging to $E\vert U_{x}$, then
$\wh{N}'_{x}(\wh{e})= \wh{N}'_{x}(\wh{e}')$.

Next we extend $\wh{N}'_{x}$ to $E\vert  \wh{U}_{x}$ where $\wh{U}_{x}=\pi^{-1}(\pi (U_{x}))$ is the saturation of the set $U_{x}$. Here $\pi:X\to \abs{X}$ is  the quotient map onto the orbits space $\abs{X}$ of the ep-groupoid $X$.

If $\wh{e}=(y, e)\in E\vert  \wh{U}_{x}$, then there is a morphism $g$ having the source $y$ and the target in $U_{x}$ such that $(y',e')=(g\cdot y, g\cdot e)\in U_{x}\oplus E_{y'}$ and we define
\begin{equation*}
\begin{split}
\wh{N}_{x}(\wh{e})=\wh{N}_{x}(y, e)&:=\wh{N}'_{x}(g\cdot y,  g\cdot e)=\wh{N}'_{x}(y', e').
\end{split}
\end{equation*}

This definition is independent of the choice of a morphism having its source equal to $y$ and its target in $U_{x}$.  Indeed, if  there is another morphism $g'$ satisfying $(y'', e'')=(g'\cdot y, g'\cdot e)\in U_{x}\oplus  E_{y'}$, we define  the morphism $h=g'\circ g^{-1}:y'\to y''$. It follows that $(y'', e'')=(h\cdot y', h\cdot e')$, and since $\wh{N}'_{x}$  is,  by construction,  an invariant auxiliary norm on $U_{x}$, we conclude  $\wh{N}_{x}'(y', e')=\wh{N}_{x}'(y'', e'')$ showing that $\wh{N}_{x}$ is well defined.

To see that the map  $\wh{N}_{x}$ is continuous on $E_{0,1}\vert \wh{U}_x$ we take $(x, e)\in E_{0,1}\vert \wh{U}_{x}$. Then  there is a morphism $g$ such that $(y', e')=(g\cdot y, g\cdot e)\in U_{x}\oplus E_y$.  By definition
$\wh{N}_{x}(x, e)=\wh{N}'_{x}(y', e')$ and $\wh{N}'_{x}$ is continuous near $(y', e')$.  The morphism $g$ has an extension to an sc-diffeomorphism of the from $t\circ s^{-1}$ defined near $x$ and the map $\mu$ is also a local sc-diffeomorphism. It follows that the map $\wh{N}_{x}$ near $(y, e)$ is equal  to the composition of a continuous map $\wh{N}_{x}'$ with the sc-diffeomorphisms $(t\circ s^{-1}, \mu \circ (t\circ s^{-1}, \text{id}))$,  and hence is continuous.

Finally, we glue all the local auxiliary norms $\wh{N}_{x}$ to obtain a globally defined auxiliary norm $N$ on $E$. To do this we consider the open covering  $\bigl(\pi^{-1}(\pi (U_{x})) \bigr)_{x\in X}$ of $X$. The ep-groupoid $X$ is a paracompact space. In view of Theorem 4.8 in \cite{HWZ7}, there exists a continuous partition of unity $(\beta_{x})_{x\in X}$ subordinated  to the open covering  $\bigl(\pi^{-1}(\pi (U_{x})) \bigr)_{x\in X}$ which is invariant under morphisms of $X$. For simplicity we use here for the partition of unity the same indices as for the open covering, allowing the supports of $\beta_x$ to be empty.
We would like to point out that  Theorem 4.8 in \cite{HWZ7}  requires the assumption that  the sc-structure of $X$ is based on separable sc-Hilbert spaces in order to obtain sc-smooth partition of unity, however,   in our case at hand  we don't need this assumption since we are only interested in a continuous partition of unity.
Now  for $\wh{e}=(y, e)=(P (\wh{e}), e)$ we set
$$N(\wh{e})=N(y, e)=\sum_{x\in X}\beta_{x}(P(\wh{e}))\cdot \wh{N}_{x}(\wh{e}).$$
Then $N$ is the desired auxiliary norm on $E$ which is compatible with the morphisms.
\end{proof}

\section{Polyfolds}\label{sect7.4}
Polyfolds and their bundles are the basic spaces on which we will
study Fredholm sections. The ep-groupoids can be viewed as the
models for this new class of spaces. The strong bundles over
ep-groupoids are  the  models for strong polyfold
bundles.
\subsection{Basic Definitions and Results}\label{basicresult}
In the following we shall use the category $\mathfrak{Ep}$ whose
objects are ep-groupoids and  morphisms are
generalized maps.

\begin{defn}\label{ngroupdefn3}
{\em Let $Z$ be a second countable paracompact topological space. A {\bf
polyfold structure} on $Z$ is a pair $(X,\alpha  )$ consisting of an
ep-polyfold groupoid  $X$ and a homeomorphism $\alpha:\abs{X}\to Z$ between the  orbit space  and the space $Z$. }
\end{defn}
Given a polyfold structure $(X,\alpha)$ on $Z$,  the
$\ssc^0$-structure on $\abs{X}$ defines a $\ssc^0$-structure on $Z$.

In order to formulate the equivalence relation between two polyfold structures, we recall from the previous section that an s-isomorphism $\mathfrak{a}:X\Rightarrow  Y$ between two ep-groupoids
is the equivalence class of a diagram
$$X\xleftarrow{F}W\xrightarrow{G}Y$$
connecting ep-groupoids in which both functors $F$ and $G$ are equivalences. Since equivalences induce
 $\ssc^0$-homeomorphism between orbit spaces, we obtain the $\ssc^0$-homeomorphism $\abs{G}\circ \abs{F}^{-1}:\abs{X}\to \abs{Y}$.  This homeomorphism is independent of the choice of a representative in the equivalence  class. Indeed, if
$$X\xleftarrow{F'}W'\xrightarrow{G'}Y$$
is an equivalent  diagram, it follows from the definition of the equivalence that
$\abs{G'}\circ \abs{F'}^{-1}=\abs{G}\circ \abs{F}^{-1}$ and we see that the s-isomorphism
$\mathfrak{a}:X\Rightarrow Y$ induces a canonical $\ssc^0$-homeomorphism
$$\abs{\mathfrak{a}}:\abs{X}\to \abs{Y}$$
between the orbit spaces.

\begin{defn}
{\em Let $Z$ be a second countable paracompact topological space.
Two polyfold structures $(X, \alpha)$ and $(Y ,\beta)$ on
$Z$ are called {\bf equivalent}, $(X, \alpha
) \simeq (Y, \beta)$, if there exists an s-isomorphism
$\mathfrak{a}:X\Rightarrow Y$ satisfying  $\beta\circ
\abs{\mathfrak{a}}=\alpha$.

\mbox{}
$$
\begindc{\commdiag}[3]
\obj (20,20){$\abs{X}$}
\obj (50,20){$\abs{Y}$}
\obj (35,0){$Z$}
\mor (20,20)(50,20){$\abs{\mathfrak{a}}$}[\atleft,\solidarrow]
\mor (20,20)(35,0){$\alpha$}[\atright,\solidarrow]
\mor (50,20)(35,0){$\alpha'$}[\atleft,\solidarrow]
\enddc
\mbox{}.
$$

 }
\end{defn}

In more detail, two polyfold structures on $Z$ are equivalent if and only if there exists
 a third ep-groupoid $W$ and two equivalences in the diagram
$
X\xleftarrow{F} W \xrightarrow{G}Y
$
satisfying
$$\alpha \circ \abs{F}=\beta \circ \abs{G}.$$

\begin{defn}\label{ngroupdefn4}
{\em A {\bf polyfold} is a second countable paracompact topological space
$Z$ equipped with an equivalence class of polyfold structures.}
\end{defn}


Consider a   polyfold  $Z$  and  a polyfold structure  $(X, \beta)$  on $Z$. The orbit space $\abs{X}$ of the ep-groupoid $X$ is equipped with a filtration
$\abs{X_0}=\abs{X}\supset \abs{X_1}\supset \cdots \supset  \abs{X_{\infty}}:=
\abs{\bigcap_{i\geq 0}X_i}$
where   $\abs{X_{\infty}}$ is dense in every $\abs{X_i}$. This  filtration induces,  via the homeomorphism $\beta:\abs{X}\to Z$,  the  filtration
$$Z_0
=Z\supset  Z_1\supset \cdots  \supset Z_{\infty}:
=\bigcap_{i\geq 0}Z_i$$
on the polyfold $Z$ in which the space $Z_{\infty}$ is dense in every $Z_i$.  Any  other   equivalent polyfold structure on $Z$ induces the same filtration. Every space $Z_i$ carries a polyfold structure which we denote by $Z^i$.

Next we define an  sc-smooth map between two polyfolds $Z$ and $Z'$.   We assume that   the polyfold structures of $Z$ and $Z'$ are represented by the pairs $(X, \alpha)$ and $(X', \alpha')$, respectively, and consider  a pair $(a,\mathfrak{a})$ consisting of a continuous map
$a:Z\rightarrow Z'$ and a  generalized map
$\mathfrak{a}:X\Longrightarrow X'$ satisfying
$$
\alpha'\circ |\mathfrak{a}| = a\circ\alpha,
$$

\mbox{}
$$
\begindc{\commdiag}[3]
\obj (20,20){$\abs{X}$}
\obj (40,20){$\abs{X'}$}
\obj (20,0){$Z$}
\obj (40,0){$Z'$}
\mor (20,20)(40,20){$\abs{\mathfrak{a}}$}[\atleft, \solidarrow]
\mor (20,0)(40,0){$a$}[\atleft, \solidarrow]
\mor (20,20)(20,0){$\alpha$}[\atright, \solidarrow]
\mor (40,20)(40,0){$\alpha'$}[\atleft, \solidarrow]
\enddc
\mbox{}.
$$
\mbox{}\\[4pt]

Now we take two other  equivalent representatives of the polyfold structures of
$Z$ and $Z'$, namely $(Y,\beta )\simeq (X,\alpha )$ and
$(Y',\beta')\simeq (X', \alpha')$. Hence there exist s-isomorphisms
$$
\mathfrak{f}:Y\Longrightarrow X\ \hbox{and}\ \
\mathfrak{f}':Y'\Longrightarrow X'
$$
satisfying  $\alpha\circ |\mathfrak{f}|=\beta$ and
$\alpha'\circ|\mathfrak{f}'|=\beta'$.

If $(b,\mathfrak{b})$ is another pair consisting of a continuous map $b:Z\to Z'$ and a generalized map $\mathfrak{b}:Y\Longrightarrow  Y'$ satisfying $\beta'\circ \abs{\mathfrak{b}}=b\circ \beta$,  we introduce the
following definition.
\begin{defn}
{\em The two pairs  $(a,\mathfrak{a})$ and $(b,\mathfrak{b})$ are called {\bf
equivalent} provided $a=b$ and
$
\mathfrak{f}'\circ\mathfrak{b}=\mathfrak{a}\circ\mathfrak{f}.
$}
\end{defn}

This defines indeed an equivalence relation. (The composition $\circ$ of two generalized maps is defined in section \ref{sect2.3}.)  The situation is illustrated by the following diagram.

$$
\begindc{\commdiag}[3]
\obj (-10,15){$Z$}
\obj (-10,0){}
\obj (10,0){$\abs{Y}$}
\obj (10,30){$\abs{X}$}
\obj (40,30){$\abs{X'}$}
\obj (40,0){$\abs{Y'}$}
\obj (60,15){$Z'$}
\obj (60,0){}
\mor (-10,15)(-10,-9){}
\mor (60,-9)(60,15){}
\mor (10,28)(-10,15){$\alpha$}
\mor (10,1)(-10,15){$\beta$}
\mor (10,30)(10,0){$\abs{\mathfrak{f}}$}[\atright,\solidarrow]
\mor (40,30)(40,0){$\abs{\mathfrak{f}'}$}[\atright,\solidarrow]
\mor (10,0)(40,0){$\abs{\mathfrak{b}}$}[\atleft,\solidarrow]
\mor (10,30)(40,30){$\abs{\mathfrak{a}}$}[\atleft,\solidarrow]
\mor (40,28)(60,15){$\alpha'$}
\mor (40,1)(60,15){$\beta'$}
\cmor ((-10,-5)(-9,-9)(-5,-10)(30,-10)(55,-10)(59,-9)(60,-5))  \pup(25,-12){$a=b$}
\enddc
\mbox{}\\[2ex]
$$

Finally,  we can define the notion of an $\ssc$-smooth polyfold map.
\begin{defn}
An {\bf sc-smooth map  $:Z\rightarrow Z'$ between two polyfolds} is
an equivalence class $[(a,\mathfrak{a})]$ of pairs  consisting of a
continuous map $a:Z\rightarrow Z'$  between the underlying
topological spaces and a generalized map $$
\mathfrak{a}:X\Rightarrow X',
$$
where $(X,\alpha)$ and $(X',\alpha')$ are representatives of the
polyfold structures of $Z$ and $Z'$ , so that
$$
\alpha'\circ |\mathfrak{a}| = a\circ \alpha.
$$
\end{defn}
Sometimes we shall call the map  $a:Z\rightarrow Z'$ a polyfold map
and do not explicitly mention the `overhead' $\mathfrak{a}$ if there
is no danger of confusion. So we might say $a:Z\rightarrow Z'$ is a
sc-smooth map between polyfolds, but the reader must be aware that
this is just an abbreviation  for a lot of data.

The standard
example of a polyfold is the orbit space $|X|$ of an ep-groupoid,
where a polyfold structure is given by the pair  $(X,Id_{|X|})$. Given two
such polyfolds $|X|$ and $|Y|$ the space of sc-smooth polyfold maps
$\abs{X}\to \abs{Y}$ is nothing else but all pairs
$(\abs{\mathfrak{a}},\mathfrak{a})$ where $\mathfrak{a}$ is a
generalized map $X\Rightarrow Y$. A representative (something
one can work with if one has to make constructions) for such a generalized  map
is a diagram $X\xleftarrow{F}A\xrightarrow{G}Y$ in which $F$ is an
equivalence and $G$ an sc-smooth functor.

In \cite{HWZ2} we have introduced the {\bf degeneracy index}  $d:X\to \N$
on an M-polyfold $X$ as follows.
Around a point $x\in X$ we choose  an
M-polyfold chart $\varphi:U\rightarrow K^{\mathcal S}$ where
$K^{\mathcal S}$ is the splicing core associated with the splicing
${\mathcal S}=(\pi,E,V)$. Here $V$ is an open subset of a partial
quadrant $C$ contained in the sc-Banach space $W$. By definition
there exists a linear sc-isomorphism from $W$ to ${\mathbb R}^n\oplus
Q$ mapping  $C$ onto $[0,\infty)^n\oplus Q$.  Identifying the
partial quadrant $C$ with $[0,\infty )^n\oplus Q$ we shall use the
notation $\varphi=(\varphi_1,\varphi_2)\in [0,\infty)^n \oplus
(Q\oplus E)$ according to the splitting of the target space of
$\varphi$. We associate with the point $x\in U$ the integer  $d(x)$
defined by
\begin{equation*}\label{di}
d(x)=\sharp \{\text{coordinates of $\varphi_1(x)$ which are equal to
$0$}\}.
\end{equation*}
By Theorem 3.11 in \cite{HWZ2}, the integer $d$ does not depend on the choice of the M-polyfold chart used. A point $x\in X$ satisfying $d(x)=0$ is called an interior point of $X$. The set $\partial X$ of  {\bf  boundary points} of $X$ is defined as
$$\partial X=\{x\in X\vert \, d(x)>0\}.$$
A point $x\in X$ satisfying $d(x)=1$ is called a {\bf good boundary point}. A point satisfying $d(x)\geq 2$ is called a {\bf corner} and $d(x)$ is the order of this corner.

\begin{defn}\label{ddindex}
The closure of a connected component of the set $X(1)=\{x\in X\vert \, d(x)=1\}$ is called a {\bf face of the M-polyfold $X$}.
\end{defn}

Around every point $x_0\in X$ there exists an open neighborhood $U=U(x_0)$ so that every $x\in U$ belongs to precisely $d(x)$ many faces of $U$. This is easily verified. Globally it is always true that $x\in X$ belongs to at most $d(x)$ many faces and the strict inequality is possible.

In order to define the degeneracy index on a polyfold
we first look at an ep-polyfold groupoid $X$. Its degeneracy index
$d:X\to \N$ is defined on the M-polyfold $X$ of objects as well as
the M-polyfold ${\bf X}$ of morphisms. As usual we denote by $s,
t:{\bf X}\to X$ the source and target maps of $X$. We have already
seen previously that the existence of a morphism $g:x\rightarrow x'$
implies that these three items $x, x'$ and $g$ have the same maximal
level. The degeneracy map is another integer valued map defined on
$X$ having this property.
\begin{lem}\label{groupoidlem1}
The following holds true for   ep-polyfold groupoids $X$ and $Y$.
\begin{itemize}
\item[(i)] If  $g:x\rightarrow x'$ is a morphism,   then
$d(x)=d(x')=d(g)$.
\item[(ii)] If the functor $F:X\to Y$ is an equivalence,  then
$d_{X}(x)=d_{Y}(F(x))$ for all $x\in X$.
\end{itemize}
\end{lem}
\begin{proof}

Since $s$ and  $t$  are local $\ssc$-diffeomorphisms,  we conclude
from the statement about the corner recognition, Proposition 3.13 in \cite{HWZ2},  that $d(g)=d(s(g))=d(t(g))$. The same
proposition implies  also the second assertion because  the map $F:X\to
Y$ is a local sc-diffeomorphism.
\end{proof}
Suppose the pair $(X, \alpha  )$ is a polyfold structure on $Z$. If $\abs{x}\in \abs{X}$ and $x, x'\in \abs{x}$, then there exists  a morphism $g:x\to x'$ and we conclude  from Lemma
\ref{groupoidlem1} that $d(x)=d(x')$. Hence the map $\abs{d}:\abs{X}\to \N_0$ given by
$\abs{d}(\abs{x})=d (x)$ is well defined. Using this map,  we define the map $d: Z\to \N_{0}$ on the polyfold by setting
\begin{equation}\label{dindex}
d(z)=\abs{d}(\alpha^{-1}(z)).
\end{equation}
The definition does not depend on the particular choice of the polyfold structure in the   equivalence class. To prove this claim we consider two equivalent  polyfold structures $(X, \alpha)$ and $(Y, \beta)$  on $Z$.  Hence  there exists a diagram $X\xleftarrow{F}W\xrightarrow{G}Y$ connecting ep-groupoids in which the functors $F$ and $G$ are equivalences satisfying
$\alpha\circ \abs{F}=\beta\circ \abs{G}:\abs{W}\to Z$. If
 $\alpha^{-1}(z)=\abs{x}$ and  $\beta^{-1}(z)=\abs{y}$, then
$\abs{G}\circ \abs{F}^{-1}(\abs{x})=\abs{y}$. Hence we find $w\in W$ such that $\abs{F}(\abs{w})=\abs{F(w)}=\abs{x}$ and $\abs{G}(\abs{w})=\abs{G(w)}=\abs{y}$. Therefore, there exists $x'\in \abs{x}$ and $y'\in \abs{y}$ satisfying $F(w)=x'$ and $G(w)=y'$ and with  Lemma \ref{groupoidlem1} we compute,
\begin{equation*}
\begin{split}
\abs{d_{X}}(\alpha^{-1}(z))&=\abs{d_{X}}(\abs{x})= d_X(x)\\
&=d_X(x')=d_X(F(w))=d_{W}(w)\\
&=d_{Y}(G(w))=d_{Y}(y')=d_Y(y)\\
&=\abs{d_{Y}}(\abs{y})=d_Y(\beta^{-1}(z))
\end{split}
\end{equation*}
proving the  claim.

\begin{defn}\label{groupoidprop3}
{\em The function $d:Z\to \N_{0}$ defined by \eqref{dindex} is called the {\bf degeneracy index of the polyfold $Z$}.}
\end{defn}
The corner structure of a polyfold will play an important role in
the symplectic field theory in \cite{HWZ5}.  Therefore, we shall look at the  faces, introduced in \cite{HWZ2}, in
more detail.
\begin{defn}\label{groupdef1}
A {\bf connected component} $C$ of   an ep-groupoid $X$ is a full
subcategory having the following properties.
\begin{itemize}
\item[$\bullet$] If $x\in C$ and if $x'\in X$ is in
 the same component of $X$ as $x$, then $x'\in C$.
\item[$\bullet$] If $x\in C$ and $h:x\to x'$ for some
$h\in {\bf X}$, then $x'\in C$.
\item[$\bullet$]  The orbit space $|C|$ is connected.
\end{itemize}
\end{defn}
Note  that a connected component is an ep-groupoid by definition.
We denote by $X(1)$ the set of all good
boundary points,
$$X(1)=\{x\in X\vert \, d(x)=1\}.$$
The space $X(1)$  carries in a natural way the structure of an M-polyfold induced
by that of $X$. The set of morphisms  between points  in $X(1)$  is an M-polyfold  denoted by ${\bf X}(1)$. Together with
the structure maps induced from $X$  we obtain the ep-polyfold
groupoid $X(1)$. If $C$ is a connected component of this groupoid
$X(1)$, then the  {\bf closure}  $\ov{C}$ is defined as the full subcategory
whose set of objects  is the set theoretical closure of the
object set $C$. It is called a {\bf face} of the ep-groupoid $X$. In general,  a face
need not to have the structure of an ep-groupoid. Given a
ep-groupoid we can view a subset $C$ of the object set as a groupoid
by taking the full subcategory of $X$ whose objects are the points
in $C$. For $C$ as just described we can first take the set
theoretic closure $\ov{C}$ and then the associated full subcategory.

\begin{lem}\label{grouplem2}
An object $x\in X$ belongs globally to at most $d(x)$-many faces.
Locally it belongs to precisely $d(x)$ many faces (this  will be made
precise during the proof).
\end{lem}
\begin{proof}
Take an object $x\in X$ and let $U(x)$ be an open neighborhood of $x$  which is
homeomorphic by means of  a chart $(U(x),\varphi,(O,{\mathcal S}))$ to an open
set $O$  in a splicing core defined by
\begin{equation*}
\begin{split}
O:&=\{((v,f),e)\, \vert \, \, \text{$ (v,f)\in [0,\infty)^n\oplus W$ and $e\in E,$}\\
&\phantom{==\, }\pi_v(e)=e,\
 \abs{(v,f)}<1, \ \abs{e} <1\}.
\end{split}
\end{equation*}
We may assume that $x$ corresponds to $((0,f_0),e_0)$ for some
$f_0\in W$ and some $e_0\in E$. Observe that only in the case that
$x$ is smooth we may assume that  also $f_0=0$ and $e_0=0$. The set
$X(1)\cap U(x)$ corresponds to the points $((v,f),e)$ in $O$  for
which $v=(v_1,\ldots v_n)$ has precisely one coordinate  $v_j$ with
$1\leq j\leq n$ vanishing. Define  the subset $O_j\subset O$ by
$$
O_j=\{((v,f),e)\in O\vert \, \,  \text{$v_j=0$\, and $v_i\neq 0$\,
for $ i\in \{1,\ldots, j-1,j+1,\ldots, n$}\}\}
$$
and put $U_j=\varphi^{-1}(O_j)$.  The closure $O_j^\ast$ of $O_j$ in
$O$ consists of all $((v,f),e)\in O$ having the $j$-coordinate
vanishing. It corresponds under $\varphi$ to the closure $U_j^\ast$
of $U_j$ in $X(1)\cap U(x)$. If a point $((v,f),e)\in O$ satisfies
$d((v,f),e)=m\leq n$, then there exist indices  $j_1,\ldots j_m\in
\{1,\ldots , n\}$  so that $v_{j_i}=0$. Hence the point $((v,f),e)$
belongs to $O_{j_i}^\ast$ for all $i=1,\ldots ,m$. The converse  is
also true.  Finally we observe that globally,  two different (local) closed sets
$U_i^\ast$ and $U_j^\ast$ might belong to the same face. This shows
that $x$ belongs to at most $d(x)=m$ faces.
\end{proof}
In our applications the so-called face-structured ep-groupoids show up.
We introduce this notion in the next definition.
\begin{defn}\label{groupdefn2}
An ep-groupoid $X$ is called {\bf face-structured} if every object
$x\in X$ belongs to exactly $d(x)$-many faces.
\end{defn}

\begin{lem}\label{grouplem3}
Let $X$ be a face-structured ep-groupoid and $C$ a connected
component of the good boundary points $X(1)$. Then its closure $\ov{C}$ is an ep-groupoid. In
other words,  the faces are ep-groupoids.
\end{lem}
\begin{proof}
Let $F$ be a face in $X$. Take an object  $x_0\in F$ and take a chart
$(U,\varphi,(O,{\mathcal S}))$ around $x_0$. We may assume that
$$
\varphi(x_0)=((0,f),e)\in [0,\infty)^n\oplus
W\oplus E.
$$
Since $X$ is face-structured $x_0$ belongs globally to exactly
$d(x_0)$ many faces. This implies that there exists  an index $j\in
\{1,\ldots ,n\}$ so that $F\cap U$ corresponds to all points
$((v,f),e)\in O$ with $v_j=0$. We abbreviate this set by $O_j$.  We can define in an obvious way a new splicing as follows. We keep $E$, replace $V$ by an open subset
$V^\ast$ in $[0,\infty)^{n-1}\oplus W$ and
$\pi$ by the family of projections
$$
\rho_{(v_1,\ldots ,v_{j-1},v_{j+1},\ldots ,v_n,f)}=
\pi_{(v_1,\ldots,v_{j-1},0,v_{j+1},\ldots, v_n,f)}.
$$
Now  we can use the chart map $\varphi$ to define a chart $(U\cap
X(1),\varphi,(O_j,  {\mathcal S}^\ast))$ with the splicing ${\mathcal S}^*=(\rho, E, V^*)$ .  All charts obtained in
this way are sc-smoothly compatible.

Denote by ${\bf F}$ the morphism set coming from the full
subcategory associated with the face $F$. If a morphism has its
source in $F$, then its target is also in $F$ since its degeneracies
are the same.  Because $s,t:{\bf X}\rightarrow X$ are local
sc-diffeomorphisms the maps $s,t:{\bf F}\rightarrow F$ are local
homeomorphisms. The local transition maps $t\circ s^{-1}$ and
$s\circ t^{-1}$ are restrictions of sc-smooth maps. Hence we can use
them to define a M-polyfold structure on ${\bf F}$ for which
$s,t:{\bf F}\rightarrow F$ are local sc-diffeomorphisms. Finally, the groupoid
$F$ is proper. Indeed, if $x$ is an object in $F$ we take an open neighborhood $U\subset X$
of $x$ so that
$$
t:s^{-1}(\ov{U})\rightarrow X
$$
is proper. Taking $U':=U\cap F$ the result follows because  the
object set $F$ is closed in $X$.
\end{proof}

Next we study the geometry of faces in   more   detail.
\begin{lem}\label{faces0}
Let $F:X\rightarrow Y$ be an equivalence between two ep-groupoids. Assume that
$D$ is a connected component of $Y(1)$. Then there exists a unique  connected component $C$ of $X(1)$ so that for every $y\in
D$ there is a  morphism
$\varphi:F(x)\rightarrow y$ for some $x\in C$. Conversely, if $C$ is a connected component of $X(1)$, then there  exists  a unique connected component $D$ of $Y(1)$  satisfying $F(C)\subset D$.
\end{lem}
\begin{proof}
Let $D$ be a connected component of $Y(1)$ and  $y_0\in D$. Since $F:X\to Y$ is an equivalence, the map $\abs{F}:\abs{X}\to \abs{Y}$ is an homeomorphism. Hence there exists a morphism $\varphi_0:F(x_0)\to y_0$ for some $x_0\in X$. From Lemma  \ref{groupoidlem1} we deduce  $d(x_0)=d(y_0)=1$ implying that $x_0\in X(1)$. Define
the subset $C$ of the object  set  by
$$
C=\{x\in X(1)\vert \, \text{there are an $y\in D$ and a morphism $\varphi :F(x)\to y$}\}.
$$  We claim that $C$ is a connected component of $X(1)$. Take $x\in C$ and assume that  $x'\in X(1)$ belongs to a path component of $X(1)$. Then there is a continuous path $\alpha:[0,1]\to X(1)$ such that $\alpha (0)=x$ and $\alpha (1)=x'$. Consider the continuous path $\beta=F\circ \alpha$. Then $\beta (0)=F(x)$ and $\beta (1)=F(x')$. Since $x\in C$, there exists $y\in D$ and a morphism $\varphi_0:F(x)\to y$. By the definition of the connected component $D$ we have $F(x)\in D$ and  hence $F(x')\in D$. This implies that $x'\in C$. If $x\in C$ and $h:x\to x'$ is a morphism, then $F(h):F(x)\to F(x')$. From $F(x)\in D$ we conclude $F(x')\in D$, and consequently $x'\in C$. Finally, $\abs{C}:=\abs{F}^{-1}(\abs{D})$ is connected since $\abs{F}:\abs{X}\to \abs{Y}$ is a homeomorphism and $\abs{D}$ is connected.

Conversely, assume that $C$ is a connected component of $X(1)$.
Define the subset
$$D=\{y\in Y(1)\vert \, \text{there are an  $x\in C$ and a morphism $\varphi:F(x)\to y$}\}.$$ Then $F(C)\subset D$ and we claim that $D$ is a connected component of $Y(1)$. If $y\in D$ and $y'$ belongs to the path component of $Y(1)$, then there is a continuous path $\beta:[0,1]\to Y(1)$ such that $\beta (0)=y$ and $\beta (1)=y'$. By the definition of $D$, there is $x\in C$ and a morphism $\varphi:F(x)\to y$. Since $F$ is an equivalence, we find a continuous path $\alpha:[0,1]\to X$ such that $\abs{F}\circ \pi_X\circ \alpha =\pi_Y \circ \beta$.   Since $C$ is the connected component, the point $x'=\alpha (1)$ belongs to $C$ and since there is  a morphism $\varphi: F(x')\to y'$, the point $y'$ belongs to $D$.  Next assume  $y\in D$ and  there is a morphism $h:y\to y'$.  Then there are $x\in C$ and a morphism $\varphi:F(x)\to y$. Hence $h\circ \varphi:F(x)\to y'$ is a morphism showing that $y'\in D$. Also, $\abs{D}=\abs{F}(\abs{C})$  showing that $\abs{D}$ is connected.
We have proved that the  equivalence $F:X\to Y$ induces a bijective map  between connected components of $X(1)$ and connected components  of $Y(1)$ by
$$
C\leftrightsquigarrow D,
$$
where we associate with  a connected component $C$ the connected component $D$ of $Y(1)$ satisfying
$F(C)\subset D$ and  with  a  connected component $D$ the connected component $C$ consisting  of those $x\in X(1)$ for which there is a morphism between $F(x)$ and  a point in $D$.
\end{proof}

If $F:X\rightarrow Y$ is an equivalence and $C$ a connected component of $X(1)$,  we denote by $F_\ast(C)$ the component of $Y(1)$ containing $F(C)$. If $D$ is a
connected component of $Y(1)$,  we denote by $F^\ast(D)$ the component
of $X(1)$ mapped into $D$.
We denote  by $\ov{C}$ the closure of the connected component $C$ of
$X(1)$ and similarly by $\ov{D}$ the closure of the connected component  $D$ in $Y(1)$.

\begin{lem}\label{faces1}
If $C$  is a connected component of $X(1)$, then  $F(\ov{C})\subset \ov{F_\ast(C)}$.
 If $D$ is a connected component of $Y(1)$,  then the set
 $$W=\{x\in X(1)\vert \,  \text{there are an  $y\in \ov{D}$ and a morphism $\varphi: F(x)\to y$}\}$$
  is contained in $\ov{F^\ast(D)}$. In other words, $F_\ast$ and $F^\ast$ map faces to faces.
\end{lem}
\begin{proof}
Take $x\in \ov{C}$ and a sequence $(x_n)\subset C$ converging to $x$. Then
$F(x_k)\in F_\ast(C)$ and $F(x_k)\to F(x)$ implying that $F(x)\in \ov{F_\ast(C)}$.
Next take $x\in W$. Then there is $y\in\ov{D}$ and a morphism $\varphi:F(x)\to y$.
Since $F$ is a local sc-diffeomorphism,  we find open neighborhoods $U(x)\subset X$ and $U(F(x))\subset Y$  of the points $x$ and $F(x)$ such that $F:U(x)\to U(F(x))$ is an sc-diffeomorphism. The morphism $\varphi:F(x)\to y$ extends to a local sc-diffeomorphism of the form $\wh{\varphi}=t\circ s^{-1}$. Shrinking the set $U(x)$ if necessary,  we may assume that $\wh{\varphi}:U(F(x))\to  U(y)$ is an sc-diffeomorphim.  Take any sequence $(y_n)\subset D$ such that $y_n\to y$. Then $(y_n)\subset U(y)$ and we find a sequence $(x_n)\in U(x)$ such that $x_n\to x$ and $\wh{\varphi}(F(x_n))=y_n$.   By Lemma  \ref{groupoidlem1},
$d(x_n)=d(y_n)=1$  and since $\wh{\varphi}$  defines a morphism between $F(x_n)$ and $y_n$, we conclude that $x_n\in F^\ast(D)$.  This implies  $x\in \ov{F^\ast(D)}$ as claimed.
\end{proof}

We  summarize the previous discussion in the following proposition.

\begin{prop}\label{groupprop4}
 Assume that $F:X\to Y$ is an equivalence between ep-groupoids.
 Then the maps $F_\ast$ and $F^\ast$ map faces to faces
 and are mutual inverses. In particular,  if one of the ep-groupoids
 is face-structured so is the other.
 \end{prop}
 \begin{proof}
It only  remains  to  prove the last statement.  We assume that$X$ is
face-structured and prove that $Y$ is also face-structured. Take a point $y\in Y$. Then there is $x\in X$ and a morphism $\varphi:F(x)\to y$. By  Lemma
\ref{groupoidlem1}, $ d:=d(x)=d(y)$.  Since $X$ is face-structured, the point $x$ belongs to $d$-many faces. By Lemma \ref{faces1}, the map $F_{\ast}$ takes these faces onto exactly $d$-many faces of $Y$ containing the point $y$.  Conversely, assume that
$Y$ is face-structured. Take a point $x\in X$ and set $y=F(x)$. Then $d:=d(x)=d(y)$ and since $Y$ is face-structured,  the point $y$ belongs to exactly $d$ faces. Then the map $F^\ast$  takes these faces to $d$-many faces of $X$  containing the point  $x$. Hence $X$ is face-structured and the proof of the proposition is complete.
 \end{proof}

 \begin{prop}\label{groupprop5}
 Let $X$ be a face-structured ep-groupoid. Then the intersection
 of an arbitrary number of faces carries in a natural way the structure
 of an ep-groupoid.
 \end{prop}
 \begin{proof}
 The ep-structure on a finite intersection of faces can be
 defined as we did for a face  in  Lemma \ref{grouplem3}.
 \end{proof}
Next, assume that $Z$ is a polyfold and let the pairs $(X,\alpha)$ and $(X', \alpha')$ be two equivalent polyfold structures on $Z$. This means that there exist a third ep-groupoid $X''$ and two equivalences $F:X''\to X$ and $F':X''\to X'$ satisfying  $\alpha \circ \abs{F}=\alpha'\circ \abs{F'}$. If $D$ is a connected component of $X (1)$, then by
Lemma \ref{faces1},
$F^{\ast}(D)$ is the associated  connected component  of $X''(1)$ and $D'=F'_{\ast}\circ F^{\ast} (D)$ the connected component of $X'(1)$. Hence there  is a one--one correspondence between the faces of $X(1)$ and $X'(1)$. In particular, if the ep-groupoid  $X$ is face-structured, then the same is true for the ep-groupoid $X'$.  If $C$ is a face of $X(1)$ and $C'=F'_{\ast}\circ F^{\ast}(C)$ is the corresponding face of $X'(1)$, then $\abs{F}^{-1}(\abs{C})=\abs{F'}^{-1}(\abs{C'})$ and
$$\alpha'(\abs{C'})=\alpha'\circ \abs{F'}\circ \abs{F}^{-1}(\abs{C})=\alpha \circ \abs{F}\circ \abs{F}^{-1}(\abs{C})=\alpha (\abs{C}).$$

This allows to introduce the concepts  of a  face and of  face-structured for polyfolds.
\begin{defn}
Let $Z$ be a polyfold and let $(X, \alpha)$ be a polyfold structure on $Z$.  The polyfold $Z$ is said to be {\bf face-structured} if the ep-groupoid $X$ is face-structured. A {\bf  face}  $D$ of  the polyfold $Z$ is the image $D=\alpha (\abs{C})$ of
the orbit space of a face $C$ in $X$.
\end{defn}

\subsection{Branched Suborbifolds}
Next we shall introduce the notion of a branched suborbifold of a polyfold $Z$ and start with the definition of a branched ep-subgroupoid of the ep-groupoid $X$. It generalizes ideas from \cite{CRS} where quotients of manifolds by global group-actions are studied.

We shall view the nonnegative rational numbers, denoted by $\Q^+=\Q\cap [0,\infty )$,
as the objects in a category having only the identities
as morphisms. We would like to mention  that branched
ep-subgroupoids will show up as solution sets of polyfold Fredholm
sections.

\begin{defn}\label{def1}  A {\bf  branched ep-subgroupoid of the ep-groupoid $X$}  is  a
functor
 $$\Theta:X\rightarrow {\mathbb Q}^+$$  having the following properties.
\begin{itemize}
\item[(1)] The support  of $\Theta$, defined by $\supp \Theta=\{x\in X\vert \, \Theta(x)>0\}$,  is
contained in $X_\infty$.
\item[(2)] Every  point $x\in \supp \Theta$ is contained in an open
neighborhood $U(x)=U\subset X$ such that
$$\supp \Theta \cap U= \bigcup_{i \in I}M_i,$$
where $I$ is a finite index set and where the sets $M_i$  are finite dimensional submanifolds of $X$ (in the sense of Definition 4.19 in \cite{HWZ3}, which is recalled in the Appendix \ref{sub}) all having the same dimension, and all in good position to the boundary $\partial X$ in the sense of Definition 4.14 in \cite{HWZ3}. The submanifolds $M_i$ are called {\bf local branches} in $U$.
\item[(3)] There exist  positive rational numbers $\sigma_i$, $i\in I$, (called {\bf weights}) such that if $y\in \supp \Theta \cap U$, then
$$
\Theta(y)=\sum_{\{i \in I \vert   y\in M_i\}} \sigma_i.
$$
\item[(4)] The inclusion maps $M_i\to U$ are proper.
\item[(5)] There is a natural representation of the isotropy group ${\bf X}(x)$ acting by sc-diffeomorphisms on $ U$.
\end{itemize}
\end{defn}

The branches  $(M_i)_{i\in I}$ together with the weights
$(\sigma_i)_{i\in I}$ constitute  a {\bf local branching structure}
of $X$ in $U$. The role of local branching structures and their properties have been studied in detail in \cite{HWZ7} where also the notion of an orientation for a branched ep-subgroupoid as well as for a branched suborbifold is introduced.

\begin{remark}
Given a branched ep-subgroupoid $\Theta$ of the ep-groupoid $X$  we obtain the
induced ep-subgroupoids $\Theta^i:X^i\rightarrow {\mathbb Q}^+$ for $i\geq 1$. Observe that the supports of $\Theta=\Theta^0$ and $\Theta^i$ coincide. The reader should notice that all upcoming constructions carried out for $\Theta$ will lead to the same result if carried out for $\Theta^i$.
\end{remark}

\begin{defn}
A branched ep-subgroupoid is called  {\bf $\mathbf{n}$-dimensional} if all
its local branches are n-dimensional submanifolds of $ M$. It is called {\bf compact} if $|\supp \Theta|$ is compact.
\end{defn}
From the assumption that $\Theta:X\to \Q^+$  is a functor it follows that
$$\Theta (x)=\Theta (y)$$
if there is a morphism $\varphi:x\mapsto y$. Therefore a branched ep-subgroupoid $\Theta$ induces a canonical map $\abs{\Theta}:\abs{X}\to \Q^+$ defined as
$$\abs{\Theta} (\abs{x})=\Theta (x).$$

Next we recall the definition of an orientation from \cite{HWZ7}.
\begin{defn}
{\em Let $\Theta:X\to \Q^+$ be a branched ep-subgroupoid on  the ep-groupoid $X$ and $S$ its support.
An {\bf orientation  for $\Theta$},   denoted by $\mathfrak{o}$,   consists of an
orientation for every local branch of the tangent set
$T_xS$ at every point  $x\in S$ so that the
following compatibility conditions  are satisfied.
\begin{itemize}
\item[(1)]
At  every point $x$,  there
exists a local branching structure $(M_i)_{i\in I}$ where the finite dimensional submanifolds  $M_i$ of $X$
can be oriented in such a way  that the orientations of $T_xM_i$ induced from $M_i$ agree with the given ones for the local branches of the tangent set $T_xS$.
\item[(2)] If  $\varphi:x\rightarrow y$ is a morphism between two points in $S_{\Theta}$, then the tangent map
 $T\varphi:T_xS\rightarrow T_yS$ maps every oriented branch to
 an oriented branch preserving the
orientation.
\end{itemize}}
\end{defn}

\begin{defn}  A {\bf branched suborbifold $S$ of the polyfold $Z$ } is  a subset $S\subset Z$ contained in $Z_{\infty}$ equipped with a positive function $\vartheta:S\to \Q\cap (0,\infty )$ called {\bf weight function}, together with an equivalence class of triples $(X, \alpha ,\Theta)$ in which the pair $(X, \alpha)$ is  a polyfold structure on $Z$ and $\Theta:X\rightarrow {\mathbb Q}^+$  is a branched ep-subgroupoid  of the ep-groupoid $X$ satisfying
\begin{Myitemize}
\item $S=\alpha (\abs{\supp \ \Theta} )$
\item $\vartheta (\alpha (\abs{x}))=\Theta (x)$.
\item Define $\supp (\vartheta):=\{z\in Z\vert \, \vartheta (z)>0\}$.  If the subsets $\supp(\vartheta)\subset Z$  and  $\abs{\supp(\Theta)}\subset \abs{X}$ are  equipped with the induced topologies, then  the map
$\alpha :\abs{\supp(\Theta)}\rightarrow \supp(\vartheta)$
is a homeomorphism.
\end{Myitemize}
\end{defn}

Here two such triples $(X,\alpha ,\Theta)$ and $(X',\alpha', \Theta')$  are called {\bf equivalent} if there exists a third triple $(X'', \alpha'', \Theta'')$ and two equivalences
$X\xleftarrow{F}X''\xrightarrow{F'}X'$ satisfying $\alpha''=\alpha \circ \abs{F}=\alpha'\circ \abs{F'}$ for the induced maps $\abs{F}$ and $\abs{F'}$ on the orbit spaces, and in addition,
$$\Theta''=\Theta'\circ F'=\Theta \circ F.$$

We note that if $\mathfrak{a}:X\Rightarrow X'$ is the s-isomorphism of the equivalence class $[X\xleftarrow{F}X''\xrightarrow{F'}X']$,  then $\alpha'\circ \abs{\mathfrak{a}}=\alpha$ and we obtain (without the last requirement above) the following relation between the two branched ep-subgroupoid  $\Theta:X\to \Q^+$ and $\Theta':X'\to \Q^+$,
$$\abs{\Theta'}\circ \abs{\mathfrak{a}}=\abs{\Theta}.$$
To verify the formula we simply compute $\abs{\Theta}(\abs{x})=\Theta (x)=\vartheta (\alpha (\abs{x}))=\vartheta (\alpha'\circ \abs{\mathfrak{a}}(\abs{x}))=\abs{\Theta'}\circ \abs{\mathfrak{a}}(\abs{x}).$

The  {\bf branched suborbifold $S$ of the polyfold $Z$} is called {\bf $n$-dimensio\-nal},  resp. {\bf compact}, resp. {\bf oriented}, if the representative $(X,\alpha, \Theta)$ of the ep-groupoid is
$n$-dimensional, resp. compact, resp. oriented. In the latter case the
equivalences $F$ and $F'$ involved are required to preserve the orientations.

\subsection{Branched Integration}\label{branchedint}
There is a canonical measure on an oriented branched suborbifold of a polyfold and an associated integration theory which we would like to describe briefly in this section referring to \cite{HWZ7} for details and proofs.

The measure will be induced from the overhead $(X, \alpha ,\Theta)$ where $\Theta:X\to \Q^+$ is a branched ep-subgroupooid of the ep-groupoid $X$   of dimension $n$. We  assume that the orbit space of $\Theta$, namely $S=\abs{\supp \Theta}$,  is compact. We abbreviate the orbit space of the boundary  by
$$\partial S=\abs{\supp \Theta
\cap \partial X}=\{ |x|\in S\vert \, x\in \supp \Theta \cap \partial X\}.
$$
The orbit space $\partial
S$, in general, is  not necessarily  the boundary of $S$ in the sense of
topology. Locally the sets $S$ and $\partial S$ can be
represented as quotients of finite unions of
smooth submanifolds by a sc-smooth group action on the ambient space. This structure suffices to generalize the familiar
notion of the Lebesgue $\sigma$-algebra of subsets of smooth
manifolds as follows.
\begin{thm}[{\bf Canonical $\sigma$-Algebra}]\label{th0}
 The compact topological spaces $S$ and $\partial S$ possess canonical $\sigma$-algebras ${\mathcal
L}(S)$ and ${\mathcal L}(\partial S)$ of subsets containing the
Borel $\sigma$-algebras of $S$ and $\partial S$, respectively.
\end{thm}
The sets belonging to ${\mathcal L}(S)$ and ${\mathcal L}(\partial S)$
are  called measurable. The precise definition of the $\sigma$-algebras is given in \cite{HWZ7}.  The canonical $\sigma$-algebras for the branched ep-subgroupoid  $\Theta$
and  the induced ep-subgroupoids $\Theta^i$ are identical.

If  $X$ is an ep-groupoid,
then the tangent space  $TX$ is again an ep-groupoid. Recall that $TX$
is defined by equivalence classes $[U,\varphi,(O,{\mathcal
S}),x,h]$, where $x\in X^1$, $\varphi:U\rightarrow O$ is a chart
where $O$ is open in the splicing core $K^{\mathcal S}$ and $h\in
T_xX$. The map $TX\rightarrow X^1$ is given by
$[U,\varphi,(O,{\mathcal S}),x,h]\rightarrow x$.
If  $\psi:x\rightarrow y$ is a morphism
between points $x,y\in X_1$, then $\psi$ belongs to ${\bf X}_1$  and extends to an sc-diffeomorphism  $t\circ s^{-1}:O(x)\rightarrow O(y)$. We define $T\psi:T_xX\rightarrow
T_yX$ to be the linearization of $t\circ s^{-1}$ at $x$.
Hence $TX$ is an ep-groupoid whose  structure maps are the
tangents of the structure maps of $X$. If $F:(X',\beta')\rightarrow
(X,\beta)$ is an equivalence of polyfold structures, then
$TF:TX'\rightarrow TX$ is an equivalence. It also defines an
equivalence between the Whitney sums $\oplus_k TX'\rightarrow
\oplus_k TX$ covering $F:X\rightarrow X'$.
In  the following  we  denote by
$\oplus_kTX\rightarrow X^1$  the Whitney sum of $k$-many copies of
$TX$.

At this point we  would like to make some comments about our  notation.
For us the tangent bundle of $X$ is $TX\rightarrow X^1$,  that is, it is  only defined for the base points in $X^1$. An   {\bf sc-vector field} on $X$  is an  sc-smooth section of  the tangent bundle $TX\rightarrow X^1$ and hence it is  defined on $X^1$. Similarly, an  sc-differential form  on $X$ which we will define next,  is  only defined over the base points in $X^1$. The definition of a vector field and the following  definition of an sc-differential form
are justified since the construction of $TX$, though only defined over $X^1$, requires the knowledge of $X$.
\begin{defn}
An {\bf $\boldsymbol{\ssc}$-differential ${\boldsymbol{k}}$-form on the ep-groupoid $X$}  is
an sc-smooth map $\omega:\oplus_k TX \to \R$ which is
linear in each argument separately, and skew symmetric.   In addition,  we require that
$$(T\varphi)^\ast\omega_y=\omega_x$$
for all morphisms   $\varphi:x\rightarrow y$ in ${\bf X}_1.$
\end{defn}

If $\omega$ is an  sc-differential form on $X$,  we may also  view it  as an  sc-differential form on $X^i$.  We denote by $\Omega^\ast(X^i)$ the graded commutative algebra of sc-differential forms on $X^i$. Then we have the inclusion map
$$
\Omega^\ast(X^i)\rightarrow\Omega^\ast(X^{i+1}).
$$
which  is injective since $X_{i+1}$ is dense in $X_i$ and the forms are sc-smooth. Hence we have a directed system whose  direct limit is denoted by
$\Omega^\ast_\infty(X)$. An element $\omega$ of  degree $k$ in $\Omega^\ast_\infty(X)$ is a skew-symmetric map $\oplus_k(TX)_\infty\rightarrow {\mathbb R}$ such that it has an  sc-smooth extension to an  sc-smooth k-form
$\oplus_kTX^i\rightarrow {\mathbb R}$ for some $i$. We shall refer to an element of $\Omega^k_\infty(X)$ as an  {\bf sc-smooth differential form}  on $X_\infty$. We note, however, that it is part of the structure that the $k$-form is defined and sc-smooth on some $X^i$.

Next we associate with an sc-differential $k$-form $\omega$ its exterior differential
$d\omega$ which is a $(k+1)$-form on the ep-groupoid $X^1$. The compatibility with morphisms will be essentially an automatic consequence of the definition. Hence it suffices for the moment to consider the case that $X$ is an M-polyfold.
Let $A_0, \ldots , A_k$ be $k+1$ many sc-smooth vector fields on $X$. We define
$d\omega$ on $X^1$, using the familiar  formula,   by
\begin{equation*}
\begin{split}
d\omega (A_0, \ldots ,A_k)&=\sum_{i=0}^k(-1)^iD(\omega (A_0, \ldots , \wh{A}_i, \ldots , A_k))\cdot A_i\\
&+\sum_{i<j}(-1)^{(i+j)}\omega ([A_i, A_j], A_0, \ldots , \wh{A}_i,\ldots ,\wh{A}_j, \ldots ,A_k).
\end{split}
\end{equation*}

The right-hand side of the formula above only makes sense at the base points $x\in X_2$. This  explains why  $d\omega$ is a $(k+1)$-form on $X^1$.
By the previous discussion the differential $d$ defines a map
$$
d:\Omega^k(X^i)\rightarrow \Omega^{k+1}(X^{i+1})
$$
and consequently induces a map
$$
d:\Omega^\ast_\infty(X)\rightarrow \Omega^{\ast+1}_\infty(X)
$$
having the usual property $d^2=0$. Then $(\Omega^\ast_\infty(X),d)$ is a  graded differential algebra which we shall call the de Rham complex.

If $\varphi:M\to X$ is an sc-smooth map from a finite-dimensional manifold $M$ into an M-polyfold $X$, then it induces an  algebra homomorphism
 $$
 \varphi^\ast:\Omega^\ast_\infty(X)\rightarrow \Omega_\infty^\ast(M)
 $$
 satisfying
$$d (\varphi^*\omega )=\varphi^*d\omega.$$

To formulate the  next theorem we recall the natural representation $\varphi:G_x\to \text{Diff}_{\ssc}(U)$ of the isotropy group $G_x=G$. This group can  contain a subgroup  acting trivially on $U$. Such a subgroup is a normal subgroup of $G$ called the {\bf ineffective part} of $G$ and denoted by $G_0$. The  {\bf effective part} of $G$ is the quotient group
$$G_{\text{e}}:=G/G_0.$$
We  denote  the order of $G_{\text{e}}$ by $\sharp G_{\text{e}}$.

\begin{thm}[{\bf Canonical Measures}]\label{th1}
Let $X$ be an ep-groupoid and assume that $\Theta:X\to \Q^+$ is an oriented branched
ep-subgroupoid of dimension $n$ whose orbit space $S=\abs{\supp \Theta}$  is compact and equipped with the weight function
$\vartheta:S\to \Q^+$ defined by
$$
\vartheta (|x|):=\Theta (x),\quad |x|\in S.
$$
Then there exists a map
$$
\Phi_{(S,\vartheta)}:\Omega^n_\infty(X)\rightarrow {\mathcal M}(S,{\mathcal
L}(S)),\quad \omega \mapsto    \mu_\omega^{(S,\theta)}
$$
which  associates to every  sc-differential $n$-form $\omega$ on $X_\infty$ a signed finite measure
$$\mu_{\omega}^{(S, \vartheta )}\equiv \mu_{\omega}$$
on  the canonical measure space $(S,{\mathcal L}(S))$. This map is uniquely characterized by the
following properties.
\begin{itemize}
\item[(1)] The map $\Phi_{(S,\theta)}$ is linear.
\item[(2)] If $\alpha=f\tau$ where $f\in\Omega^0_\infty(X)$ and $\tau\in\Omega^{n}_\infty(X)$, then
$$
\mu_\alpha(K)=\int_K
fd\mu_\tau
$$
for  every set $K\subset S$ in the $\sigma$-algebra ${\mathcal L}(S)$.
\item[(3)] Given a  point $x\in \supp \Theta$ and an oriented branching structure
${(M_i)}_{i\in I}$  with  the associated weights $(\sigma_i)_{i\in I}$ on the open neighborhood  $U$ of $x$ according to Definition \ref{def1}, then  for every set  $K\in {\mathcal L}(S)$ contained in a compact subset of $\abs{\supp \Theta \cap U}$,  the $\mu_{\omega}$-measure of $K$ is given by the formula
$$
\mu_\omega (K)=\frac{1}{\sharp G_{e}}\sum_{i\in I}
\sigma_i\int_{K_i}\omega\vert M_i
$$
where $K_i\subset M_i$ is the preimage of $K$ under  the projection map
$M_i\to \abs{\supp \Theta \cap U}$ defined by $x\rightarrow |x|$.
\end{itemize}
\end{thm}

In the theorem above we denoted  by $\int_{K_i}\omega\vert M_i$ the signed measure of the set
$K_i$ with respect to the Lebesgue signed measure associated with the smooth $n$-form $\omega\vert M_i$ on the finite dimensional manifold $M_i$. It is given by
$$\int_{K_i}\omega\vert M_i=\lim_k\int_{U_k}j^*\omega$$
where $j:M_i\to X$  is the inclusion map and $(U_k)$ is  a decreasing sequence of open neighborhoods of $K_i$ in $M_i$ satisfying $\bigcap_k U_k=K_i$.

The analogous result holds true for the orbit space $\partial S=\{|x|\in S\vert \,  x\in \supp \Theta\cap \partial X\}$ of the boundary.

\begin{thm}[ {\bf Canonical Boundary  Measures}] Under the same assumptions
as in Theorem \ref{th1} there exists a map
$$
\Phi_{(\partial S,\vartheta)}:\Omega^{n-1}_\infty(X)\rightarrow {\mathcal
M}(\partial S,{\mathcal L}(\partial S)), \quad \tau\mapsto
\mu_\tau^{(\partial S,\vartheta)}
$$
which assigns to every sc-differential $(n-1)$-form $\tau$ on $X_\infty$ a signed finite measure
$$\mu_{\tau}^{(\partial S, \vartheta)}\equiv \mu_{\tau}$$
on  the canonical measure  space
$(\partial  S, {\mathcal L} (\partial S))$.
 This map is uniquely characterized by the
following properties.
\begin{itemize}
\item[(1)] The map $\Phi_{(\partial S,\vartheta )} $ is linear.
\item[(2)] If $\alpha=f\tau$ where  $f\in\Omega^0_\infty(X)$ and $\tau\in\Omega^{n-1}_\infty(X)$, then every $K\in {\mathcal L}(\partial S)$
has the $\mu_{\alpha}$-measure
$$\mu_\alpha(K)=\int_K
fd\mu_\tau.$$
\item[(3)] Given a  point $x\in \supp \Theta \cap \partial X$  and an oriented branching structure  ${(M_i)}_{i\in I}$ with weights $(\sigma_i)_{i\in I}$ on the open neighborhood $U\subset X$ of $x$, then the measure of  $K\in {\mathcal L}(\partial S)$  contained in a compact subset of $\abs{\supp \Theta \cap U\cap \partial X}$ is given by the formula
$$
\mu_\tau(K)=\frac{1}{\sharp G_{e}}\sum_{i\in I}
\sigma_i\int_{K_i}\tau \vert \partial M_i
$$
where $K_i\subset  \partial M_i$ is the preimage of $K$ under the projection map  $\partial M_i\to \abs{\supp \Theta \cap U\cap \partial X}$ defined by $x\mapsto |x|$.
\end{itemize}
\end{thm}
Finally,  the following   version of Stokes' theorem holds.

\begin{thm}[{\bf Stokes Theorem}]\label{thmst0}
Let $X$ be   an ep-groupoid and let  $\Theta:X\to \Q^+$ be an oriented $n$-dimensional
branched ep-subgroupoid of $X$  whose orbit space  $S=\abs{\supp \Theta}$ is compact. Then,
for every  sc-differential $(n-1)$-form $\omega$ on $X_\infty$,
$$
\mu^{(S,\vartheta)}_{d\omega}(S)=\mu^{(\partial S,\vartheta)}_{\omega}(\partial
S).
$$
\end{thm}

Our  construction is compatible with
equivalences between ep-groupoids as the following  theorem shows.

\begin{thm}[{\bf Equivalences}]\label{th4}
Assume that  $F:X\to  Y$ is an equivalence between the
ep-groupoids $X$ and $Y$.  Assume that  $\Theta:Y\rightarrow {\mathbb Q}^+$ is  an
oriented $n$-dimensional branched ep-subgroupoid of $Y$  whose orbit
space $S=\abs{\supp \Theta}$ is compact and equipped with the weight
function $\vartheta:S\to \Q^+$ defined by $\vartheta (|y|)=\Theta
(y)$ for $|y|\in S$. Define the $n$-dimensional branched
ep-subgroupoid  on $X$ by $\Theta':=\Theta\circ F:X\to \Q^+$ and denote by $S'$ and $\vartheta'$  the associated orbit space   and the  weight function on $S'$. Moreover, assume that
$\Theta'$ is equipped with the induced orientation. Then, for every
sc-differential $n$-form $\omega$ on $Y_\infty$,
$$
\mu_{\omega}^{(S,\vartheta)}\circ \abs{F} = \mu_{\omega'}^{(S',\vartheta')}
$$
where the $n$-form $\omega'$ on $X_\infty$ is the pull back form $\omega'=F^*\omega$.
Similarly,
$$
\mu_{\tau}^{(\partial S,\vartheta)}\circ \abs{F}=
\mu_{\tau'}^{(\partial S',\vartheta')}
$$
for every $(n-1)$-form $\tau$ on $Y_\infty$.
\end{thm}

Theorem \ref{th4}  allows  to
rephrase the previous theorems  in the  polyfold set-up.

An sc-differential form on the topological space $Z$ or $Z^i$ is defined
via the overhead of the postulated polyfold structures. We define an
sc-differential form $\tau$ on the polyfold $Z$ as an equivalence
class of triples  $(X, \beta , \omega)$ in which the pair $(X,
\beta)$ is a polyfold structure on $Z$ and $\omega$ an
sc-differential form on $X_\infty$. Two such  triples $(X, \beta , \omega)$
and $(X', \beta' , \omega')$ are called equivalent, if there is a
third ep-groupoid $X''$ and two equivalences $X\xleftarrow{F}
X''\xrightarrow{F'} X'$ satisfying $\beta\circ \abs{F} = \beta'\circ
\abs{F'}$ and, in addition,
$$
F^*\omega =(F')^*\omega'
$$ for the pull back forms on $X''$.
We shall abbreviate an equivalence class by $\tau =[\omega]$ where $(X, \beta ,\omega)$ is a representative of the class and call $\tau$ an sc-differential form on the polyfold $Z$. Again we have a directed system obtained from the inclusions $Z^{i+1}\rightarrow Z^i$ which allows us to define, analogously to the ep-groupoid-case, the notion of a differential form on $Z_\infty$. We shall  still use the symbol $[\omega]$ for such forms.  It is important to keep in mind that a representative $\omega$ of $[\omega]$ is an  sc-differential form on $X_\infty$ where $(X,\beta)$ is a polyfold structure on  $Z$.

The exterior derivative of an sc-differential form $[\omega ]$ is defined by
$$
d[\omega] =[d\omega].
$$

Let $S$ be a   branched suborbifold $S$ of a polyfold $Z$ equipped with a weight
 function $w:S\to \Q^+\cap (0, \infty )$ together
 with an equivalence class of triples
 $(X, \beta ,\Theta)$ in which the pair $(X, \beta)$
 is a polyfold structure on $Z$ and $\Theta:X\to \Q^+$
 is an ep-subgroupoid  of $X$ satisfying $S=\beta (\abs{\supp \Theta })$
 and $w (\beta ( \abs{x}))=\Theta (x)$ for $x\in \supp\Theta$.
 The ``boundary'' set $\partial S$ of a
branched suborbifold $S$ by setting
$\partial S=\beta (\abs{\supp \Theta \cap \partial X})$
 for a representative $(X, \beta , \Theta)$ of the equivalence class.
 From the  previous results it follows
that for a compact and oriented branched suborbifold  $S\subset Z$ of a polyfold $Z$ there is a canonical
$\sigma$-algebra $(S, {\mathcal L}(S))$ of measurable subsets and a
well defined integration theory for which Stokes' theorem holds.

\begin{thm}\label{thmst1}
Let $Z$ be a polyfold and $S\subset Z$ be an oriented compact branched suborbifold
defined by the equivalence class $[(X, \beta ,\Theta)]$ and equipped with the weight function $w:S\to \Q^+\cap (0,\infty )$.
For an sc-differential $n$-form $\tau$ on $Z_\infty$ and $K\in {\mathcal L}(S)$,  we  define
$$
\int_{(K,w)}\tau:=\int_{\beta^{-1}(K)}
d\mu_\omega^{(\beta^{-1}(S),\vartheta)}=\mu_\omega^{(\beta^{-1}(S),\vartheta)}(\beta^{-1}(K)),
$$
where  the equivalence class $\tau$ is represented by the triple $(X, \beta ,\omega)$ and where the weight function $\vartheta$  on
$\beta^{-1}(S)=\abs{\supp \Theta}$ is defined by $\vartheta (|x|)=\Theta (x).$
Then    the integral  $\int_{(K,w)}\tau$ is independent of the representative $(X,\beta, \omega)$ in the equivalence class.
Moreover,   if $\tau$ is an sc-differential $(n-1)$-form on $Z_\infty$, then
$$
\int_{(\partial S,w)}\tau =\int_{(S,w)} d\tau.
$$
\end{thm}
For all the details and the proofs of the results of Section \ref{branchedint} we refer to \cite{HWZ7}.

\subsection{Strong Polyfold Bundles}
In the following we assume that $p:W\to Z$ is a continuous  and surjective map between two
paracompact second countable spaces.
\begin{defn}
{\em A  {\bf strong (polyfold) bundle structure}  for $p$ consists of a
triple  $(E,\Gamma,\gamma)$ in which $E\xrightarrow{P} X$ is a strong
bundle over the ep-groupoid $X$, $\Gamma$ is a homeomorphism between
the orbit space $\abs{E}$ and $W$,  and $\gamma$ is a homeomorphism
from the orbit space $\abs{X}$ to $Z$. Further, we require that
$$
p\circ\Gamma=\gamma\circ \abs{P}.
$$

\begin{equation*}
\begin{CD}
\abs{E}@>\abs{P}>>\abs{X}\\
@V\Gamma VV     @VV\gamma V \\
W@>p>>Z.\\
\end{CD}
\end{equation*}
\mbox{}\\[4pt]

}
\end{defn}

The two triples  $(E,\Gamma,\gamma)$ and $(E',\Gamma',\gamma')$ are called {\bf equivalent} if
there exists an  s-bundle isomorphism $\mathfrak{A}:E\Rightarrow
E'$ with the underlying s-isomorphism $\mathfrak{a}:X\Rightarrow X'$  satisfying on the orbit spaces
$$
\Gamma=\Gamma'\circ \abs{\mathfrak{A}}\quad \text{and}\quad \gamma =\gamma'\circ \abs{\mathfrak{a}}.
$$

We note that the second relation is a consequence of the first relation.

\mbox{}
$$
\begindc{\commdiag}[3]
\obj (20,20){$\abs{E}$}
\obj (50,20){$\abs{E'}$}
\obj (35,0){$W$}
\mor (20,20)(50,20){$\abs{\mathfrak{A}}$}[\atleft, \solidarrow]
\mor (20,20)(35,0){$\Gamma$}[\atright, \solidarrow]
\mor (50,20)(35,0){$\Gamma'$}[\atleft, \solidarrow]

\enddc
\mbox{}.
$$

We are in the position to define
the notion of a strong polyfold bundle.
\begin{defn}
A  {\bf strong polyfold bundle}  $p:W\rightarrow Z$ consists of a
surjective continuous map between two second countable paracompact
topological spaces together with  an equivalence class of strong bundle
structures  $(E,\Gamma,\gamma)$ for $p$.
\end{defn}

Next  we introduce the notion of an sc-smooth section of the strong polyfold bundle $p:W\to Z$. In  order to do so, we take a model $(E, \Gamma,\gamma)$ where $P:E\rightarrow X$ is a strong bundle
over the ep-groupoid $X$ and consider a pair  $(f,F)$ in which $f$ is a continuous section of the bundle
$p$ and  $F\in\Gamma(P)$ an sc-smooth section of the bundle $P$ satisfying
$$f\circ\gamma=\Gamma\circ |F|,$$
in diagrams,

\mbox{}
$$
\begindc{\commdiag}[3]
\obj (20,20){$\abs{X}$}
\obj (40,20){$Z$}
\obj (20,0){$\abs{E}$}
\obj (40,0){$W$}
\mor (20,20)(40,20){$\gamma$}[\atleft, \solidarrow]
\mor (20,0)(40,0){$\Gamma$}[\atright, \solidarrow]
\mor (20,20)(20,0){$\abs{F}$}[\atright, \solidarrow]
\mor (40,20)(40,0){$f$}[\atleft, \solidarrow]
\enddc
\mbox{}.
$$

Let  $(E', \Gamma',\gamma')$  be an equivalent strong bundle structure for $p$ where  $P':E'\to X'$  is a strong bundle over the ep-groupoid $X'$.  We  consider the pair $(f',F')$  of sections of the bundle $p$ and $F'\in \Gamma (P')$ satisfying $f'\circ \gamma'=\Gamma' \circ \abs{F'}$ and  call  the two pairs $(f, F)$ and $(f', F')$ {\bf equivalent} if $f'=f$ and if there exists an s-bundle isomorphism $\mathfrak{A}:E\Rightarrow E'$ whose push-forward satisfies
$$\mathfrak{A}_*(F)=F'.$$

\begin{defn}
With the above notation  an equivalence class $f=[f,F]$ of sections  is called an
{\bf sc-smooth section of the polyfold bundle}  $p:W\rightarrow Z$ and the pair
$(f,F)$ is called a  {\bf representative} of the map $f$ for the model
$P:E\rightarrow X$. The $\ssc^+$-sections  and the
Fredholm sections are defined  similarly.
\end{defn}
In view of  Proposition \ref{rty}, these concepts are well defined. We will denote
the space of sc-smooth sections of the bundle $p:W\rightarrow Z$ by $\Gamma(p)$
and the corresponding $\ssc^+$-sections by $\Gamma^+(p)$.

The auxiliary norm of a strong polyfold bundle is defined as follows.

\begin{defn}
An {\bf auxiliary norm}  for the strong polyfold bundle $p:W\rightarrow Z$ consists of a map
$N:W_{0,1}\rightarrow [0,\infty)$ having the following property. If
the strong bundle $P:E\rightarrow X$ over the ep-groupoid $X$ is a model
representing the strong bundle structure for $p$,  then
$$
N^\ast:E_{0,1}\rightarrow [0,\infty), \quad e\mapsto  N(\Gamma(\abs{e}))
$$
is an auxiliary norm on the bundle $P:E\to X$. The subset $W_{0,1}\subset W$ is defined by
$W_{0,1}=\Gamma (\abs{E_{0,1}}).$
\end{defn}
Let us observe that in general $N$ is fiber-wise not a norm since in the fibers do not have a linear structure. Nevertheless, a
local representative $N^\ast$ is an auxiliary norm having  the
additional property that the existence of a morphism
$\varphi:h\rightarrow h'$ implies $N^*(h)=N^*(h')$. Alternatively,  given an
auxiliary norm $N^\ast$ for the local model $E\rightarrow X$
satisfying $N^\ast(h)=N^\ast(k)$ if there exists a morphism
$h\rightarrow k$, then  we  can define $N(w)$ for $w\in W_{0,1}$ by
$N(w)=N^\ast(h)$ if $\Gamma(\abs{h})=w$. This defines an auxiliary
norm for the strong polyfold bundle $p$.

Next we introduce the notion of mixed convergence in $W$.
\begin{defn}
A sequence $(w_k)$ in $W_{0,1}$ is  called {\bf mixed
convergent} to $w$ if there exists a local model $(E,\Gamma)$ so
that the sequence $(|h_k|):=(\Gamma^{-1}(w_k))\subset |E|$ has
suitable representatives, say $h_k$ of $|h_k|$ and $h$ of $|h|$ so
that $h_k$ is m-convergent to $h$ on $E_{0,1}$. Let us note that the particular
choice of local coordinates in the definition is irrelevant.
\end{defn}
For the the definition of the  m-convergence we refer to \cite{HWZ3}.

\subsection{Sc$^+$-Multisections}\label{sect7.9}

In this section we define $\ssc^+$-multisections. Again we view  the
non-negative rational numbers ${\mathbb Q}^+=\Q\cap [0,\infty )$ as a category with  the identities as
morphisms. The following definition brings a definition in
\cite{CRS} into the groupoid framework.

\begin{defn}\label{defn7.9.2}
Let $P:E\to X$ be a strong bundle over the ep-groupoid $X$. Then an
 $\mathbf{sc}^{+}$-$\mathbf{multisection}$ of $P$  is a functor
$$\Lambda:E\to \Q^+$$
such that the following local representation (called {\bf local
section structure})  holds true. Every object $x\in X$ possesses
 an open neighborhood $U_x$ on which the isotropy group $G_x$ acts by its natural representation $\varphi_g \in \text{Diff}_{\text{sc}}(U_x)$ and finitely many
$\ssc^+$-sections $s_1,\ldots ,s_k:U_x\to E$ (called {\bf local
sections}) with associated positive rational numbers
$\sigma_1,\sigma_2,\ldots ,\sigma_k$ (called {\bf weights})
satisfying the following properties.
\begin{itemize}
\item[(1)] $$\sum_{j=1}^k\sigma_j=1.$$
\item[(2)] $$ \Lambda (e)=\sum_{\{j\in I\vert
s_j(P(e))=e\}}\sigma_j$$
\end{itemize}
for all $e\in E\vert U$ for which there exists a section $s_j$ satisfying $s_j (P(e))=e$. If there is no such section, then  $\Lambda (e)=0$.
\end{defn}

The functoriality of $\Lambda$ implies $ \Lambda ( e')=\Lambda ( e), $
if there is a morphism $e'\to e$ in ${\bf E}$. Explicitly,
$$\Lambda (\mu (\varphi , e))=\Lambda (e)$$
for all $(\varphi, e)\in     {\bf X}{{_s}\times_{p}}E.$ Hence $\Lambda$
induces the  map $\abs{\Lambda}:|E|\to \Q^+$ on the orbit space. We shall denote the
collection of all $\text{sc}^{+}$-multisections of  the strong bundle $P:E\to X$ by
$\Gamma^+_m(P)$. For every $x\in X$, the set
$$
\supp (x):=\{e\in E\ |\ P(e)=x,\ \Lambda (e)>0\}
$$
is finite, and $ \sum_{e\in \supp(x)}\Lambda (e)=1$. Moreover,
if $x\in U$, then
$$
\supp(x)=\{s_1(x), \ldots ,s_k (x)\},
$$
where $s_1,\ldots ,s_k:U\to E$ are the  local $\ssc^+$-sections.

\begin{figure}[htbp]
\mbox{}\\[2ex]
\centerline{\relabelbox \epsfxsize 3.8truein \epsfbox{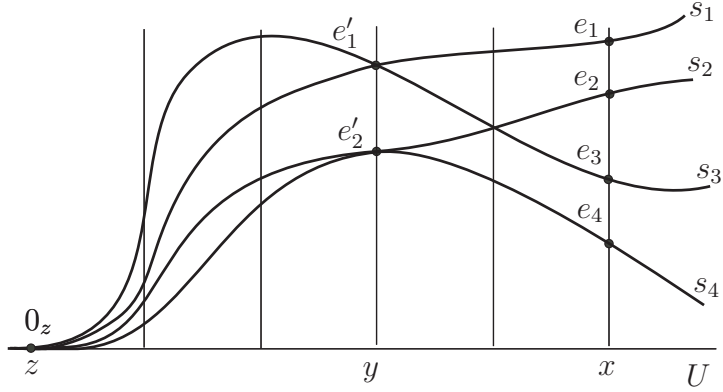}
\relabel {u}{$U$}
\relabel {x}{$x$}
\relabel {y}{$y$}
\relabel {z}{$z$}
\relabel {o}{$0_z$}
\relabel {e1p}{$e_1'$}
\relabel {e2p}{$e_2'$}
\relabel {o}{$0_z$}
\relabel {e1}{$e_1$}
\relabel {e2}{$e_2$}
\relabel {e3}{$e_3$}
\relabel {e4}{$e_4$}
\relabel {s1}{$s_1$}
\relabel {s2}{$s_2$}
\relabel {s3}{$s_3$}
\relabel {s4}{$s_4$}
\endrelabelbox}
\caption{Here $k=4$. If $P(e_j)=x$ for $j=1,2,3,4$, then $\Lambda (e_j)=\sigma_j$ while if $P(e_j')=y$ for $j=1,2$, then $\Lambda (e_1')=\sigma_1+\sigma_3$ and $\Lambda (e_2')=\sigma_2+\sigma_4$. Moreover, $\Lambda (0_z)=\sum_{j=1}^4\sigma_j=1$ where $0_z\in E$ is the  zero vector in the fiber over $z$. }
\label{Fig3.5.1}
\end{figure}

\begin{defn}\label{trivial}
The $\ssc^+$-multisection $\Lambda$ is called {\bf trivial on the set $V\subset X$} if $\Lambda$ is identically equal to $1$ on the zero section over $V$, i.e.,
$\Lambda (0_x)=1$ for all $x\in V$ (and hence $\Lambda (e_x)=0$ for all $e_x\neq 0_x$).
\end{defn}

The {\bf support of the $\ssc^+$-multisection $\Lambda$ } is the smallest closed set in $X$ outside of which $\Lambda$ is trivial.

If $\Lambda $ and $\Lambda'$ are two multisections in $\Gamma_m^+(P)$ we define their {\bf sum} as the multisection
$$
(\Lambda \oplus \Lambda' )(e)=\sum_{e'+e''=e}\Lambda (e')\cdot \Lambda' (e'').
$$
Explicitly, if the local section structure of $\Lambda$ is represented by the $\ssc^+$-sections $s_1,\ldots ,s_k:U\to E$ having the associated weights $\sigma_1,\ldots \sigma_k$ and the local section structure of $\Lambda'$ by the $\ssc^+$-sections $s_1',\ldots ,s_l':U\to E$ with
the associated weights $\sigma_1',\ldots \sigma_l'$, then at the vector $e_x\in E$ satisfying $P(e_x)=x$,
$$
(\Lambda \oplus \Lambda' )(e_x)=\sum_{s_i(x)+s_j'(x)=e_x}\sigma_i\cdot \sigma_j'.
$$
Hence  $s_i+s_j':U\to E$ are the  $\ssc^+$-sections and $\sigma_i\cdot \sigma_j'$ the associated weights of the multisection $\Lambda \oplus \Lambda'\in \Gamma^+_m (P)$  where $1\leq i\leq k$ and $1\leq j\leq l$.

The sum of two $\ssc^+$-multisections is by definition their
associated convolution product. We prefer to call it sum rather than
convolution product since in the single-valued case it precisely
corresponds to the sum.

 If $\Lambda_0,\Lambda_1\in \Gamma^+_m(P)$
and $\alpha \in (0, 1)\cap \Q^+$, we can define the $\ssc^+$-multisection  $\Lambda_{\alpha}\in
\Gamma_m^+(P)$ by
$$\Lambda_{\alpha}( e)=\alpha \Lambda_1( e)+(1-\alpha )\Lambda_0( e).$$
It is locally represented by all the local sections of
$\Lambda_0$ and $\Lambda_1$,  together with all the associated weights
multiplied by $\alpha$,  respectively  by $(1-\alpha)$.

We shall make use of
an auxiliary norm $N^\ast$ for a strong bundle   $P:E\rightarrow X$ over the
ep-groupoid $X$.

The first step in the construction of
an $\ssc^+$-multisection is the following lemma.
\begin{lem}\label{lem7.9.4}
Consider the strong bundle $P:E\rightarrow X$ over the ep-groupoid $X$ equipped with a
compatible auxiliary norm $N^\ast$. Assume the ep-groupoid $X$ is
modeled on sc-Hilbert spaces. Given a smooth point $x_0\in X$, a smooth
$h_0\in E_{x_0}$ and an open neighborhood $U\subset X$ of $x_0$,  then there exists an $\ssc^+$-section $s$ of the strong
bundle $E\rightarrow X$ (of objects, ignoring the morphisms)
satisfying
$$
s(x_0)=h_0
$$
and having its support in $U$. Moreover,  if $N^*(h)<\varepsilon$ we can
choose  the section $s$ in such a way that
$$
N^\ast(s(y))<\varepsilon.
$$
for all $y\in X$.
\end{lem}
\begin{proof}
The result is local. Hence we take a strong M-polyfold bundle chart
$\Phi:p^{-1}(U)\to K^{\mathcal R}$ covering the sc-diffeomorphism $\varphi:U\to O$ as defined in Definition 4.8  of \cite{HWZ2}.
The set $O$ is an open subset of
of the splicing core $K^{\mathcal S}=\{(v, e)\in V\oplus E'\vert \pi_v (e)=e\}$ associated with the  splicing ${\mathcal S}=(\pi, E', V)$ and  $K^{\mathcal R}=\{((v, e), u)\in O\oplus F\vert \, \rho_{(v, e)}(u)=u\}$ is the splicing core associated with the strong bundle splicing ${\mathcal R}=(\rho, F, (O, {\mathcal S}))$.
 In these coordinates the smooth  point $x_0$ corresponds to the smooth point $\varphi (x_0)=:(v_0, e_0)\in O$. Moreover,  $\Phi (x_0,h_0)=((v_0, e_0), h_0')\in K^{\mathcal R}$ where $h_0'$ is a smooth point in $F$, i.e., $h_0'\in F_{\infty}$.  For points $(v, e)\in O$ close to $(v_0, e_0)$ we define  the local section $s:K^{\mathcal R}\to O$ by
$$s(v, e)=((v, e), \rho_{(v, e)}(h_0'))\in K^{\mathcal R}.$$
At the point $(v_0, e_0)$ we have $\rho_{(v_0, e_0)}(h_0')=h_0'$ and  $s(v_0, e_0)=((v_0, e_0), h_0')$ as desired. In view of the definition of a  strong bundle splicing, $\rho_{(v, e)}(u)\in F_{m+1}$ if $(v, e)\in O_m\oplus F_{m}$ and $u\in F_{m+1}$, see Definition 4.2 in \cite{HWZ2}. Moreover,
the triple ${\mathcal R}^1=(\rho,  F^1, (O, {\mathcal S}))$ is also a general sc-splicing which
together with the fact that $h_0'$ is a smooth point in $F$ implies that the section $s$ of the bundle $K^{{\mathcal R}^1}\to O$ is sc-smooth. Consequently, $s$ is an $\ssc^+$-section  of the local strong M-polyfold bundle $K^{\mathcal R}\to O$.  We transport this section by means of the map $\Phi$ to obtain a local  $\ssc^+$ -section $s$ of  the given strong bundle $P:E\to X$; it  satisfies  $s(x_0)=h_0$. Using Lemma \ref{l2} in the appendix (whose proof makes use of the sc-Hilbert structure),  we find an sc-smooth  bump-function $\beta$ which is equal to $1$ near $x_0$ so that the section $\beta \cdot s$  has the desired properties.
\end{proof}

In general, the section constructed  in the proof of Lemma \ref{lem7.9.4} will  not be
compatible with morphisms.  Below we shall describe  a  general recipe for the
construction of  $\ssc^+$-multisections which are compatible with the morphisms. We fix  a smooth point $x\in X$
and an open neighborhood $U\subset X$ of $x$ which has the distinguished properties
listed in the structure Theorem \ref{lem7.9.3}.  In particular,  we have the natural
representation $\varphi:G_{x}\rightarrow \hbox{Diff}_{sc}(U)$  of the isotropy group $G_{x}$ at $x$ and a
sc-smooth map $\Gamma:G_x\times U\rightarrow {\bf X}$ having the
following properties.
\begin{itemize}
\item[$\bullet$] $\Gamma(g,x)=g$.
\item[$\bullet$] $s(\Gamma(g,y))=y$ and $t(\Gamma(g,y))=\varphi_g(y)$.
\item[$\bullet$] If $h:y\rightarrow z$  is a morphism between $y,z\in U$,
then there exists a unique $g\in G_x$ with $\Gamma(g,y)=h$.
\item[$\bullet$\: ] Assume that $y_0\in X$ is an object for which
there exists no morphism $y_0\to x'$ for an  $x'$ in $\ov{U}$. Then
there exists an open neighborhood $V$ of $y_0$ so that for every
$z\in V$ there is no morphism to an element in $\ov{U}$.
\item[$\bullet$\: ] Assume that  $y_0\in X$ is an object for which
there exists no morphism $y_0\to x'$ for every $x'\in U$, but there
exists a morphism to some element in $\ov{U}$. Then,  given an open
neighborhood $W$ of $\partial U$ (the set theoretic boundary of
$U$), there exists an open neighborhood $V$ of $y_0$ so that if
there is a morphism $y\to x'$ for some $y\in V$ and $x'\in U$, then
$x'\in W$.
\end{itemize}
With the smooth point $x\in U\subset X$ already chosen above we now choose a smooth vector $e_0\in E$ satisfying $P(e_0)=x$ and take, using Lemma \ref{lem7.9.4}, an $\ssc^+$-section $h$ of the strong bundle $E\to X$ (of objects) having its support in an open neighborhood $V$ of $x$ which is invariant under the sc-diffeomorphism
$\varphi_g \in \mbox{Diff}_{\ssc}(U)$ for all $g\in G_x$ and whose closure satisfies $\ov{V}\subset U$.

Recall the definition of the  strong bundle map $\mu:{\bf X}{{_s}\times_{p}}E\to E$ from
Section \ref{sect7.4}. It acts as follows. If $\varphi:x\to y$ is a
morphism in ${\bf X}$, then
$$\mu (\varphi,\cdot ):E_x\to E_y$$
is a linear isomorphism.  For every $g\in G_x$,  we define the $\ssc^+$-section $h_g$ of $E\vert U$  by
$$
h_g (\varphi_g(y)):=\mu (\Gamma(g,y),h(y)), \quad y\in U,
$$
and introduce  the map
$$\Lambda_{U}:E\vert U\to \Q^+$$
by
$$
\Lambda_{U}(e)=\frac{1}{\sharp G_x}\cdot \sharp \{g\in G_x\vert \,
h_g(P(e))=e\}.
$$

\begin{figure}[htbp]
\mbox{}\\[2ex]
\centerline{\relabelbox \epsfxsize 4truein \epsfbox{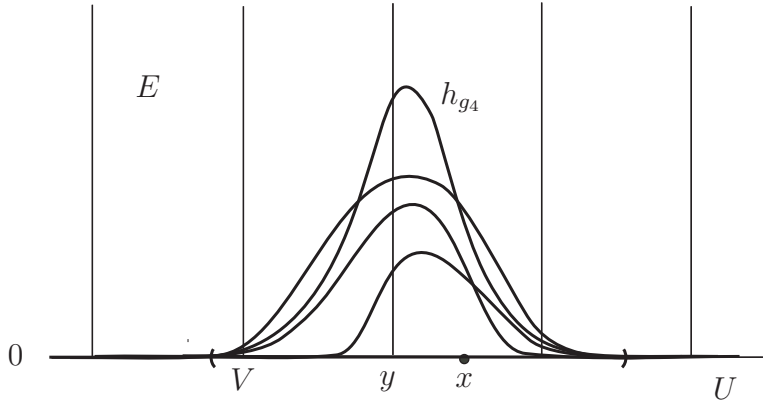}
\relabel {u}{$U$}
\relabel {v}{$V$}
\relabel {x}{$x$}
\relabel {y}{$y$}
\relabel {o}{$0$}
\relabel {e}{$E$}
\relabel {h}{$h_{g_4}$}
\endrelabelbox}
\caption{The associated weights are $\frac{1}{\sharp G_x}$. }
\label{Fig3.5.2}
\end{figure}

\begin{lem}\label{lem7.9.5}
The map $\Lambda_U$  defined on $E|U$ satisfies  $\Lambda_U (e)=\Lambda_U (e')$
if there exists a morphism $e\to e'$ in ${\bf E}$.
\end{lem}
\begin{proof}
Assume that $e$ and $e'\in E\vert U$ are related by a morphism in ${\bf E}$. hence, if  $P(e)=y$ and $P(e')=y'$ where $y$ and $y'\in U$ so that we can use the notation $e=e_y$ and $e'=e_{y'}$, then there is a morphism
$\psi$ in ${\bf X}$  satisfying
$$
\psi:y\rightarrow y'\quad \text{and}\quad \mu(\psi,e_y)=e_{y'}.
$$
By Theorem \ref{localstructure} there exists a uniquely determined $g_0\in G_x$ satisfying  $\Gamma(g_0,y)=\psi$
and $\varphi_{g_0}(y)=y'$. If  $g\in G_x$ we compute, using the properties  of  the strong bundle map $\mu:{\bf E}\to E$,
\begin{equation*}
\begin{split}
h_g(\varphi_{g_0}(y))&=
\mu(\Gamma(g,\varphi_{g^{-1}}(y')),h(\varphi_{g^{-1}}(y')))\\
&=\mu(\Gamma(g,\varphi_{g^{-1}g_0}(y)),h(\varphi_{g^{-1}g_0}(y)))\\
&=\mu(\Gamma(g_0g_0^{-1}g,\varphi_{g^{-1}g_0}(y)),h(\varphi_{g^{-1}g_0}(y)))\\
&=\mu(\psi\circ\Gamma(g_0^{-1}g,\varphi_{g^{-1}g_0}(y)),h(\varphi_{g^{-1}g_0}(y)))\\
&=\mu(\psi,\mu(\Gamma(g_0^{-1}g,\varphi_{g^{-1}g_0}(y)),h(\varphi_{g^{-1}g_0}(y)))\\
&=\mu(\Gamma (g_0, y),h_{g_0^{-1}g}(y)).
\end{split}
\end{equation*}
Since $\mu (\psi, \cdot ):E_y\to E_{y'}$ is an isomorphism, it follows that  $h_{g_0^{-1}g}(y)=e_y$ for some $g\in G_x$ if and only if
$h_{g}(\varphi_{g_0}(y)=e_{y'}$. This implies
$$
\Lambda_U(e)=\Lambda_U(e').
$$
The proof  of Lemma \ref{lem7.9.5} is complete.
\end{proof}

Lemma \ref{lem7.9.5} shows that the map $\Lambda_U:E\vert U\to \Q^+$  possesses  on $U$  the local section structure defined by the $\ssc^+$-sections $h_g:U\to E\vert U$ for all $g\in G_x$.

Next  we extend the local $\ssc^+$-multisection $\Lambda_U$ from $E\vert U$  to all of $E$. In order to first extend $\Lambda_U$ from $E\vert U$ to $E\vert \ov{U}$ we recall that the invariant set $V\subset U$ contains the supports of all sections $h_g$ for $g\in G_x$   and satisfies $\ov{V}\subset U$. Consequently, the multisection $\Lambda_U$ is trivial over the set $U\setminus V$ and, in view of the properties listed in the structure Theorem \ref{lem7.9.3}, we  can define the multisection $\Lambda$ on the fiber $E_y$ over the boundary point $y\in \ov{U}\setminus U$ by $\Lambda (e)=0$ if $e\neq 0_y$ and by $\Lambda (0_y)=1$ for the zero $0_y$ in $E_y$. If $y\not \in \ov{U}$ and if there is no morphism $y\to x'$ to some  $x'\in \ov{U}$ we
define $\Lambda$ to be trivial over $y$.  On the other hand, if there exists a morphism $e\to e'$ with $P(e')\in \ov{U}$, we define $\Lambda  (e)=\Lambda (e')$.

So far we have defined a map $\Lambda:E\to \Q^+$ and it follows from the construction of $\Lambda$ and from Lemma \ref{lem7.9.5} that $\Lambda (e)=\Lambda (e')$ if there is a morphism $e'\to e$ in  ${\bf E}$,  so that $\Lambda$ induces a functor $\Lambda:E \to \Q^+$.

\begin{prop}\label{prop7.9.6}
Let $P:E\rightarrow X$ be a strong bundle over the  ep-groupoid $X$.
We assume that  the sc-smooth structures are based on separable  Hilbert spaces. Then
the functor $\Lambda:E\to \Q^+$ constructed above is an $\ssc^+$-multisection.
\end{prop}
\begin{proof}
Let $U\subset X$ be the distinguished open set considered above and let $y\in U$.  Then there are by construction finitely many local
$\ssc^+$-sections, namely $h_g$ for $g\in G_x$ with associated weights $\frac{1}{\sharp G_x}$,  so that $\Lambda$ has the desired local structure. If  $y\not \in U$ and if  there is no morphism $y\to x'$  to some
point $x'\in \ov{U}$, then $\Lambda$ is trivial over $y$. By the structure
Theorem \ref{lem7.9.3}, there is
an open neighborhood $V$ of $y$ which does not admit morphisms into
$\ov{U}$. Consequently, the multisection $\Lambda$ is trivial over the open set $V$. Finally, if $y\not \in U$ and if there exists a morphism $y\to x'\in \ov{U}$, then   $t\circ s^{-1}$ defines a local diffeomorphism $U(y)\to U(x')$ between open neighborhoods so that $\Lambda\vert U(y)$ inherits the local section structure from $\Lambda\vert U(x')$.  This completes the proof of Proposition \ref{prop7.9.6}.
\end{proof}
We  summarize
the previous discussion in the following theorem.
\begin{thm}\label{thm13.1}
Let $P:E\rightarrow X$ be a strong bundle over an ep-groupoid $X$ and let $N^\ast$ be  an auxiliary norm for $P$.
Assume the sc-smooth structures are based on separable  Hilbert spaces. Assume that  $e$ is a smooth point in $E$ and $U$ is a saturated open neighborhood in $X$ of
the point $x=P(e)$. Then  there exists an  $\ssc^+$-multisection $\Lambda:E\to \Q^+$ having its  support in $U$ (i.e., $\Lambda$ is trivial on $X\setminus U$)  and
satisfying $\Lambda(e)>0$. In addition,  if $N^\ast(e)<\varepsilon$,  then
$N^\ast(h)<\varepsilon$ for all $h$ satisfying  $\Lambda(h)>0$.
\end{thm}
If $\Lambda:E\to \Q^+$ is an $\ssc^+$- multisection for the strong bundle $E\rightarrow X$ and $\Phi:E'\rightarrow E$
a strong bundle equivalence  covering the equivalence
$\varphi:X'\rightarrow X$ between the underlying ep-groupoids, then  the pull-back
$\Phi^\ast\Lambda:E'\to \Q^+$  is again an  $\ssc^+$-multitsection. If
$\Theta$ is an $\ssc^+$-multisection for $E\rightarrow X$ we can
also define  the push-forward $\Phi_\ast(\Lambda)$. As in the
discussion of sections and generalized isomorphisms we have the
following result.
\begin{prop}
If  $\mathfrak{A}:P\rightarrow Q$ is  an s-bundle isomorphism, where
$P:E\rightarrow X$ and $Q:L\rightarrow Y$ are strong bundles over
ep-groupoids, then $\mathfrak{A}$ induces a bijection
$$
\mathfrak{A}_\ast:\Gamma_m^+(P)\rightarrow\Gamma^+_m(Q)
$$
between the corresponding spaces of multisections. The inverse is
the pull-back $\mathfrak{A}^\ast$.
\end{prop}

The previous construction of the multisection $\Lambda$ allows  also a parameterized version as  explained  in the following remark, which will be useful in the perturbation theory later on.
 \begin{remark}{\bf (Parametrized $\ssc^+$-multisections)}\label{rem13.2}
Take a point $x\in X$ and choose an open neighborhood $U$ of $x$ in the object set of the ep-gropupoid $X$. We take the set $U$ so small that the isotropy group $G_x$
 acts on $U$  by its   natural representation.  By means of Lemma \ref{lem7.9.4}, we choose  a finite number of $\ssc^+$-sections $h^1,\ldots,  h^k$  of the strong polyfold bundle $P:E\to X$ (of objects ignoring the morphisms) having their supports in $U$ and taking at some  smooth point $x_0\in U$  the prescribed smooth values $e^j:=h^j(x_0)$.
Assuming that $N^*(e^j)<\frac{\varepsilon}{k}$,  we may
achieve, again by Lemma \ref{lem7.9.4}, that
$$
N^*(h^j(y))<\frac{\varepsilon}{k}\quad \text{ for all $1\leq j\leq k$ and all $y\in U$.}
$$
By $\abs{t}_{\infty}$ we denote the ${\ell }^{\infty}$-norm  of $t=(t_1,\ldots,t_k)\in \R^k$. For every $t$ satisfying  $\abs{t}_\infty< 1$,  we  define the $\ssc^+$-section  $h_t$  as the linear combination
$$
h_t (y)= \sum_{j=1}^{k} t_j\cdot h^j(y).
$$
The support of  the section $h_t$ is contained in $U$ and  $N^*(h_t (y))<\varepsilon$ for all $y\in X$ and all $\abs{t}_{\infty}<1$. Using the group action by  the isotropy group $G_x$,  we  define for  every $g\in G_x$ the $\ssc^+$- sections  $h_{t,g} $ of $E\vert U$ by
\begin{equation*}
\begin{split}
h_{t, g}(\varphi_g(y))&=\mu \bigl(\Gamma(g,y),h_t(y)\bigr)\\
&=
\sum_{j=1}^kt_j\cdot \mu\bigl(\Gamma (g, y), h^j (y)\bigr).
\end{split}
\end{equation*}
Now proceeding as in Lemma \ref{lem7.9.4} and Proposition \ref{prop7.9.6}, we obtain for every $\abs{t}_{\infty}<1$ an  $\ssc^+$-multisection $\Lambda^t:E\to \Q^+$.

In order to construct transversal perturbations of Fredholm sections later on we will make use of the free choice of the smooth images $e^j$ at  the smooth point $x_0$ to  fill up the
cokernel of the linearized  Fredholm section at $x_0$.
\end{remark}

Multisections for polyfolds will also be defined by means of the overheads as follows. If
$p:W\rightarrow Z$ is a strong polyfold bundle, we choose a strong polyfold bundle structure $(E, \Gamma ,\gamma )$ for $p$ in which $P:E\to X$  is a strong bundle over the ep-groupoid $X$ and $\Gamma:\abs{E}\to W$  the homeomorphism covering the homeomorphism
$\gamma:\abs{X}\to Z$.

Consider a pair $(\lambda,\Lambda)$ in which $\Lambda:E\to \Q^+$ is an $\ssc^+$-multisection on $P$ and
$\lambda:W\rightarrow {\mathbb Q}^+$ is the  function satisfying
$$
\lambda(w)= \Lambda(e)\quad \text{if $w=\Gamma (\abs{e}).$}
$$
If $(E', \Gamma', \gamma')$ is a second such model for our strong polyfold bundle $p$ and if $(\lambda',\Lambda')$ is the corresponding pair, we call the two  {\bf  pairs equivalent} if
$$\lambda =\lambda'$$
and, moreover, if there exists an s-bundle isomorphism $\mathfrak{A}:E\Rightarrow E'$ so that for a representative
$E\xleftarrow{\Phi}E''\xrightarrow{\Psi}E'$ (in which $\Phi$ and $\Psi$ are strong bundle equivalences), the $\ssc^+$-multisections $\Lambda:E\to \Q^+$ and $\Lambda':E'\to \Q^+$ are related by
$$\Phi^\ast\Lambda=\Psi^\ast \Lambda'.$$
\begin{defn}
An {\bf $\ssc^+$-multisection for the strong polyfold bundle} $p:W\to Z$ is an equivalence class $[\lambda, \Lambda]$ of pairs.
\end{defn}
\section{Global Fredholm Theory}

In this section we transplant the basic ideas from the Fredholm
theory in \cite{HWZ2} to the polyfold set-up.
\subsection{Fredholm sections}
We begin with the notion of a Fredholm section of a strong polyfold bundle.

\begin{defn}
{\em The section $f$ of the strong polyfold bundle $p:W\rightarrow Z$ is
called a {\bf Fredholm section}  provided there exists a
representative $F$ of $f$ which is a Fredholm section of the strong bundle
$P:E\rightarrow X$ over the ep-groupoid $X$.  The latter means that
$F:X\to E$ is an sc-smooth functor and a Fredholm section  of the strong bundle $E\rightarrow X$, where
$X$ is the M-polyfold of objects, as defined in  \cite{HWZ3}.

The Fredholm section
$f$ is  called {\bf proper}  provided the solution set $S=S(f)=\{z\in Z\vert \, f(z)=0\}$
is compact in $Z$.}
\end{defn}

Let us observe the following fact, already established in the
M-polyfold case.
\begin{prop}
If $f$ is a proper Fredholm section of the strong polyfold bundle $p:W\to Z$, then the solution set $S=f^{-1}(0)$ is compact in $Z_\infty$.
\end{prop}
\begin{proof}
Let $(E, \Gamma, \gamma )$ be a strong bundle structure for $p$ in which $P:E\to X$ is a strong bundle over the ep-groupoid $X$, $\Gamma$ is a homeomorphism between the orbit space $\abs{E}$ and $W$, and $\gamma:\abs{X}\to Z$ is a homeomorphism.  Assume that the section $F$  of the bundle $P:E\to X$ is  a  representative  of the section $f$. Then $F$ is a proper Fredholm section of the bundle $P$ and, in addition, the section  $F$ is a functor. Since $F$ is regularizing the solution set $F^{-1}(0)$ consists of smooth points
and therefore $f^{-1}(0))=\gamma(\abs{F^{-1}(0)})$ consists of smooth
points in $Z$. Take a sequence $(z_k)\subset   f^{-1}(0)$. Since, by assumption,  the solution set $S$ is compact in $Z=Z_0$,  we may assume  possibly taking a subsequence
that $z_k\rightarrow z$ in $Z=Z_0$.   Choose a point $x\in X$ such that $\pi (x)=\abs{x}=\gamma^{-1}(z)$ where $\pi:X\to \abs{X}$ is  the  quotient map.  Take an open neighborhood $U$ of $x$ in $X$  which is invariant under the isotropy group $G_x$. Then the set $\pi (U)$ is an open neighborhood of $\abs{x}=\gamma (z)$ and
 $\gamma^{-1}(z_k)\in \pi (U)$ for  $k$ large.   Since the set $\pi^{-1}(\pi (U))$ consists of   points which can be connected by morphisms with points in $U$, we  find  points $x_k\in U$ such that $\pi (x_k)=\abs{x_k}=\gamma^{-1}(z_k)$.  Consequently,  the sequence $(x_k)$ converges to $x$ in $X_0$.
Since  $F(x_k)=0$ and $x_k\rightarrow x$, it follows from the local normal form of a Fredholm section,  arguing as in the proof of Theorem 5.11 in \cite{HWZ3},   that $x_k\rightarrow x$ in $X_\infty$.
Therefore,  $z_k\rightarrow z$ in $Z_\infty$ as claimed.
\end{proof}
\subsection{Properness}

We shall introduce a useful notion   for the constructions later on. Assume that $p:W\to Z$ is a  strong polyfold bundle and let $f:Z\to W$ be a Fredholm section of $p$. We assume further that $N:W_{0,1}\to [0,\infty )$ is an auxiliary norm and $U$ is an open neighborhood of $f^{-1}(0)$.
\begin{defn}\label{control}
We say that the pair $(U, N)$ {\bf controls compactness} if  every sequence $(z_k)\subset \ov{U}$  satisfying
$$\liminf_{k\to \infty} N(f(z_k))\leq 1$$
possesses  a converging subsequence.
\end{defn}
\begin{lem}\label{lemcontrol}
Let   $P:E\to X$ be  a local model for the strong polyfold bundle $p$ and let $F:X\to E$  Fredholm section   representing $f$. Assume that $N^*:E_{0,1}\to [0,\infty )$ is  the auxiliary norm on $P$ representing $N$.
If the pair $(U,N)$ controls compactness, then the pair $(U^*, N^*)$ where $U^*=\pi^{-1}(U)$ has the following property. If $(x_k)$ is any sequence in  $\ov{U^*}$  satisfying
$$\liminf_{k\to \infty}N^* (F(x_k))\leq 1,$$
then there exist a sequence $(y_k)$ of points in $\ov{U^*}$ and a sequence  $(\varphi_k)$ of morphisms $\varphi_k:x_k\to y_k$ so that
$(y_k)$ has a converging subsequence.
\end{lem}
\begin{proof}
Without lost of generality we may assume that $Z=\abs{X}$ and $W=\abs{E}$. Take a sequence $(x_k)\subset
\ov{U^*}$ satisfying $\liminf_{k\to \infty} N^*(F(x_k))\leq 1$. Consider the equivalence classes $z_k=\abs{x_k}$. Because $x_k\in \ov{U^*}$, the  points $z_k$ belong to $\ov{U}$. By definition, $f(z_k)=\abs{F(x_k)}$ and  it follows that $N(f(z_k))=N^*(F(x_k))$ so that  $\liminf_{k\to \infty} N(f(z_k))\leq 1$.  Since, by assumption,  the pair $(U, N)$ controls compactness, it follows that the there is a subsequence, again denoted by $(z_k)$, converging to some point $z\in \ov{U}$.  Choose a point $x$ such that $z=\abs{x}$. There exists a neighborhood basis  $(V_j)_{j\in \N}$ of $x$ such that $V_{j+1}\subset  V_j$ for all $j\in \N$.  Then the sets $\pi (V_j)$ form a decreasing sequence of  open neighborhoods  of $z$ in $\abs{X}$.  Since $z_k\to z$,   we find for every $j\in \N$ an index $k_j$ such that  $z_{k_j}\in \pi (V_j)$ and $k_{j+1}>k_j$.  Hence $x_{k_j}\in \pi^{-1}(\pi (V_j))$ and since  the set $\pi^{-1}(\pi (V_j))$ consists of points which are related by morphisms to the points in $V_j$,  we find  for every index $k_j$ a point  $y_{k_j}\in V_j$ and a morphism $\varphi_{k_j}: x_{k_j}\to y_{k_j}$.  We note that $y_{k_j}\to x$ as $j\to \infty$. For $k\neq k_j$, we choose  $y_k=x_k$ and take  as morphism $\varphi_k$ the identity  morphism $1_{x_k}$. Then $(y_{k_j})$ is the desired subsequence of the sequence $(y_k)$ converging to $x$. The proof  of the lemma is complete.
\end{proof}

As a consequence of the local properness of a Fredholm section (Theorem 5.9 in \cite{HWZ3})  we obtain the following result.

\begin{thm}\label{thm11.12}
Let $p:W\to Z$ be the strong polyfold bundle with reflexive fibers and let $N$ be an auxiliary norm for $p$.  Assume that $f$ is a proper
Fredholm section of the bundle $p$.  Then there exists an open neighborhood $U$ of the
set $S=f^{-1}(0)$ so that  the pair $(U, N)$ controls compactness.
\end{thm}
\begin{proof}
We choose a local model  $P:E\to X$  for the strong polyfold bundle $p$ and a proper  Fredholm section $F$  of the bundle $P$ representing  the section $f$. Without loss of generality we may assume that $Z=\abs{X}$ and $W=\abs{E}$.  The map
$N^*: E_{0,1}\to [0,\infty )$ defined by $N^*(e)=N(\abs{e})$ is an  auxiliary norm for the strong polyfold bundle $P$.
 The solution set $S=f^{-1}(0)\subset Z$ is,   by assumption,  compact.  Take a point
$x\in X$  for which  $\abs{x}\in S$. Then $F(x)=0$ and,  by the regularizing
property of the Fredholm section $F$, the solution  $x$ is a smooth point in $X$. In view the local compactness for Fredholm sections of fillable strong M-polyfold bundles, there
exists  (Theorem 5.9 in \cite{HWZ3}) an open neighborhood $U(x)\subset X$ so that  every sequence
$(x_k)$  in $\overline{U(x)}$ satisfying
$$
\liminf_{k\to \infty}N^*(F(x_k))\leq 1
$$
possesses a  convergent subsequence.  Shrinking $U(x)$ is necessary, we may assume that $U(x)$ is invariant under the action of the isotropy group $G_x$. Then the set  $\abs{U(x)}$  is open in
$Z$. By the compactness of $S$ in $Z$ we find finitely
many points $x_0,\ldots ,x_K\in X$  so that
$U:=\abs{U(x_0)}\cup \cdots \cup \abs{U(x_K)}$ is an open neighborhood of
$S$.   Abbreviate  $U^\ast=\pi^{-1}(U)$ where $\pi:X\to \abs{X}=Z$ is the quotient map onto the orbit space. Since $F$ is a functor, the set $U^*$ is an open neighborhood of  $F^{-1}(0)$.  Let
$(z_k)$ be a  sequence of points in $\ov{U}$ satisfying
$$
\lim_{k\rightarrow\infty} N(f(z_k))\leq 1.
$$
Then we  find a sequence $(y_k)\in \ov{U(x_0)}\cup\cdots \cup \ov{U(x_K)}$ satisfying
$\abs{y_k}=z_k$. Consequently,
$$
\lim_{k\rightarrow \infty} N^*(F(y_k))\leq 1.
$$
By construction,  the sequence $(y_k)$ has a convergent subsequence.
Therefore, the sequence $z_k=\abs{y_k}$ in the orbit space $\abs{X}=Z$  has a convergent subsequence and the proof of the theorem is complete.
\end{proof}

\subsection{Transversality and Solution Set}
We consider  a strong polyfold bundle $p:W\rightarrow
Z$ and let $f$ and $\lambda$ be a   proper Fredholm section  and an
$\ssc^+$-multisection of the bundle $p$, respectively.  We denote by  $P:E\to X$ the local model of the bundle $ p$ and by $F$  the  Fredholm section  representing $f$  and by $\Lambda:E\to \Q^+$  the correspondng $\ssc^+$-multisection representing  $\lambda$. First we define the solution set.
\begin{defn}
The {\bf solution set} $S:=S(f,\lambda)$ of the pair $(f,\lambda)$
is the set
$$
S(f,\lambda)=\{z\in Z |\
\lambda(f(z))>0\}.
$$
\end{defn}
At this point $S$ as a subset of $Z$ is just a second countable
paracompact topological space (as a closed subset of $Z$). We shall
see, however, that in case a certain transversality condition is
met, $S$ carries an additional structure so that not only one can
talk about orientability of $S$ but also about  the  integration of
sc-differential forms over $S$.
 Of course, in developing this additional structure
 the {\bf overhead}
given by the various  representatives $(F,\Lambda)$ describing $(f,\lambda)$ will be
important.  Taking the representative $(F,\Lambda)$  of
$(f,\lambda)$,  we consider the solution set
$${\mathcal S}=S(F,\Lambda)=\{x\in X\vert \Lambda (F(x))>0\}.$$

Recall that if $x$ belongs to the solution set, then there exist, in view of the definition of an $\ssc^+$-multisection, an open neighborhood $U\subset X$ of $x$ and finitely many local $\ssc^+$-sections $s_i:U\to E$, for $i\in I$, having the associated positive rational weights $\sigma_i$, $i\in I$, so that
$$\sum_{i\in I}\sigma_i=1\quad \text{and}\quad \Lambda (F(x))=\sum_{\{j\in I\, \vert F(x)=s_j(x)\}}{\sigma_j}$$
and  there is at least one index $j\in I$ such that
$$F(x)=s_j(x). $$
If $\Lambda (F(x))=0$, then there is no index $j\in I$ for which $F(x)=s_j(x).$

The  natural map
$$
{\mathcal S}\rightarrow S, \quad x\mapsto \gamma(\abs{x})
$$
 induces a homeomorphism $|{\mathcal S}|\rightarrow S$.  The
solution set $S$ comes with the  natural map $\lambda_f:S\rightarrow
{\mathbb Q}^+\cap(0, \infty)$,  defined by
$$
\lambda_f(z) = \lambda(f(z)),
$$
and called  the {\bf weight function}  on $S$.

We want to study  the pairs $(S(f,\lambda),\lambda_f)$ provided some
transversality conditions are met so  that the pair
$(S(f,\lambda),\lambda_f)$ has the structure of a smooth branched
suborbifold with boundary with corners. These transversality
conditions are defined in terms of a representative $(F,\Lambda)$ as follows.

\begin{defn}\label{DER1}  Let $p:W\to Z$ be a strong polyfold bundle and let $(f, \lambda)$ be a pair in which $f$ is a
Fredholm section of $p$ and
$\lambda:W\to \Q^+$ an $\ssc^+$-multisection on $p$. Assume that  the strong bundle $P:E\to X$ over the ep-groupoid $X$ is a model representing $p$ and  the Fredholm section $F:X\to E$ of the bundle $P$
and the $\ssc^+$--multisection $\Lambda :E\to \Q^+$ of  $P$ are  representatives of $f$ and $\lambda$.
\begin{itemize}
\item[(1)]  The pair  $(f,\lambda)$ is called a  {\bf transversal pair}  if  for every
$z\in S(f,\lambda)$ the following holds.
 If $x\in X$ represents $z$ and $(s_i)$ is a local section structure for $\Lambda$ near $x$, then for every $i$ for which  $F(x)-a_{i}(x)=0$,
the linearisation $(F-a_i)'(x):T_xX\rightarrow E_x$  is  surjective.
\item[(2)]   The pair   $(f,\lambda)$ is in {\bf good position} if  for every
$z\in S(f,\lambda)$ the following holds. If $x\in X$ represents $z$ and $(a_i)$ is a
local section structure for $\Lambda$ near $x$,  then for every $i$
for which $F(x)-a_i(x)=0$,  the linearization $(F-a_i)'(x)$ is surjective  and
its  kernel is in good position to the corner structure of $X$ as defined in  \cite{HWZ3}.
\item[(3)]  The pair  $(f,\lambda)$ is in {\bf general position to
the boundary}  $\partial Z$  if for every $z\in S(f, \lambda)$ the following
holds. If $x$ represents $z$ and $(a_i)$ is a local section
structure for $\Lambda$ near  $x$,  then for every $i$ for which
$F(x)-a_i(x)=0$ the linearization $(F-a_i)'(x)$ is  surjective  and
the kernel of $(F-a_i)'(x)$ is transversal to $T^\partial_xX$ in $T_xX$. Here
$T^\partial_xX$ is the intersection of all tangent spaces at $x$ to
the local faces containing $x$ and we refer to  \cite{HWZ3} for more details.
\end{itemize}
\end{defn}
\begin{prop}\label{lemmontreal}
If one of the above properties holds true for one  local section structure, then it holds true for all the other local sections structures.
\end{prop}

In order to prove the proposition we first  introduce the concept of a linearization of a Fredholm section with respect to a multisection. It then allows a new elegant formulation of Definition \ref{DER1}. We consider the Fredholm section $F:X\to E$ of the strong bundle $p:E\to X$ and let $\Lambda:E\to \Q^+$ be an $\ssc^+$-multisection. We fix a point belonging to the solution set $S(F, \Lambda )=\{x\in X\vert \, \Lambda (F(x))>0\}$ and define the linearization of $F$ at the solution $x$ with respect to $\Lambda$ as follows.

In view of the definition of the $\ssc^+$-multisection $\Lambda$, there exist an open neighborhood $U_x$ of $x$ and a finite collection of local sections $(a_i)_{i\in I}$ with associated weights $(\sigma_i)_{i\in I}$ such that
$$\Lambda (e)=\sum_{\{i\in I\vert \, a_i (Pe)=e\}}\sigma_i.$$

By $I'$  we denote the set of indices $i\in I$ for  which $F(x)-a_i(x)=0$. If $i, j\in I'$, we call the two linearizations $(F-a_i)'(x)$ and $(F-a_j)'(x):T_xX\to E_x$ {\bf equivalent} if
$$(F-a_i)'(x)\cdot \delta x =(F-a_j)'(x)\cdot \delta x \quad \text{for all $\delta x\in T_xX$}.$$
Denoting by $[(F-a_i)'(x)]$  the equivalence class of the operator $(F-a_i)'(x)$, we define the {\bf linearization $F_{\Lambda}'(x)$ of $F$ at $x$ with respect to  the multisection $\Lambda$ }  to be the finite collection
of all equivalence classes
$$[(F-a_{i_1})'(x)], \, [(F-a_{i_2})'(x)],\, \ldots \, ,[(F-a_{i_n})'(x)].$$
The notion of the linearization $F_{\Lambda}'(x)$ is independent of the choice of the local section structure of the multisection $\Lambda$ as the following proposition, which has Proposition \ref{lemmontreal} as an immediate consequence, shows.

\begin{prop}\label{linearlambda}
Assume that  $(a_i)_{i\in I}$ and $(b_j)_{j\in J}$ are  two  local section structures for the multisection $\Lambda$ in $U_x$ and let
$$[(F-a_{i_1})'(x)], \, [(F-a_{i_2})'(x)],\, \ldots \, ,[(F-a_{i_n})'(x)]$$
and
$$[(F-b_{j_1})'(x)], \, [(F-b_{j_2})'(x)],\, \ldots \, ,[(F-b_{j_m})'(x)]$$
be the  equivalence classes defined above.  Then $n=m$ and for every $i\in \{i_1, \ldots ,i_n\}$, there exists exactly one $j\in \{j_1, \ldots ,j_m\}$ such that
$$[(F-a_{i})'(x)]=[(F-b_j)'(x)].$$
\end{prop}
To prove the proposition  we will need  the following lemma.
\begin{lem}\label{lemB}
Let $(a_i)_{i\in I}$ and $(b_j)_{j\in J}$ be   two  local section structures for the multisection  $\Lambda$ in the  open neighborhhod $U_x\subset X$ around the solution $x\in S(F, \Lambda)$.
Then given $i\in I$ for which $F(x)-a_i(x)=0$ and given $\delta x\in T_xX$, there exists an index $j\in J$ such that $F(x)-b_j(x)=0$ and
$$[F(x)-a_i(x)]'\cdot \delta x=[F(x)-b_j(x)]'\cdot \delta x.$$
\end{lem}\begin{proof}
Since the smooth tangent vectors are dense on every level, we first consider the case of smooth tangent vectors and  deal with the arbitrary tangent vectors  later on.
We fix  the index $i\in I$ for which $F(x)-a_i(x)=0$. We work in local coordinates around the solution $x$. Hence we assume that $U=U_x$ is a $G_x$-invariant open set of the splicing core $K^{{\mathcal S}}=\{(v, e)\in V\oplus G\vert \pi (v, e)=e\}$ associated with the splicing ${\mathcal S}=(\pi, V, G)$ in which $V$ is an open subset of the partial quadrant $C$ of the sc-Banach space $W=\R^n$ for some $n$, and $G$ is an sc-Banach space. We introduce the notation $x=(v_0, e_0)$ and define for $y=(v, e)\in U_x\subset V\oplus G$ the map $A:U_x\to E$ by
$$A(y)=F(y)-a_i(y).$$
If $\delta x=(\delta v, \delta e)\in W\oplus G$ is a smooth tangent vector in $T_xU$, then
$$(D_v\pi )(x)\cdot \delta v+\pi_{v_0}(\delta e)=\delta e.$$
For $t$ small we therefore find
\begin{equation*}
\begin{split}
&\pi_{v_0+t\delta v}(e_0+t\delta e)\\
&=\pi_{v_0}(e_0)+t\cdot D\pi(v_0, e_0)[\delta v, \delta  e]+o(t)\\
&=e_0+t\cdot [(D_v\pi )(v_0, e_0)\cdot \delta v+\pi_{v_0}( \delta e]+o(t)\\
&=e_0+t\delta e+o(t),
\end{split}
\end{equation*}
where $\frac{o(t)}{t}\to 0$ as $t\to 0$ on every level of the sc-Banach space $Q$. Consequently, applying the linear projection
$\pi_{v_0+\delta v}$ to both sides,
\begin{equation*}
\begin{split}
\pi_{v_0+t\delta v}(e_0+t\delta e+o(t))&=(\pi_{v_0+t \delta v})^2 (e_0+t \delta e)\\
&=\pi_{v_0+t \delta v}(e_0+t  \delta e)\\
&=e_0+t\delta e +o(t).
\end{split}
\end{equation*}
This implies that the curve $t\mapsto (v_0+t\delta v, e_0+t\delta e+o(t))$ through $x=(v_0, e_0)$ belongs to $U_{x}\subset K^{{\mathcal S}}$ for small values of $t$.  Let $(t_n)$ be any sequence converging to $0$. We define the sequence of points $x_n\in U_x$ by
$$x_n=(v_0+t_n\delta v, e_0+t_n \delta e +o(t_n)).$$
Then $x_n\to x=(v_0, e_0)$ as $n\to \infty$.
Since the local system sections $(a_i)_{i\in I}$ and $(b_j)_{j\in J}$ define the multisection
$\Lambda$, it follows that
$$\Lambda (a_i(x_n))=\sum_{\{j\in J\vert \, b_j(x_n)=a_i(x_n)\}}\tau_j>0$$
for every $n$.
Therefore, there exists a sequence $(j_n)\in J$ of indices such that
$$b_{j_n}(x_n)=a_i(x_n).$$
Because $J$  is a finite set, there must exist an index $j\in I$ and a subsequence of $(x_n)$ (denoted again by $(x_n)$) such that  for all $n$
$$b_j(x_n)=a_i(x_n).$$
As  $n\to \infty$,  it follows that
$$b_j(x)=a_i(x).$$
Introducing the map $B:U_x\to E$ by
$$B(y)=F(y)-b_j (y)$$
we have proved that  $A(x)=B(x)=0$ and
$$A(x_n)=B(x_n)$$
for all $n$.

Hence, in view of $A(x_n)=A(x)+DA(x)\cdot (x_n-x)+o(x_n-x)$ and $B(x_n)=B(x)+DB(x)\cdot (x_n-x)+o(x_n-x)$ and $A(x)=B(x)$, we conclude that
$$[DA(x)-DB(x)]\cdot (x_n-x)=o(x_n-x)$$
on every level.
Dividing by $t_n$ and taking the limit as $n\to \infty$ we find for the tangent vector $\delta x$,
$$DA(x)\cdot \delta x=DB(x)\cdot \delta x.$$

If  $\delta x\in T_xX$ is an arbitrary tangent vector, we take a sequence $(\delta x^n)$ of smooth tangent vectors converging to $\delta x$.  By the first part of the proof,  there exists a sequence $(j_n)\subset J$ such that
$$D[F-a_i](x)\cdot \delta x^n=D[F-b_{j_n}](x)\cdot \delta x^n.$$  Because  $J$ is a finite set, there exist an index $j$ and a subsequence, again denoted by $(\delta x^n)$, such that
$$D[F-a_i](x)\cdot \delta x^n=D[F-b_{j}](x)\cdot \delta x^n $$
for all $n$.
Taking the limit as $n\to \infty$,  we conclude
$$DA(x)\cdot \delta x=DB(x)\cdot \delta x.$$
This completes the proof of Lemma \ref{lemB}.
\end{proof}

\begin{proof}[Proof of Proposition \ref{linearlambda}]
We abbreviate by $I'$ the set of indices $i\in I$ for which $F(x)-a_i(x)=0$ and by $J'$ the set of indices $j\in J$ for which $F(x)-b_j(x)=0$.  In order to  prove the proposition it suffices to show that  for given $i\in I'$ there exists $j\in J'$ such that $(F-a_i)'(x)\cdot \delta x=(F-b_j)'(x)\cdot \delta x$ for all $\delta x\in T_xX$. If this is not the case,  then for every $j\in J'$, there exists a vector $\delta x^j\in T_xX$ such that $(F-a_i)'(x)\cdot \delta x^j\neq (F-b_j)'(x)\cdot \delta x^j$ for all $j\in J'$.   From
\begin{equation}\label{eqAA}
a'_i(x)-b_j'(x)=(F-b_j)'(x)-(F-a_i)'(x),
\end{equation}
if follows that the  kernels $\text{ker}\ (a_i'(x)-b_j'(x))$ are closed proper subspaces of $T_xX$.
Consequently, applying the Baire category theorem, we conclude  $T_xX\setminus  \bigcup_{j\in J'}\text{ker}\, \bigl(a_i'(x)-b_j'(x)\bigr)\neq \emptyset$.
Hence, in view of  \eqref{eqAA},
\begin{equation}\label{eqBB}
(F-a_i)'(x)\cdot \delta x\neq (F-b_j)'(x)\cdot \delta  x
\end{equation}
for every $\delta x\in T_xX\setminus  \bigcup_{j\in J'}\text{ker}\, \bigl( a_i'(x)-b_j'(x)\bigr)$ and every $j\in J'$.  But by Lemma \ref{lemB}, given $\delta x\in T_xX\setminus  \bigcup_{j\in J'}\text{ker}\, (F-b_j)'(x)$,  we find some $j\in J'$  such that
$(F-a_i)'(x)\cdot \delta x=(F-b_j)'(x)\cdot \delta x$,
contradicting \eqref{eqBB}.  Consequently, for given $i\in I'$, there exists  an index $j\in J'$ such that
$[(F-a_i)'(x)]=[(F-b_j)'(x)]$.  By the same token,  for  given $j\in J'$, there exists an index $i\in I$ such that
$[(F-b_j)'(x)]=[(F-a_i)'(x)]$. It follows that the number of equivalence classes is the same, that is, $n=m$. The proof of Proposition \ref{linearlambda} is complete.
\end{proof}

\begin{defn}\label{newlinearization}
The linearization $F_{\Lambda}'(x)$ at the solution $x$  with respect to the multisection $\Lambda$ is called {\bf surjective}, {\bf in good position}, or {\bf in general position}, if the representatives of each of the equivalence classes
$$[(F-a_{i_1})'(x)], \, [(F-a_{i_2})'(x)],\, \ldots \, ,[(F-a_{i_n})'(x)].$$
are surjective, in good position, or in general position.
\end{defn}

In view of Proposition \ref{linearlambda}, we can reformulate  Definition \ref{DER1} independently of the choice of the local section structure of  the multisection $\Lambda$ as follows.

\begin{defn}\label{DER11}  Let $p:W\to Z$ be a strong polyfold bundle and let $(f, \lambda)$ be a pair in which $f$ is a
Fredholm section of $p$ and
$\lambda:W\to \Q^+$ an $\ssc^+$-multisection on $p$. Assume that  the strong bundle $P:E\to X$ over the ep-groupoid $X$ is a model representing the bundle $p$ and let   the Fredholm section $F:X\to E$ of the bundle $P$
and the $\ssc^+$--multisection $\Lambda :E\to \Q^+$ of  $P$ be   representatives of $f$ and $\lambda$.
\begin{itemize}
\item[(1)]  The pair  $(f,\lambda)$ is called a  {\bf transversal pair} if for every $x$ satisfying  $\Lambda(F(x))>0$, the linearization $F'_{\Lambda}(x)$  is  surjective.
\item[(2)]   The pair   $(f,\lambda)$ is in {\bf good position} if  for every $x$ satisfying
 $\Lambda(F(x))>0$, the linearization $F'_{\Lambda}(x)$ is surjective and in good position to the corner structure of $X$.
\item[(3)]  The pair   $(f,\lambda)$ is in {\bf general  position to the boundary} $\partial Z$  if  for every $x$ satisfying
 $\Lambda(F(x))>0$, the linearization $F'_{\Lambda}(x)$ is surjective and in general position to
the boundary $\partial X$.
\end{itemize}
\end{defn}

We note that the actual  choice of the representing pair
$(F,\Lambda)$ in the previous definition is irrelevant. The
condition of being in good position is very important  and we refer to  \cite{HWZ3} for a comprehensive discussion.

\begin{thm}\label{trans1}
Let $p:W\rightarrow Z$ be a strong polyfold bundle in which  $\partial
Z=\emptyset$ and let  $f$ be  a proper Fredholm section of $p$.
Assume that the pair  $(f,\lambda)$ is transversal and that the solution
set $S=S(f,\lambda)=\{z\in Z\ |\ \lambda(f(z))>0\}$ is
compact. Then  the pair $(S,\lambda_f)$ carries
in a natural way the structure of a compact branched suborbifold of $Z$
without boundary. If $f$ is oriented,  then  the branched suborbifold $S$ is oriented.
\end{thm}

In the case with boundary we have the following result.
\begin{thm}\label{trans2}
Let $p:W\rightarrow Z$ be a strong polyfold bundle  and let  $f$ be  a proper Fredholm section of $p$. Assume that  the pair $(f,\lambda)$ is in good position and the solution set $S=S(f,\lambda)=\{z\in Z\ |\ \lambda(f(z))>0\}$ is compact. Then the pair $(S,\lambda_f)$ carries
in a natural way the structure of a compact branched suborbifold of $Z$  with boundary with corners.
Moreover, if   $f$ is oriented,  then the  branched suborbifold $(S, \lambda_f)$ is oriented.
\end{thm}
\begin{proof}
Let $(E, \Gamma, \gamma )$ be a strong bundle structure for $p$ in which $P:E\to X$ is a strong bundle over the ep-groupoid $X$ and  $\Gamma:\abs{E}\to W$ is the  homeomorphism covering the  homeomorphism  $\gamma:\abs{X}\to Z$.  We  choose  a Fredholm section $F$ of the bundle $P$ representing the section $f$ and an  $\ssc^+$-multisection $\Lambda:E\to \Q^+$ of the bundle  $P$ representing the $\ssc^+$-multisection $\lambda$ of the strong polyfold bundle $p$.
The $\ssc^+$-multisections $\lambda$ and $\Lambda$ are related as follows,
$$\lambda (w)=\Lambda (e)$$
where $w=\Gamma (\abs{e})$.
We define the functor $\Theta:X\to \Q^+$ by
$$\Theta (x)=\Lambda (F(x))\quad \text{ for $x\in X$}$$
 and recall that $\lambda_f (z)=\lambda (f(z))$ for $z\in Z$.
We claim that $\Theta:X\to \Q^+$ is a branched ep-subgroupoid of the ep-groupoid $X$ with boundary with corners.

To see this we take $x\in X$ such that $\Lambda (F(x))>0$. In view of the definition of an $\ssc^+$-multisection,  there exist an open neighborhood $U$ of $x$ in $X$,  finitely many local $\ssc^+$-sections $s_i$ for $i\in I$,  and positive rational  weights
$(\sigma_i)_{i\in I}$   so that
\begin{equation*}
\sum_{j=1}^k\sigma_j=1\quad \text{and}\quad \Lambda (F(x ))=\sum_{\{j\vert
F(x)=s_j(x)\}}\sigma_j.
\end{equation*}
The sum over the empty set is equal to $0$. Since $\Lambda (F(x))>0$  there exists at least one index $i\in I$ such that $F(x)=s_i(x)$.  In view of the regularizing property of  $F$ and the fact that $s_i$ is an  $\ssc^+$-section, it follows that the point $x$ is smooth. Hence the support of $\Theta$, defined by  $\supp \Theta =\{x\in X\vert \, \Theta (x)>0\}$, is contained in $X_{\infty}$. By our assumption,  for every $i\in I$ for  which the point $x$ solves $F(x)=s_i(x)$, the linearization $(F-s_i)'(x)$ is surjective and its kernel is in good position to the cornel structure of $X$. Consequently,
in view of Theorem 5.22 in \cite{HWZ3},  the solution set
$$M_i:=\{y\in U\vert \, (F -s_i)(y)=0\}$$
is  a smooth manifold with boundary with corners. If $F(x)-s_i(x)\neq 0$, then we set $M_i=\emptyset$.  Hence
$$\supp \Theta \cap U=\bigcup_{i\in I}M_i.$$
In addition,
$$
\Theta(y)= \Lambda(F(y))=\sum_{\{i\in I\ |\
F(y)=s_i(y)\}}\sigma_i=\sum_{\{i\in I\ |\ y\in M_i\}}\sigma_i
$$
for every $y\in U$. Since $U$ can also be  taken so that the isotropy group $G_x$ acts  of  $U$ by  sc-diffeomorphisms of  $U$, we have proved our claim,  that $\Theta=\Lambda \circ F:X\to \Q^+$ is a branched ep-subgroupoid of the ep-groupoid $X$.

Let $z\in S$, i.e., $\lambda (f(z))>0$.
Since, by the definition of the overhead,  $f(z)=\Gamma \circ \abs{F}\circ \gamma^{-1}(z)=\Gamma \circ \abs{F}(\abs{x})=\Gamma (\abs{F(x)})$, it follows  that $\lambda (f(z))=\Lambda ( F(x))=\Theta (x)$,  where $z=\gamma (x)$. Hence, $x\in X_{\infty}$. We conclude that  $S=\gamma (\abs{\supp \Theta})$ and that  if $z=\gamma (\abs{x})$, then $\lambda_f (z)= \Lambda (F(x))=\Theta (x)$.  Moreover, the map $\gamma:\abs{\supp \Theta }\to S$ is a homeomorphism if both $\abs{\Theta}\subset \abs{X}$ and $S\subset Z$ are
equipped with the  induced topologies.

For orientation questions  we make use
of the results in \cite{HWZ8}. The relevant facts are briefly  summarized  in Appendix \ref{oranddet}. Recall that an orientation for $f$
is given by an orientation of  the determinant bundle $\text{DET}(F)\rightarrow X_\infty$  defined by the projection map  $\text{DET}(F, x)\mapsto x$ introduced in Appendix \ref{oranddet}.  In contrast to the usual definition of a determinant bundle $\det(F)\rightarrow X$  the fiber over a smooth point $x$ consists of a convex family of linear sc-Fredholm operators
which differ by $\ssc^+$-operators and are obtained as linearizations (These are not unique except at solutions!).
Since locally the multisection is represented by $\ssc^+$-sections
we conclude from  \cite{HWZ8} that the
local solution sets have natural orientations compatible with the morphisms. This completes the proof of Theorem \ref{trans2}.
\end{proof}

Next we shall prove parts of Theorem \ref{poil}. We consider a  strong polyfold bundle $p:W\rightarrow Z$ and two proper oriented
pairs   $(f_0,\lambda_0)$ and $(f_1,\lambda_1)$ in general position, where $f_j:Z\to W$ are proper Fredholm sections and where $\lambda_j:W\to \Q^+$ are  $\ssc^+$-multisections of $p$. Abbreviating by $\pi:[0,1]\times Z\to Z$ the projection onto the second factor we denote by $\pi^*(p)$ the strong polyfold pull-back bundle over $[0,1]\times Z$. If $t\mapsto \lambda_t$ is an  sc-smooth homotopy of $\ssc^+$-multisections of the bundle $p$ connecting the $\ssc^+$-multisectios $\lambda_0$ with $\lambda_1$, we consider the pair $(\wh{f}, \wh{\lambda})$ consisting of the Fredholm section $\wh{f}$ and the $\ssc^+$-multisection $\wh{\lambda}$ of $\pi^*(p)$ defined by
\begin{align*}
\what{f}(t,z)=f_t(z)\quad&\text{for $(t, z)\in [0,1]\times Z$}\\
\what{\lambda}(t,z)=\lambda_t(z)\quad&\text{for $(t, z)\in [0,1]\times W$}.
\end{align*}
Now we assume that   the pair $(\what{f}, \what{\lambda})$  is in general position and that the solution set
$S=\{(t, z)\in [0,1]\times Z\vert \, \wh{\lambda}(t, \wh{f}(t, z))>0\}$ is compact. Then by Theorem \ref{trans2} the solution set carries the structure of a branched compact suborbifold of $[0,1]\times Z$ whose boundary is in good position to $\partial ([0,1]\times Z)$. Introducing the weighting function $w(t, z)=\wh{\lambda}(t, \wh{f}(t, z))$ on $S$, the pair $(S, w)$ contains two obvious boundary pieces, namely $(S_1, w_1)$ associated with the pair $(f_1, \lambda_1)$ and $-(S_0, w_0)$ associated with $(f_0, \lambda_0)$. Here we have to take the minus sign if we equip $(S_0, w_0)$ with the orientation coming from $(f_0, \lambda_0)$ by using the obvious orientation convention for the family $f_t$ as explained in \cite{HWZ8}.
These two pieces can be identified as part of the solution space contained in the two faces of $[0,1]\times Z$ defined by $\{i\}\times Z$ for $i=0,1$. There is another boundary piece $\partial S$ which lies in faces of the form $[0,1]\times (\text{face in $Z$})$ intersecting $S_0$ and $S_1$ only in points of degeneracy at least $2$ (with respect to the degeneracy index of the polyfold $[0,1]\times Z$) and which is denoted by $(\wh{\partial}S, w)$. The boundary $(\partial S, w)$ has a natural orientation, it is a branched suborbifold only after having removed the points of degeneracy  at least $2$ which is a closed set of measure $0$.

If the pair $(\omega,\vartheta)$ represents a cohomology class in $H_{dR}(Z,\partial Z)$ and
$j:\partial Z\rightarrow Z$ is  the inclusion map, then
$$
d\omega =0\quad \mbox{and}\quad  j^\ast\omega =d\theta.
$$
\begin{lem}
For an oriented and proper pair in general position
the integration map
$$
[\omega,\vartheta]\rightarrow
\int_{(S,w)}\omega-\int_{(\partial S,w)}\vartheta
$$
defines a linear map on the deRham  cohomology group $H^\ast_{dR}(Z,\partial Z)$ which is an
invariant under nice homotopies.
\end{lem}
\begin{proof}
We compute,  assuming that $\omega$ is of degree $n$ and $S_0$ and $S_1$
are of dimension $n$,
\begin{equation*}
\begin{split}
&\biggl[ \int_{(S_1, w_1)}\omega - \int_{(\partial S_1, w_1)}\theta \biggr]-\biggl[ \int_{(S_0, w_0)}\omega - \int_{(\partial S_0, w_0)}\theta \biggr]\\
&\phantom{\quad}=
\int_{(S_1,w_1)}\omega-\int_{(S_0,w_0)}\omega
+\int_{(\what{\partial}S,w)}\omega-\int_{(\partial S_1,w_1)}\vartheta
+\int_{(\partial S_0,w_0)}\vartheta-\int_{(\what{\partial}S,w)}\omega\\
&\phantom{\quad}=\int_{(\partial S,w)}\omega-\int_{(\partial S_1,w_1)}\vartheta
+\int_{(\partial S_0,w_0)}\vartheta-\int_{(\what{\partial} S,w)}\omega\\
&\phantom{\quad}=\int_{(S,w)} d\omega
-\int_{(S_1,w_1)}d\vartheta+\int_{(S_0,w_0)}d\vartheta-
\int_{(\what{\partial} S,w)}\omega\\
&\phantom{\quad}=-\int_{(S_1,w_1)} d\vartheta+ \int_{(S_0,w_0)}d\vartheta-
\int_{(\what{\partial} S,w)}d\vartheta\\
&\phantom{\quad}=-\int_{(\partial S,w)} d\vartheta= 0.
\end{split}
\end{equation*}

The last integral vanishes   since integration of a global sc-form over the boundary
$(\partial S,w)$ is like integration over a closed manifold, so that  in particular
the integral of an exact form vanishes.
\end{proof}

The perturbation theory below guarantees such  nice
homotopies referred to in lemma, and will then complete the proof of  Theorem \ref{poil}.

\subsection{Perturbation Results}

In this section we shall adapt the techniques introduced for M-polyfolds in \cite{HWZ3} to the functorial setting and show how given sections can be made transversal by small perturbations.

We first recall the kind of problems studied in
\cite{HWZ3}. In there we consider a proper Fredholm section of the strong M-polyfold bundle $P:E\to X$ and want to bring the compact set of solutions of $F(x)=0$ into a general position by a small perturbation section $s$, i.e., we study the solution set of $F(x)-s(x)=0$. To do so one  first  constructs  finitely many
$\ssc^+$-sections $s_j$, $j=1,\ldots ,k$, so that they fill up the cokernel of the linearizations $F'(x)$ at all solutions $x$ of $F(x)=0$ and then considers the parametrized proper Fredholm section $F(t, x)=F(x)+\sum_{j=1}^kt_js_j(x)$ for $(t, x)\in \R^k\times X$ with small parameters $t=(t_1, \ldots ,t_k)$. It has the property that its linearizations $F'(0, x)$ at $t=0$ and the solutions $x\in X$ of $F(x)=0$ are surjective. Assuming that $\partial X=\emptyset$ one concludes that the set $M=\{(t, y)\in \R^k\times X\vert \, \text{$F(t, x)=0$ and $\abs{t}$ sufficiently small}\}$ of the parametrized Fredholm section is a smooth manifold. Then  the regular values of the projection $M\to \R^k$ given by $(t, y)\mapsto  t$ give the parameter value $t^*=(t^*_1, \ldots ,t^*_k)$ near $t=0$ for which the perturbation $s=\sum_{j=1}^kt^*_js_j$ has the desired properties in order to conclude that the solution set $\{x\in X\vert \, F(t^*, x)=0\}$ of the perturbed problem is a compact smooth manifold.

In the case of $\partial X\neq \emptyset$, one has to add enough $\ssc^+$-sections so that also  the kernels of the linearizations $F'(x)$ at the solutions $x$ of $F(x)=0$ are in good position to the boundary $\partial X$. Then the corresponding solution set $M$ of the parametrized Fredholm section for small parameter values is a smooth manifold with boundaries with corners.  Again one looks at the projection map $M\to \R^k$ defined by
$(t, y)\mapsto t$. This map can be restricted to the local faces (which are smooth manifolds with boundaries with corners) and one finds small regular values $t^*$ for a finite number of problems. For these parameter values the solution set
$\{x\in X\vert \, F(t^*, x)=0\}$ of the perturbed problem is in general position to the boundary $\partial X$ and is, therefore, a compact manifold with boundaries with corners.

In the multivalued case we shall have to consider  locally a finite number of problems of the kind just described and,   using a compactness and a covering argument,   in total a finite number of such problems. This way  we shall obtain finitely many finite dimensional submanifolds (perhaps with boundary with corners) and look again at the projection map
$(t, y)\mapsto t$. A common regular value $t^*$ which exists by the Sard theorem gives rise to an $\ssc^+$-multisection $\Lambda^{t^*}$ which  is regular, so that the pair $(F, \Lambda^{t^*})$ is in general position and gives rise to the branched ep-subgroupoid $\Lambda^{t^*}\circ F:X\to \Q^+$ of the ep-groupoid $X$.

Now we assume that $\lambda$ is an $\ssc^+$-multisection on the strong bundle $p:W\to Z$ and let $N:W_{0,1}\to [0,\infty )$ be an auxiliary norm on $p$.  We choose an $\ssc^+$-multisection $\Lambda:E\to \Q^+$ on the strong bundle $P:E\to X$ representing $\lambda$ and let  $N^*:E_{0,1}\to [0,\infty )$ be an auxiliary  norm on $P$ representing $N$.
Then we define the  auxiliary norm  $N(\lambda)$ of the multisection $\lambda$ as follows.

We start by defining the auxiliary norm $N^*(\Lambda )$  of the $\ssc^+$-multisection $\Lambda$.
For every $x\in X$, there exists an open neighborhood $U_x$ of $x$ on which there is a local section structure $(s_i)_{i\in I}$ consisting of $\ssc^+$-sections and the associated  set of positive rational numbers $(\sigma_i)_{i\in I}$ so that
$\Lambda (e)=\sum_{\{i\vert \, s_i(Pe)=e\}}\sigma_i$ and we define
$$N^*(\Lambda )(y)=\max_{i\in I}N^* (s_i(y)),\quad y\in U_x.$$
This definition is independent of the choice of the local section structure. Indeed, if $(t_j)_{j\in J}$ is another local sections structure on $U_x$ defining the multisection  $\Lambda$ and  $y\in U_x$, then for every $i\in I$   there exists  an index $j\in J$ such that $s_i (y)=t_j(y)$. Conversely,   for every $j\in J$ there is an index $i\in I$ such that $t_j(y)=s_i(y)$. Hence $\max_{i\in I}N^* (s_i(y))=\max_{j\in J}N^* (t_i(y))$ as claimed.  Since the morphisms extend to local sc-diffeomorphisms, it follows that $N^*(\Lambda)$ is invariant under morphisms, that is, $N^*(\Lambda )(x)=N^*(\Lambda )(x')$  if there exists a morphism  $\varphi:x\to x'$. Consequently we define  the {\bf auxiliary norm of the $\ssc$-multisection $\lambda$}  by
$$N(\lambda )(z)=N^*(\Lambda )(x), \quad \text{where $z=\abs{x}$}.$$

\begin{lem}\label{lemcomactA}
Assume that $f:Z\to W$ is a proper Fredholm section of the strong bundle $p:W\to Z$  and that the pair $(U, N)$ controls compactness. Then given an $\ssc^+$-multisection $\lambda$ satisfying $N(\lambda )\leq 1$, the solution set
$$S(f, \lambda )=\{z\in Z\vert \, \lambda (f(z))>0\}$$ is compact.
\end{lem}
\begin{proof}
It suffices to prove the result in local models.  Hence we  assume that $P:E\to X$ is a local model for $p$ and that  $F:X\to E$ is a proper Fredholm section representing $f$ and we assume that $\Lambda :E\to \Q^+$ is an $\ssc^+$-multisection representing $\lambda$.  Then   $N^*(\Lambda )=N(\lambda )\leq 1$. Set $U^*=\pi^{-1}(U)$.
The result will  follow by showing that the solution set
$$S(F, \Lambda )=\{x\in X\vert \, \Lambda (F(x))>0\}$$
is compact.
To see this  we  take a sequence $(x_n)\subset S(F, \Lambda )$. Then $\Lambda (F(x_n))>0$. Hence given
the point $x_n$,  there exists a neighborhood $U_{x_n}$ and a local section structure $(s_i)_{i\in I}$ together with associated weights $(\sigma_i)_{i\in I}$ so that $\Lambda (F(x_n))=\sum \sigma_i$ where the sum is taken over the indices $i$ for which $F(x_n)=s_i(x_n)$. From  $\Lambda (F(x_n))>0$   one concludes $F(x_n)=s_i(x_n)$ for some $i$,  and from  $N^*(\Lambda )\leq 1$ one  concludes  $N^*(F(x_n))\leq 1$ for all $n$.
By assumption, the pair $(U, N)$ controls compactness  so that by Lemma \ref{lemcontrol}
there is  a sequence of  points $(y_n)\in U^*$ and  a sequence of morphisms $\varphi_n:x_n\to y_n$ such that $(y_n)$ contains  a converging subsequence.
We assume without loss of generality that $y_n\to y$. In an open neighborhood $U_y$ of $y$ there is a local section structure $(t_j)_{j\in J}$ with associated weights $(\tau_j)_{j\in J}$ so that $\Lambda (F(z))=\sum_{\{ j\vert t_j(z)=F(z)\}}\tau_j$ for $z\in U_y$.
For large $n$,  $x_n\in U_y$ and since the index set $J$ is finite, $F(x_n)=t_{j_0}(x_n)$   for some fixed  index  $j_{0}\in J$  and some subsequence,  denoted again by $(x_n)$. Hence $F(y)=t_{j_0}(y)$ showing that $\Lambda (F(y))>0$.  Moreover, $y\in U^*$ since the support of $\Lambda $ is contained in $U^*$.  The proof of the lemma is complete.
\end{proof}

We are ready to prove the perturbation result in the case of no boundary, $\partial Z=\emptyset$. In order to have sc-smooth functions and sc-smooth sections with supports in preassigned open sets available  (Appendix \ref{partition}), we shall assume in the following that the {\bf sc-structures of the ep-groupoids $X$ used as models of the polyfold $Z$ are based on sc-separable Hilbert spaces.}

\begin{thm}\label{ghjl}
Let  $p:W\rightarrow Z$ be  the strong polyfold bundle having empty boundary
and let  $f$ be  a proper (oriented) Fredholm section of  $p$.
We assume that the sc-smoothness  structure of the polyfold $Z$ is based on separable Hilbert spaces. Fix  an auxiliary norm $N$ for $p$ and an open neighborhood $U$ of the solution set $S(f)=\{z\in Z\vert \, f(z)=0\}$  so that the pair $(U, N)$ controls compactness. Assume that $0<\varepsilon<\frac{1}{2}$. Then for
every $\ssc^+$-multisection $\lambda$ having  its  support  in $U$ and satisfying
$N(\lambda)<\frac{1}{2}$,  there exists an $\ssc^+$-multisection $\tau:W\to \Q^+$ supported in $U$ and satisfying $N(\tau)<\varepsilon$
so that  the pair $(f,\lambda\oplus \tau)$ is a transversal pair. In
particular,  the associated solution set
$$S(f,\lambda\oplus \tau)=\{z\in Z\vert \, (\lambda \oplus \tau)(f(z))>0\}$$
 is
an (oriented) compact  branched suborbifold of $Z$ without boundary.
\end{thm}
\begin{proof}
Take an $\ssc^+$-multisection $\lambda:W\to \Q^+$ satisfying $N(\lambda )<\frac{1}{2}$.
In view of Lemma \ref{lemcomactA}, the solution set
$$S(f, \lambda  )=\{z\in Z\vert \, \lambda (f(z))>0\}$$
is  compact subset of the polyfold $Z$.

We choose a strong polyfold bundle $P:E\to X$ over  the ep-groupoid $X$  as a  model for $p:W\to Z$ and  let the proper Fredholm section $F:X\to E$ of the bundle $P$  represent the proper Fredholm section $f$ of $p$. We may assume without loss of generality
that
$$Z=\abs{X}\quad \text{and}\quad W=\abs{E}.$$

Assume that the $\ssc^+$-multisection $\Lambda:E\to \Q^+$ of  $P$ represents the
 $\ssc^+$-multisection $\lambda$ of $p$.  If $\pi:X\to \abs{X}$ denotes the quotient map onto the orbit space, the  set $U^{*}=\pi^{-1}(U)$ is an open neighborhood of the compact solution set
$$S(F, \Lambda )=\{x\in X\vert \, \Lambda (F(x))>0\},$$
and since $N(\lambda )=N^*(\Lambda )$, we have $N^*(\Lambda )<\frac{1}{2}.$

 We fix a solution $x\in S(F,\Lambda)$ and take an open neighborhood $U_x\subset X$ of $x$ on which the isotropy group $G_x$ acts by its natural representation $\varphi_g \in \text{Diff}_{\text{sc}}(U_x)$. The local system of sections of $U_x$ we shall denote by $(a_i)_{i\in I}$ and the associated weights by $(\sigma_i)_{i\in I}$. Then, if $y\in U_x$,
 $$
 \Lambda (F(y))=\sum_{\{i\in I\vert \, F(y)=a_i(y)\}}\sigma_i.
 $$
Moreover,  $\Lambda (F(y))=0$ if there is no index $i\in I$ satisfying $F(y)=a_i(y)$. By assumption on the solution, $\Lambda (F(x))>0$,  so that there is a subset $J\subset I$ such that
 $$F(x)-a_i(x)=0\quad \text{for all $i\in J$}.$$

 Recall that for $g\in G_x$ the map $\mu(g, \cdot ):E_x\to E_x$ is a linear sc-isomorphism. Therefore, we can choose smooth linearly independent vectors
 $e^1,\ldots, e^m\in E_x$  such that for every $g\in G_x$ the vectors
 $$\mu (g, e^1), \ldots ,\mu (g, e^m)$$
 in $E_x$ span the cokernels of the linearizations
 $$(F-a_i)'(x)$$
 for every $i\in I'$. If $g=\text{id}\in G_x$ is the identity element of the isotropy group $G_x$, then $\mu (\text{id}, e^j)=e^j$ for $1\leq j\leq m$.

By means of Lemma \ref{lem7.9.4}, we find  $\ssc^+$-sections   $s^1,\ldots ,s^m:U_x\to E\vert U_x$  having their supports  in $U_x$ and satisfying $s^j (x)=e^j$ for $1\leq j\leq m$.

For every $g\in G_x$, we next define
the $\ssc^+$-sections $s_g^j:U_x\to E\vert U_x$ by
$$
s_g^j(\varphi_g (y)):=\mu (\Gamma (g,y), s^j(y)),
$$
where $y\in U_x$ and $1\leq j\leq m$.  Introducing  the sum
$$s_{g}^t( y):=\sum_{j=1}^{m}t_j s_g^j(y),$$
where  $y\in U_x$ and $t=(t_1, \ldots ,t_n)\in \R^n$, we have defined the  $\sharp G_x$-many parametrized $\ssc^+$-sections $s^t_g:U_x\to E\vert U_x$ having their supports in $U_x$.  If $\text{id}\in G_x$ is the  identity element, then the morphism $\Gamma (\text{id}, y)$ is equal to $1_y:y\to y$ and using that $\mu (1_y, e_y)=e_y$ for all $e_y\in E_y$, the
$\ssc^+$-sections $s^t_g$ satisfy
$$s^t_g (\varphi_g(y))=\mu (\Gamma (g, y), s^t_{\text{id}}(y))$$
with
$$s^t_{\text{id}}(y)=\sum_{j=1}^mt_js^j(y).$$
The map
$$\Lambda^t_{U_x}:E\vert U_x\to \Q^+$$
is  defined by
$$\Lambda^t_{U_x}(e)=\dfrac{1}{\sharp G_x}\cdot \sharp \{g\in G_x\vert \, s_{g}^t(P(e))=e\}.$$
Proceeding as in Proposition \ref{prop7.9.6} one extends  the local multisection
$\Lambda^t_{U_x}$ from $E\vert U_x$ to the parameterized $\ssc^+$-multisection
$$\Lambda_x^t:E\to \Q^+.$$

Continuing with the proof of Theorem \ref{ghjl}, we introduce for $i\in J$ and $g\in G_x$ the perturbed sections $F^i_g:\R^m\oplus U_x\to E\vert U_x$  by
$$F^i_g(t, y)=F(y)-a_i(y)-\sum_{j=1}^mt_js_g^j(y),$$
where  $t=(t_1, \ldots,t_m)\in \R^m$.
In the special case $g=\text{id}\in G_x$ one has
$$F^i_{\text{id}}(t, y)=F(y)-a_i (y)-\sum_{j=1}^mt_js^j(y).$$
In view of Theorem 3.9 in \cite{HWZ3},  the perturbed sections $F_g^i$ are proper Fredholm sections of the bundle $E^1\to \R^m\oplus X^1$ over the set $\R^m\oplus (U_x)^1$.
\begin{lem}\label{lemAA}
For every $i\in J$ and $g\in G_x$ the linearization
$$DF_g^i(0, x):T_0\R^m\oplus T_xX\to E_x$$
of the map $(t, y)\mapsto F^i_g(t, y)$ at the special point $(0, x)$ at which $F(x)-a_i(x)=0$, is surjective.
\end{lem}
\begin{proof}
The linearization of the map $(t, y)\mapsto F^i_g (t, y)$ is equal to
\begin{eqnarray*}
&&
DF^i_g(t, y)\cdot (\delta t, \delta y)\\
&=&(F'(y)-a_i'(y))\cdot \delta y-\sum_{j=1}^m\delta t_j\cdot s^j_g(y)-\sum_{j=1}^mt_j\cdot D_ys^j_g(y)\cdot \delta y,
\end{eqnarray*}
where $\delta t=(\delta t_1,\ldots ,\delta t_m)\in \R^m$ and $\delta y\in T_xX$. At the point $y=x$ we have,  using $\varphi_g (x)=x$, that $s^j_g (x)=\mu (\Gamma (g, x), s^j(x))=\mu (g, s^j (x))=\mu (g, e^j)$. Therefore the linearization at $(t, y)=(0, x)$ is represented by the linear map
$$
DF^i_g(0,x)\cdot (\delta t, \delta y)=(F'(x)-a_i'(x))\cdot \delta y-\sum_{j=1}^m\delta t_j\cdot \mu (g, e^j),$$
which, in view of the definition of $\mu (g, e^j)$ for $1\leq j\leq m$, is obviously surjective as claimed in the lemma.
\end{proof}

Since by Lemma \ref{lemAA} the linearizations $DF^i_g(0, x)$ at the point $(0, x)$ are surjective it follows,  by the arguments in   section 4.2 in \cite{HWZ3}, that  the linearizations $DF^i_g(t, y)$ at the points  solving $F^i_g (t, y)=0$ are also surjective if  $\abs{t}$ is small and $y$ belongs to a possibly smaller invariant neighborhood $V_x\subset U_x$ of the distinguished point $x$.

By assumption the solution set  $S(f, \lambda )$  of the proper Fredholm section $f$ of $p$ is compact. Consequently, there exist finitely many solutions $x_1, \ldots ,x_m$ belonging to the solution set  $S=S(F, \Lambda)$ so that for the corresponding open neighborhoods $V_{x_1}, \ldots ,V_{x_m}$ in $X$,
the sets $\abs{V_{x_1}}, \ldots ,\abs{V_{x_m}}$ in the orbit space $\abs{X}$ cover
$f^{-1}(0)$.  Then we define the parametrized $\ssc^+$-multisection  $\Lambda_t$ as the sum
$$\Lambda_t=\Lambda^{t_1}_{x_1}\oplus \cdots \oplus \Lambda^{t_m}_{x_m}$$
where $t=(t_1, \ldots ,t_m)\in \R^{n_1}\times \cdots \R^{n_m}=\R^{N}$  with  $N=n_1+\cdots +n_m$. By construction, at every point $(t, x)=(0, x) \in \{0\}\times S$, for the new parametrized local branching structure at  $x\in X$, we have surjectivity of the linearizations of every local problem, so that the set
$M=\{(t, x)\in \R^N\times X\vert \, \text{$(\Lambda\oplus \Lambda_t)(F(x))>0\}$ and $\abs{t}$ small }\}$ is locally a collection of finitely many finite dimensional submanifolds of $X$. The projection maps $M\to \R^N$ given by $(t, x)\mapsto t$ can be viewed as finitely many maps defined on smooth submanifolds so that it makes sense to talk about regular values. We choose by means of the Sard's theorem a small regular value $t^*\in \R^N$. Then the pairs $(F, \Lambda_{t^*})$ and $(f, \lambda_{t^*})$ are transversal and the map $\Theta:X\to \Q^+$ defined by $\Theta (x)=(\Lambda \oplus \Lambda_{t^*})(F(x))$ is a branched ep-subgroupoid of the ep-groupoid $X$. This completes the proof of Theorem  \ref{ghjl}.
\end{proof}

We point out that,  in view of Proposition \ref{linearlambda},  the perturbation $\Lambda^t$ is independent of local section structures of the unperturbed multisection $\Lambda$.

Next we consider the case in which the polyfold $Z$ possesses a boundary $\partial Z$.
\begin{thm} \label{pertmontreal1}
Let $p:W\rightarrow Z$ be a strong polyfold bundle,
where $Z$ has possibly a boundary with corners.  We assume that the sc-smooth
structure of the polyfold $Z$ is based on separable Hilbert spaces.
Assume that $f$ is a
proper (oriented) Fredholm section of $p$, $N$ an auxiliary norm and $U$ an
open neighborhood of  the solution set $S(f)=\{z\in Z\vert \, f(z)=0\}$ so that  the pair $(N,U)$ controls compactness.
Let $\lambda:W\to \Q^+$ be an  $\ssc^+$-multisection satisfying
$N(\lambda(z))<\frac{1}{2}$ for all $z$ and supported in $U$. Then
for $\varepsilon\in (0,\frac{1}{2})$ there exists
$\ssc^+$-multisection $\tau$ satisfying $N(\tau(z))<\varepsilon$ and supported in
$U$ so that $(f,\lambda\oplus \tau)$ is in general position. In
particular, the solution set
$$S(f,\lambda\oplus\tau)=\{z\in Z\vert \, (\lambda \oplus \tau )(f(z))>0\}$$ is an (oriented)  compact branched suborbifold of the polyfold $Z$
with boundary with corners.
\end{thm}
The proof is a variation of the corresponding proof of Theorem 5.22 in \cite{HWZ3}
for Fredholm sections on M-polyfold bundles.
Before  we start proving the theorem we  recall some notations.  If $x$ is a  point of the M-polyfold $X$, we denote by $d=d(x)$ the degeneracy index from
section \ref{basicresult}.
By ${\mathcal F}^1, \ldots, {\mathcal F}^d$ we abbreviate  the local faces of $X$ at the point $x$.  Every  local face ${\mathcal F}^j$ has the  tangent
space $T_x{\mathcal F}^j$ at the point $x\in X$.  The subset $T_{x_j}^{\partial }X\subset T_xX$ is the intersection $ T^\partial_x X=\bigcap_{1\leq
j\leq d} T_x{\mathcal  F}^j. $ In the case that $x$ is an
interior point, the degeneration index $d(x)$ is equal to $0$ and we   set $T^\partial_x X=T_x X$.
\begin{proof}
Proceeding as in the proof of Theorem \ref{ghjl},  we  let $\lambda:W\to \Q^+$ be an $\ssc^+$-multisection satisfying $N (\lambda )<\frac{1}{2}$.   Since $(N, U)$ controls compactness, the solution set
$S(f,\lambda )$ is a compact subset of $\abs{X}$ in view of Lemma \ref{lemcomactA}. Set $U^*=\pi^{-1}(U)$.
Let  $\Lambda:E\to \Q^+$ be the  $\ssc^+$-multisection  representing  the $\ssc^+$-mulitsection $\lambda:W\to \Q^+$ in the strong polyfold bundle $P:E\to X$. We assume that the proper Fredholm section $F:X\to E$ of $P$ represents the proper Fredholm section $f:Z\to W$ of $p$. Again we may assume that $Z=\abs{X}$ and $W=\abs{E}$.

Fix a solution $x$ belonging to the compact solution set $S(F, \Lambda)$  and take an open neighborhood $U_x\subset X$ of  $x$ which is invariant under the action of the isotropy group $G_x$.  On the neighborhood $U_x$ we have a local system of sections $(a_i)$ for $i\in I$ with associated weights $(\sigma_i)_{i\in I}$, so that
$$\Lambda (F(x))=\sum_{\{i\in I\vert \, F(x)=a_i(x)\}}\sigma_i.$$
Because $\Lambda (F(x))>0$,  there is a subset $J$ of the index set $I$ such that
$F(x)-a_i(x)=0$ for all $i\in J$.

Next we construct the parameterized $\ssc^+$-perturbation of the section $F-a_i$ for every $i\in J$. As in the proof of  Theorem \ref{ghjl},  we choose smooth linearly independent vectors $e^1, \ldots ,e^m\in E$ with the property that for every $g\in G_x$  the vectors $\mu (g, e^1), \ldots ,\mu (g, e^m)$ span all the cokernels of the linearizations $F'(x)-a'_i(x):T_xX\to E_x$  for every $i\in J$.   By  Lemma \ref{lem7.9.4},  there are
$\ssc^+$-sections $s^1, \ldots ,s^m:U_x\to E\vert U_x$ having their supports in $U_x$ and satisfying $s^j(x)=e^j$ for $j=1, \ldots ,m$.
To achieve transversality of the kernels of the linearized perturbed Fredholm sections we choose additional $\ssc^+$-sections $s^{m+1}, \ldots ,s^{N}:U_x\to E\vert U_x$ having their supports in $U_x$. This is done as follows.  First observe that the kernel of the linear map
$$(\lambda, h)\mapsto (F-a_i)'(x)\cdot h+\sum_{j=1}^N\lambda_j\cdot s^j(x)$$
consists of  $\{0\}\oplus \text{ker}(F-a_i)'(x)$ together with the vectors $(\lambda, h)$ which are the solutions of the equation $(F-a_i)'(x)=-\sum_{j=1}^N\lambda_j\cdot s^j(x)$.
Since the space $T_x^{\partial }X$ is of finite codimension in the tangent space $T_xX$, we find  finitely many smooth linearly independent vectors $h^{m+1},\ldots ,h^{n}$ so that
$$
\text{span}\{h^{m+1},\ldots ,h^n\}\oplus T^{\partial}_xX=T_xX.
$$
With this choice of the vectors $h^{m+1},\ldots ,h^n$, we define the smooth vectors $e^{i, l,g}\in E_x$ by
$$e^{i,l, g}:=-\mu (g, (F-a_i)'(x)h^l)$$
for all  $i\in J$, $m+1\leq l\leq n$, and $g\in G_x$.
We extend, using  Lemma \ref{lem7.9.4}, these vectors to $\ssc^+$-sections  having their supports in $U_x$ and label them by $s^{m+1}, \ldots ,s^N$.

Next we introduce  the perturbed sc-smooth section $F_{i}:\R^N\oplus U_x\to E$  by
$$F_{i}(t, y):=F(y)-a_i(y)+\sum_{j=1}^Nt_j\cdot s^j(y),$$
where   $y\in U_x$ and $t=(t_1,\ldots ,t_N)\in \R^N$.
This  is a proper Fredholm section having the following additional properties.
Its  linearization  $F'_{i}(0, x)$ at the point $(0, x)$  is surjective and  its kernel  is transversal to $\R^N\oplus T^{\partial}_xX$ in $\R^N\oplus T_xX$ for every $i\in J$.  Moreover,  for every subset $\tau$ of the set $\{1,\ldots ,d(x)\}$, the linearization $F_i'(0, x) $ restricted to the tangent space $T_{(0,x)}(\R^N\oplus \bigcap_{j\in \tau}{\mathcal F}^j)$ is surjective and the kernel of this restriction is transversal to the subspace $T^{\partial}_{(0,x)}(\R^N\oplus \bigcap_{j\in \tau}{\mathcal F}^j)$.

For $g\in G_x$ and $1\leq j\leq N$,  we define the $\ssc^+$-section $s_g^j:U_x\to E\vert U_x$ by
$$s^j_g (\varphi_g (y)):=\mu (\Gamma (g, y), s^j (y))$$
and the $\ssc^+$-section $s^t_g$ by
$$s^t_{ g}(y):=\sum_{j=1}^Nt_j\cdot s_g^j(y)$$
for $y\in U_x$ and $t=(t_1, \ldots ,t_N)\in \R^N.$
Hence we have  the $\sharp G_x$-many parameterized $\ssc^+$-sections $U_x\to E\vert U_x$ having supports in $U_x$,  which  define the map
$\Lambda_{U_x}^t:E\vert U_x\to \Q^+$ by
$$\Lambda^t_{U_x}(e)=\dfrac{1}{\sharp G_x}\cdot \sharp \{g\in G_x\vert s^t_{g}(P(e))=e\}.$$
As in Proposition  \ref{prop7.9.6} we extend the local $\ssc^+$-multisection $\Lambda^t_{U_x}$ from $E\vert U_x$ to the parametrized $\ssc^+$-multisection
$$\Lambda_x^t:E\to \Q^+.$$
If $i\in J$ and  $g\in G_x$, we introduce the perturbed section $F_{ g}^i:\R^N\times U_x\to E$ by
$$F^i_{g}(t, y)=F(y)-a_i(y)+s^t_{g}(y).$$
It  is a proper Fredholm section which as we show next has the same properties as the section $F_i$.

\begin{lem}\label{lempert00}
The linearization $DF^i_{g}(0, x):\R^N\oplus T_xX\to E_x$ is surjective and its kernel is transversal to $\R^N\oplus T_x^{\partial}X$ in $\R^N\oplus T_xX$. Moreover,  for every subset $\tau$ of the set $\{1,\ldots ,d(x)\}$, the restriction of  $DF^i_{g}(0, x) $ to the tangent space $T_{(0,x)}(\R^N\oplus \bigcap_{j\in \tau}{\mathcal F}^j)$ is surjective and the kernel of this restriction is transversal to the subspace $T^{\partial}_{(0,x)}(\R^N\oplus \bigcap_{j\in \tau}{\mathcal F}^j)$.
\end{lem}
\begin{proof}
Fix $i\in J$ and $g\in G_x$. In view of the fact that $\text{span}\{h^{m+1},\ldots ,h^n\}\oplus T^{\partial }_xX=T_xX$, it suffices to show that for every vector $h^l$, there is a point $\lambda^l\in \R^N$ such that $(\lambda^l, h^l)\in \R^N\oplus T_xX$ belongs to the kernel of the linearization $(F^i_g)'(0, x)$.  Recall that vectors $e^{i, l,h}$ for $i\in J$, $m+1\leq l\leq n$ and $h\in G_x$ are denoted as $s^j(x)$ with $m+1\leq j\leq N$. Hence given $m+1\leq l\leq n$  there is an index $m+1\leq j\leq N$ such that $e^{i, l, g^{-1}}=s^j(x)$.
By definition,   $e^{i,l, g^{-1}}=\mu (g^{-1}, (F-a_i)'(x)h^l)$ so that
\begin{equation*}
\begin{split}
-s_g^j(x)&=-\mu (g, s^j (x))=\mu (g,\mu (g^{-1}, (F-a_i)'(x)h^l))\\
&=\mu (g^{-1}\circ g,  (F-a_i)'(x)h^l)=(F-a_i)'(x)h^l.
\end{split}
\end{equation*}
Hence  if $\lambda^l\in \R^N$ is defined by $\lambda^l_k=1$ if $k=j$ and
$\lambda^l_k=0$ for $k\neq j$, the pair $(\lambda^l, h^l)$ belongs to the kernel of the linearizations
$(F^i_g)'(0, x)$ as claimed.
 The remaining part of the lemma is proved in a  similar way.
 \end{proof}

\begin{lem}\label{lemmontreal1}
There exist a $G_x$-invariant open neighborhood $V_x\subset U_x$ and a parametrized $\ssc^+$-multisection $\Lambda^t_x:E\to \Q^+$ which is supported in $U_x$ and linear in $t$ such that for $\varepsilon>0$ sufficiently small the solution set
$$S_{x,\varepsilon}=\{ (t, y)\in \R^N\times X^1\vert \, \text{$(\Lambda \oplus \Lambda^t_x)(F(y))>0$, $\abs{t}<\varepsilon, $\text{and} $y\in V_x$}\}$$
has the following properties. For every $i\in I$ and $g\in G_x$ and $(t, y)\in S_{x,\varepsilon}$ solving the equation
$$F^{i}_{g}(t, y)=F(y)-a_i(y)+s^t_{g}(y)=0$$
the following holds true.
\begin{itemize}
\item[(1)]  The linearization $DF^i_{g}(t, y)$  is  surjective
\item[(2)] The kernel of the linearization $DF^i_{g}(t,y)$ is transversal to the subspace $T^{\partial }_{(t, y)}(\R^N\times X)$ of the tangent space $T_{(t, y)}(\R^N\times X)$
\item[(3)] For every subset $\sigma$ of $\{1,\ldots d(y)\}$,
the  linearization of $DF^i_{g}(t, y)$  restricted to the tangent space ${T_{(t
,y)}(\R^N \times  \bigcap_{j\in \sigma} \mathcal F}^j)$ is
surjective,  and the kernels of these  restrictions are   transversal to the
subspace $T^{\partial}_{(t, y)}(\R^N\times  \bigcap_{j\in \sigma}
{\mathcal F}^j)$ in the tangent space $T_{(t, y)}(\R^N\oplus
\bigcap_{j\in \sigma} {\mathcal F}^j)$.
\end{itemize}
\end{lem}
\begin{proof}
We already know form Lemma \ref{lempert00} that the conclusions of the lemma  hold true at the special point $(t, y)=(0, x)$.
Working in local coordinates as in the proof of Lemma 5.23 in \cite{HWZ3},   we  find for every $i\in J$ and $g\in G_x$ a positive number  $\varepsilon_{i,g}$ and an open neighborhood
$V_{i,g}\subset U_x$ of $x$ such that the conclusions (1)-(3)  hold for $(t, y)\in \R^N\oplus V_{i,g}$ if  $\abs{t}<\varepsilon_{i,g}$.  Then the lemma follows by taking for the set $V_x$ a $G_x$-invariant open neighborhood of $x$ which is contained in the intersection of the  sets $V_{j,g}$ and choosing  for $\varepsilon$ a  positive number smaller than the  numbers $\varepsilon_{i,g}$.
\end{proof}

Finally we can finish  the proof of Theorem \ref{pertmontreal1}.

By assumption the solution set  $S(f, \lambda)$ of the proper Fredholm section $f$ of $p$ is a compact subset of the orbit space $\abs{X}$.  Consequently, there exist finitely many solutions $x_1, \ldots ,x_m$ belonging to $S(F, \Lambda)$ so that for the corresponding open neighborhoods $V_{x_1}, \ldots ,V_{x_m}$ in $X$,
the sets $\abs{V_{x_1}}, \ldots ,\abs{V_{x_m}}$ in the orbit space $\abs{X}$ cover
$S(f, \lambda ))$.  Then we define the parametrized $\ssc^+$-multisection  $\Lambda_t$ as the sum
$$\Lambda_t=\Lambda^{t_1}_{x_1}\oplus \cdots \oplus \Lambda^{t_m}_{x_m}$$
where $t=(t_1, \ldots ,t_m)\in \R^{N_1}\times \cdots \R^{N_m}=\R^{N}$  with  $N=N_1+\cdots +N_m$.

Using Lemma \ref{lemmontreal} one concludes that the solution set
$$S_{\varepsilon}=\{ (t, y)\in \R^N\times X^1\vert \, \text{$(\Lambda \oplus \Lambda_t)(F(y))>0$ and $\abs{t}<\varepsilon$}\}$$
consits of a finite collection of finite dimensional manifolds with boundary with corners. Now one studies the projection map $S_{\varepsilon}\to \R^l$ defined by $(t, x)\mapsto t$ and finds a small common regular value $t^*$ for various restrictions of the map to intersections of local faces. For this parameter value $t^*$, the pair $(F, \Lambda \oplus \Lambda_{t^*})$ is in general position and the associated solution space has an orbit space which is a compact branched suborbifold  of the polyfold $Z\subset \abs{X}$ in general position to the boundary. In particular, the associated solution set
$S(f, (\lambda \oplus \vartheta_{t^*})=\{z\in Z\vert \,  (\lambda \oplus \vartheta_{t^*})(f(z))>0\}$ is a compact branched orbifold with boundary with corners. This finishes the proof of
Theorem \ref{pertmontreal1}.
\end{proof}




The following result is proved along the lines of the previous result. In contrast to Theorem \ref{pertmontreal1}, we impose conditions on the Fredholm sections at those   solutions which are located at the boundary $\partial X$, while the perturbation has its support away from the boundary.

\begin{thm}\label{pertzxc}
Let $p:W\rightarrow Z$ be a strong polyfold bundle, where $Z$ has a
boundary with corners.  We assume that the sc-smoothness
structure of the polyfold $Z$ is based on separable Hilbert spaces. Assume that $f$ is a proper (oriented) Fredholm section
of $p$ and  $N$ an auxiliary norm and $U$ an open neighborhood of  the solution set
$S(f)=\{z\in Z\vert \, f(z)=0\}$
so that  the pair $(N,U)$ controls compactness. Let $\lambda$ be a
$\ssc^+$-multi-section supported in $U$ and satisfying $N(\lambda(z))<\frac{1}{2}$ for all $z$.
Assume that $(f,\lambda)$ is in good position
at the boundary. Then, given $\varepsilon<\frac{1}{2}$,  there
exists an $\ssc^+$-multisection $\tau$ satisfying  $N(\tau(z))<\varepsilon$
and supported in $U$ so that $\tau$ is trivial at the boundary and the pair
$(f,\lambda\oplus \tau)$ is in good position. In particular,  the solution set
$S(f,\lambda\oplus\tau)$ is an (oriented)  compact branched suborbifold of $Z$ with
boundary with corners.
\end{thm}
\begin{proof}
The proof is straight forward.  By assumption,  the pair $(f, \lambda)$ is already  transversal and in
good position to the boundary so that the solution space of
$(f,\lambda)$ near $\partial Z$ is already a branched orbifold. Then
we can perturb the $\ssc^+$-multisection $\lambda$ by an arbitrarily
small $\ssc^+$-multisection $\tau$ by the same argument as in the proof of
Theorem \ref{ghjl}, where,  in addition,  $\tau$ is trivial  near
$\partial Z$, so that  the pair $(f,\lambda\oplus \tau)$ is tansversal and,
by construction,  still in good position at solutions $x\in\partial
Z$.
\end{proof}

Theorem \ref{pertzxc} and variants thereof are
important in the `coherent perturbation theory' used in symplectic
field theory, where one deals simultaneously with infinitely many
Fredholm problems and where the boundaries are explained as products (or
fibered products) of Fredholm problems. In this
case one has an algorithm how the data are  being perturbed which
defines inductively perturbations on the boundary so that  the
problem at the boundary is already in good position.  Then one extends the
perturbation by keeping compactness and transversality.

\subsection{Invariants}
The results above allow to define invariants. The first result is
an abstract version of the argument  used in order  to define the Gromov-Witten
invariants in  \cite{HWZ4}.

\begin{thm}\label{o1}
Assume that $f$ is a proper, oriented Fredholm section of the strong
polyfold bundle $p:W\rightarrow Z$ and $\partial Z=\emptyset$.
Then
there exists a well-defined map
$$
\Phi_f:H^\ast_{dR}(Z)\rightarrow
{\mathbb R}$$
defined on the deRham cohomology group $H^\ast_{dR}(Z)$
 so that the following holds. If $(N,U)$ is a pair
controlling compactness, where $N$ is an auxiliary norm and $U$ a
corresponding open neighborhood of the solution set  $S(f)=\{z\in Z\vert \, f(z)=0\}$,  then for any
$\ssc^+$-multisection $\lambda$ with support in $U$ and satisfying
$N(\lambda(z))<1$ for all $z$ and such that the pair  $(f,\lambda)$ is  transversal,  the following representation of $\Phi_f$ holds,
\begin{equation}\label{opl}
\Phi_f([\omega]) := \int_{(S(f,\lambda),\lambda_f)}\omega.
\end{equation}
Here the pair $(S(f, \lambda), f_{\lambda})$ is a compact oriented branched suborbifold of $Z$ in which
$S(f, \lambda)=\{z\in Z\vert \, \lambda (f(z))>0\}$  is the  solution set  equipped with the weighting function $\lambda_f (z):=\lambda (f(z)).$
\end{thm}

\begin{proof}
In view of Theorem \ref{trans2}
the map $z\mapsto  \lambda_f(z)$ defines  a compact oriented branched
suborbifold of $Z$.  Hence, by the results in \cite{HWZ7}, the integrals $\int_{(S(f, \lambda), \lambda_f)}\omega$ are well defined real numbers. In order to show that they do not depend on the choice of $\lambda$ we choose a second transversal pair $(f, \lambda')$ whose
$\ssc^+$-multisection $\lambda'$ is supported in $U$ and view $f$ as a Fredholm section of the strong polyfold bundle $W\to [0,1]\times Z$. Now we take a parametrized $\ssc^+$-multisection
$\lambda_t$ having its support in $[0,1]\times U$ such that the pairs $(f, \lambda_t)$ are transversal and connect the $\ssc^+$-multisection $\lambda_0=\lambda$ with the  $\ssc^+$-multisection $\lambda_1=\lambda'$. The disjoint union $S(f, \lambda')\coprod -S(f, \lambda)$ is the boundary  of the solution set $S(f, \lambda_t)=\{(t, z)\in [0,1]\times Z\vert \, \lambda_t (z)>0\}$ which is a compact oriented suborbifold of $[0,1]\times Z$. Stokes' theorem from \cite{HWZ7} leads to
$$\int_{(S(f, \lambda), \lambda_f)}\omega=\int_{(S(f, \lambda'), \lambda'_f)}\omega .$$
Hence the right hand side of \eqref{opl} is indeed independent of the choice of the transversal pair $(f, \lambda)$ in $U$ and hence $\Phi_f ([\omega ])$ is well defined by the formula  \eqref{opl}.  The proof of Theorem \ref{o1} is complete.
\end{proof}
Note that we have a distinguished $0$-form on the polyfold $Z$, namely the constant
$1$-function. Assume that $f$ is a proper and oriented  Fredholm section of the strong polyfold bundle $W\to Z$ and has Fredholm
index $0$. Then $\Phi_f([1])$ is a rational number and is  a version of a  degree for oriented proper Fredholm
sections of strong polyfold bundles.

The application of Theorem \ref{o1} to the  Gromov-Witten invariants sketched in the introduction is as follows. We consider the polyfold $Z$ whose elements are the equivalence classes $[(S, j, M, D, u)]$ introduced in Definition \ref{defspace} and the strong polyfold bundle $:W\to Z$ in Theorem \ref{natural}.  Let $\overline{\partial}_J$ be the sc-smooth component-proper Fredholm section of the bundle according to Theorem \ref{compproper}. Given the homology class $A\in H_2(Q, \Z)$ of the closed symplectic manifold $Q$ and two integers $g, k\geq 0$, we look at the polyfold $Z_{g,k}\subset Z$ of equivalence classes $[(S,j, M, D, u)]$ in which the nodal Riemann surface $S$ has arithmetic genus $g$ and is equipped with $k$ marked points, and the map $u$ represents the homology class $A$ of the manifold $Q$. The evaluation maps $\text{ev}_i:Z_{g,k}\to Q$ and $\sigma:Z_{g,k}\to \overline{\mathcal M}_{g,m}$ are defined in the introduction and allow to pull back the differential forms  on $Q$ and on $\overline{\mathcal M}_{g,m}$ to obtain sc-differential forms on the polyfold $Z_{g,k}$.  Wedging combinations of these forms together, we can
integrate over the oriented solution set  $(\mathfrak{M}, w)=(S(\overline{\partial}_{J},\lambda),\lambda_{\overline{\partial}_{J}})$ assuming that the pair $(\overline{\partial}_{J}, \lambda)$ is transversal. Here
$$\mathfrak{M}=S(\overline{\partial}_{J},\lambda)=\{z\in Z_{g,k}\vert \, \lambda (\overline{\partial}_{J}z)>0\}$$
and the weight function $w=\lambda_{\overline{\partial}_J}:Z_{g,k}\to \Q^+$ is defined by
$\lambda_{\overline{\partial}_J}(z)=\lambda (\overline{\partial}_J z).$
The pair $(\mathfrak{M}, w)$ is an oriented compact branched suborbifold of the polyfold $Z_{g,k}$. Hence, in view of the abstract Theorem \ref{o1}, the Gromov-Witten invariants can be constructed by means of the map
$$
\Phi^Q_{A,g,m}:H^\ast(Q)^{\otimes k}\otimes H_\ast(\overline{\mathcal
M}_{g,m}) \rightarrow \R
$$
defined by the formula
$$
\Psi^Q_{A,g,k}([\alpha_1],\ldots, [\alpha_k];[\tau])=
\int_{(\mathfrak{M},w)} \text{ev}_1^\ast(\alpha_1)\wedge\ldots \wedge
\text{ev}_k^\ast(\alpha_k)\wedge \sigma^\ast(PD(\tau))
$$
where $\alpha_1, \ldots ,\alpha_k\in H^*(Q)$ and where $\tau \in H_{\ast}(\overline{M}_{g,k})$ and where $PD$ denotes the Poincar\'e dual.  That the formula is independent of the choice of the $\ssc^+$-multisection $\lambda$ follows from Theorem \ref{o1}.

Finally, we shall complete  the proof of Theorem \ref{poil}.
\begin{proof}
Using Theorem \ref{pertmontreal1} we take a generic perturbation $\lambda$ so that the solution set $S=S(f,\lambda)$ is in good position to the
boundary. Assume that $S$ is of dimension $n$. Then for
$[\omega,\tau]\in H^n_{dR}(Z,\partial Z)$ we have the well-defined
integral $\int_{(S,w)}\omega-\int_{(\partial S,w)}\tau$.

If  $f_t$ is an oriented smooth proper homotopy of Fredholm
sections connecting  $f_0$ with $f_1$ we can view
$$
F(t,x)=f_t(x)
$$
as a proper Fredholm section of $W$ pulled-back by the projection map
$[0,1]\times Z\rightarrow Z$. Then we can fix an open neighborhood $U$
of the solution set $S(f)=\{z\in Z\vert \, f(z)=0\}$ and an auxiliary norm so that  the pair $(N,U)$
controls  compactness. Assume that $(f_i,\lambda_i)$ are two admissible
perturbations so that the corresponding solution sets are in good
position to the boundary.  The  proper homotopy $f_t$  can be perturbed  generically
 by the same argument as before  to find a homotopy
$\lambda_t$,  which for $i=0,1$ coincides with the perturbations we
already have, so that away from the boundaries $\{0\}\times Z$ and
$\{1\}\times Z$ the perturbed pair $(f_t,\lambda_t)$ is in good
position to the boundary. Now the previous discussion about the
behavior of our invariant under our homotopy  finishes the proof of Theorem \ref{poil}.
\end{proof}

\section{Appendix}

In the following we explain some of the more technical results as
well as some of the necessary background material.

\subsection{Natural Representation of Stabilizers}\label{verynewsection8.4}
We  shall study  the
local structure of the morphism set of an ep-polyfold groupoid in more detail.

We choose an object $x_0\in X$. By the ep-assumption its stabilizer group
${\bf X}(x_0)$ is a finite group and we denote it by $G_{x_0}$. The
following theorem describes the  structure of the morphism set
near the isotropy group.
\begin{thm}\label{verynewtheorem8.25}
Given an ep-groupoid $X$, an object  $x_0\in X$, and an open neighborhood $V\subset X$ of $x_0$. Then there exist  an open
neighborhood $U\subset V$ of $x_0$,  a group homomorphism
$$
\varphi:{G}_{x_0}\rightarrow \hbox{Diff}_{\ssc}(U), \quad g\mapsto
\varphi_g=t_g\circ s_g^{-1},
$$
and an  sc-smooth map
$$
\Gamma:G_{x_0}\times U\rightarrow {\bf X}
$$
having the following properties.
\begin{itemize}
\item[$\bullet$] $\Gamma(g,x_0)=g$.
\item[$\bullet$] $s(\Gamma(g,y))=y$ and $t(\Gamma(g,y))=\varphi_g(y)$ for all $y\in U$ and $g\in G_{x_0}$.
\item[$\bullet$] If $h:y\rightarrow z$  is a morphism  connecting  two objects $y,z\in U$,  then there exists a
unique $g\in G_{x_0}$ such that  $\Gamma(g,y)=h$.
\end{itemize}
\end{thm}
In particular, every morphism between points in $U$ belongs to the
image of the map $\Gamma$.  We call  the group homomorphism $\varphi:G_x\rightarrow
\mbox{Diff}_{\ssc}(U)$ a {\bf natural representation} of the
stabilizer  group   $G_{x_0}$.
\begin{proof}

 For
every $g\in G_{x_0}$ we choose two  contractible open neighborhoods
$N^t_g$ and $N^s_g\subset {\bf X}$ on which the target and source
maps $t$ and $s$ are sc-diffeomorphisms onto some open neighborhood
$U_0\subset X$ of $x_0$. Since  the isotropy group $G_{x_0}$ is
finite we can assume that the open sets $N^t_g\cup N^s_g$ for $g\in
G_{x_0}$ are disjoint and define the disjoint open neighborhoods
$N_g\subset {\bf X}$ of $g$ by
$$N_g:=N_g^t\cap N_g^s,\quad g\in G_{x_0}.$$
We  abbreviate the restrictions of the source and target maps by
$$s_g:=s\vert N_g\quad \text{and}\quad t_g:=t\vert N_g.$$

\begin{lem}\label{lem7.9.1}
With the choices made above there exists an open neighborhood
$U_1\subset U_0$ of $x_0$ so that every morphism $h\in {\bf X}$ with
$s(h)$ and $t(h)\in U_1$ belongs to $N_g$ for some  $g\in G_{x_0}$.
\end{lem}
\begin{proof}
Arguing indirectly we find a sequence $h_k\in {\bf X}$ with $h_k\not
\in N_g$ for all $g\in G_{x_0}$ and satisfying $s(h_k), t(h_k)\to
x_0$ as $k\to \infty$. By the properness assumption of ep-polyfolds
there is a convergent subsequence $h_{k_l}\to h\in {\bf X}$.
Necessarily $h\in G_{x_0}$ and hence $h\in N_g$ for some $g\in
G_{x_0}$. This contradiction implies the lemma.
\end{proof}

\begin{lem}\label{lem7.9.2}
If  $U_1$ is  the open neighborhood of $x_0$ guaranteed by Lemma
\ref{lem7.9.1}, then there exists an open neighborhood $U_2\subset
U_1$ of $x_0$ so that the open neighborhood $U$ of $x_0$,  which is
defined as the union
$$U:=\bigcup_{g\in G_{x_0}}t_g\circ s_g^{-1}(U_2),$$
is contained in $U_1$ and invariant
 under all the maps $t_g\circ s_g^{-1}$ for $g\in G_{x_0}$.
\end{lem}
\begin{proof}
We choose an open neighborhood $U_2\subset U_1$ of $x_0$  so small
that the union $U$ and also $t_g\circ s_g^{-1}(U)$ are contained in
$U_1$ for all $g\in G_{x_0}$. Consider the map $t_g\circ
s^{-1}_g:U\to X$ and choose $x\in U$. Then we can represent it as
$x=t_h\circ s_h^{-1}(u)$ for some $h\in G_{x_0}$ and some $u\in
U_2$. Now, $v:=t_g\circ s_g^{-1}\circ t_h\circ s_h^{-1}(u)$ belongs
to $U_1$ and the formula implies the existence of a morphism $u\to
v$ in ${\bf X}$. By Lemma \ref{lem7.9.1} the morphism has
necessarily the form $v=t_{g'}\circ s^{-1}_{g'}(u)$ for some $g'\in
G_{x_0}$. Since $u\in U_2$ it follows that $v=t_g\circ
s_g^{-1}(x)\in U$ implying the desired invariance of $U$.
\end{proof}

In view of Lemma \ref{lem7.9.2} we can associate  with every $g\in
G_{x_0}$ the sc-diffeomorphism
$$\varphi (g ):=t_g\circ s_g^{-1}:U\to U$$
of the open neighborhood $U$ of $x_0$, and obtain the mapping
$$\varphi: G_{x_0}\to \text{Diff}_{\ssc}(U), \quad g\mapsto \varphi_g.
$$
 Since the neighborhoods $N_g\subset {\bf X}$ of $g$ are
disjoint and since the structure maps are continuous we conclude
that $\varphi$ is a homomorphism of groups, in the following called
the {\bf natural representation of the stabilizer group $G_{x_0}$}
by sc-diffeomorphisms of the open neighborhood $U\subset X$ of
$x_0$. Then we define
$$
\Gamma: G_{x_0}\times U\rightarrow {\bf X}:\Gamma(g,y)= s_g^{-1}(y).
$$
Summing up the consequences of Lemma \ref{lem7.9.1} and Lemma
\ref{lem7.9.2} we have proved Theorem \ref{verynewtheorem8.25}.
\end{proof}

The following structure theorem is fundamental for  the  constructions of perturbations.

\begin{thm}\label{lem7.9.3}
Every object $x_0\in X$ of an ep-groupoid possesses an open
neighborhood $U\subset X$  having the following properties.
\begin{itemize}
\item[$\bullet$\: ] On $U$, the stabilizer  group $G_{x_0}$ has the natural representation
$$\varphi:G_{x_0}\to \text{Diff}_{\ssc}(U).$$
\item[$\bullet$\: ] Assume $y_0\in X$ is an object for which
there exists no morphism $y_0\to x$, where $x$ in $\ov{U}$. Then
there exists an open neighborhood $V$ of $y_0$ so that for every
$z\in V$ there is no morphism to an element in $\ov{U}$.
\item[$\bullet$\: ] Assume $y_0\in X$ is an object for which
there exists no morphism $y_0\to x$ for every $x\in U$, but there
exists a morphism to some element in $\ov{U}$. Then given an open
neighborhood $W$ of $\partial U$ (the set theoretic boundary of
$U$), there exists an open neighborhood $V$ of $y_0$ so that if
there is a morphism $y\to x$ for some $y\in V$ and $x\in U$, then
$x\in W$.
\end{itemize}
Moreover, the open set $U$ can be taken as small as we wish.
\end{thm}
\begin{proof}
For the first statement  we refer to Theorem
\ref{verynewtheorem8.25}.  We can choose the neighborhood  $U$ of
$x_0$  as small as we wish. Hence we may assume that
$$t:s^{-1}(\ov{U})\to X$$
is proper. To prove the second statement we assume for $y_0\in X$
that there is no morphism to any element in $\ov{U}$. If no
neighborhood $V$ with the desired properties exist, we find
sequences $y_k\to y_0$ and $x_k\in \ov{U}$ so that there exist
morphisms $y_k\to x_k$. Inverting these morphisms we obtain a
sequence of morphisms $g_k:x_k\to y_k$, and the sequence of points
in $X$,
$$y_k=t(g_k)\in t(s^{-1}(\ov{U})).$$
By the properness assumption we may assume without loss of
generality that $g_k\to g\in {\bf X}$ implying $x_k\to x'$ in $X$
for some $x'$. Therefore,
$$g:x'\to y_0\quad \text{and}\quad x'\in \ov{U}.$$
This contradiction proves the second assertion.

Assume, finally, that there exists no morphism $y_0\to x$ for $x\in
U$, but a morphism $y_0\to \ov{x}\in \ov{U}$. Pick an arbitrary open
neighborhood $W$ of $\partial U$. If $V$ with the desired properties
does not exist we find sequences of morphisms $y_k\to y_0$ and
elements $x_k\in U\setminus W$ admitting morphisms $g_k:y_k\to x_k$.
Using the properness assumption again we may assume that $g_k\to g$
in ${\bf X}$ where $g:y_0\to x$ and $x\in U$,  giving a
contradiction. The proof of Theorem \ref{lem7.9.3} is complete.
\end{proof}

\subsection{Sc-Smooth Partitions of Unity}\label{partition}
In  this section  we  prove the existence of an sc-smooth partition of unity on an ep-groupoid.  We consider an ep-groupoid  whose {\bf sc-structure is  based on separable sc-Hilbert spaces}.
We view $[0,1]$ as a category with only the identity morphisms.   An sc-smooth functor $f:X\to [0,1]$ on $X$  is an sc-smooth map on the object M-polyfold  which  is invariant under morphisms, that is, $f(x)=f(y)$ if there exists a morphism $h:x\to y$.

\begin{defn}  Let $X$ be an ep-groupoid and let ${\mathcal U}=(U_{\alpha})_{\alpha \in A}$ be an open cover  of $X$ consisting of saturated sets.  An sc-partition of unity  $(g_\alpha)_{\alpha \in A}$ subordinate to ${\mathcal U}$ consists of  the  locally finite  collection of sc-smooth functors $g_\alpha :X\to [0,1]$ satisfying  $\sum_{\alpha \in A}g_\alpha =1$ and $\supp g_{\alpha}\subset U_{\alpha}$ for every $\alpha \in A$.
\end{defn}
The existence of an sc-smooth partition of unity depends on a sufficient supply of sc-smooth functions.  We shall make use of  the following result  for separable Hilbert spaces  proved in \cite{Fathi}.

\begin{lem}\label{l1}
Let $U$ and $W$ be  open subsets of a separable Hilbert space $H$  such that $\ov{W}\subset  U$. Then there exists a smooth function $f:H\to [0,1]$ having its support contained in $U$ and satisfying  $f=1$ on $\ov{W}$.
\end{lem}
The proof  of  Lemma \ref{l1}  extends  easily to the case in which  $U$ and $W$ are  open subsets of  a partial quadrant in  a separable Hilbert space. In the next lemma we extend  Lemma \ref{l1}  to the sc-context.
\begin{lem}\label{l2}
Let $W$ and $U$ be open subsets of a splicing core $K$ such that $\ov{W}\subset U$. Then there exists an sc-smooth function $f: K\to [0,1]$ such that $f$ has its support in $U$ and is equal to $1$ on $W$.
\end{lem}
\begin{proof}
Assume that  $K=K^{{\mathcal S}}=\{(v, e)\in V\oplus E\vert \, \pi_v (e)=e\}$ is the splicing core associated with the splicing ${\mathcal S}=(\pi, E, V)$.  Here $V$ is an open subset of a partial cone $C$ in a separable sc-Hilbert space $Z$,  $E$ is a separable sc-Hilbert space, and $\pi:V\oplus E\to E$ is an sc-smooth map such that $\pi (v, \cdot ):=\pi_v:E\to E$ is a bounded  linear projection for every $v\in V$.  Consider  $\Phi:W\oplus E\to W\oplus E$ defined by $\Phi (v, e)=(v, \pi (v, e))$. The map
$\Phi$ is sc-smooth and, in particular, continuous from level $0$ to level $0$ of $V\oplus E$.
Moreover, $\Phi (V\oplus E)=K$. Put  $W'=\Phi^{-1}(W)$ and $U'=\Phi^{-1}(U)$.
Then $W'$ and $U'$ are open and since  $\Phi^{-1}(\ov{W})$ is closed,  we get
$\ov{W'}=\ov{\Phi^{-1}(W)}\subset \Phi^{-1}(\ov{W})\subset U'.$
By Lemma \ref{l1}, there exists a smooth function $f_0:V\oplus E\to [0,1]$  such that $\supp f_0\subset U'$ and $f=1$ on $W'$.  Since $f_0$ is sc-smooth in view of Proposition 2.15  in \cite{HWZ2} and since  the map $\pi$ is  sc-smooth,  the composition $f_0\circ \pi$ is also sc-smooth by the chain rule, Theorem 2.16 in \cite{HWZ2}.  Hence,  putting
 $f:=f_0\vert K$,  we obtain an sc-smooth function  defined on $K$  having its support in $U$ and equal to $1$ on $W$.
\end{proof}

If $g:X\to [0,1]$ is an sc-smooth functor on the ep-groupoid $X$, we denote by $\abs{g}$  the continuous function defined on the orbit space $\abs{X}$ by $\abs{g}( \abs{ x}):=g(x)$.  Now we come to the statement of the main theorem of this section.

\begin{thm} [{\bf  sc-smooth partition of unity}] \label{scpounity}
Let $X$ be an  ep-groupoid  whose sc-structure is based on separable sc-Hilbert spaces, and let ${\mathcal O}=(O_{\alpha})_{\alpha \in A}$ be an  open cover of the orbit space  $\abs{X}$. Then there exists an sc-smooth partition of unity $(g_{\alpha})_{\alpha\in A}$ on  $X$ so that the associated continuous  partition of unity $(\abs{g_{\alpha}})_{\alpha\in A}$ of $\abs{X}$ is subordinate to ${\mathcal O}=(O_{\alpha})_{\alpha \in A}$.
\end{thm}

The proof of Theorem \ref{scpounity}  will follow from the next  two lemmata which make us of  Theorem \ref{lem7.9.3} in  Appendix \ref{verynewsection8.4}.

\begin{lem}\label{invarinatcover}
Let ${\mathcal O}=(O_{\alpha})_{\alpha\in A}$ be an open cover of $\abs{X}$. Then  there exist locally finite open covers $(W_{j})_{j\in J}$ and $(U_{j})_{\beta \in J}$  subordinate to ${\mathcal O}$ and such that $\ov{W_j}\subset U_{j}$. The sets
$W_{j}$ and $U_{j}$ are invariant  under the natural representations of the isotropy groups $G_{x_j}$ on
$U_{j}$ for some $x_j\in U_j$, and the open cover $(\pi^{-1}(\pi (U_{j}))_{j\in J}$ is locally finite.
\end{lem}
\begin{proof}
In view of the paracompactness of $\abs{X}$,  there is  a locally finite refinement $(Q_{\alpha})_{\alpha\in A}$  of the cover $(O_{\alpha})_{\alpha \in A}$.  Then $(\pi^{-1}(Q_{\alpha}))_{\alpha \in A}$ is a locally finite refinement of $(\pi^{-1}(O_{\alpha}))_{\alpha \in A}$. For every point $x\in X$,  we choose an open neighborhood $V_x$ intersecting only a finite number of sets $\pi^{-1}(Q_{\alpha})$.  We replace  $V_x$ by its  intersection  with  those  $\pi^{-1}(Q_{\alpha})$  which contains  the point $x$. Observe that  there is no morphism between  the point $x$ and the sets $\pi^{-1}(Q_{\alpha})$ which don't intersect $V_x$.  Hence, shrinking $V_x$ further,  we may assume that $V_x$ has the  properties  listed in Theorem \ref{lem7.9.3} and that there are no morphisms  between points in $V_x$ and points in the sets $\pi^{-1}(Q_{\alpha})$  not intersecting  $V_x$.  The collection $(V_x)_{x\in X}$ is an open cover of $X$ and since $X$ is paracompact, there exists a locally finite refinement  $(U'_{j})_{j\in J}$ of $(V_x)_{x\in X}$. For every  $j\in J$, choose  a  point $x(j)$ such that $U'_{j}\subset V_{x(j)}$.  We abbreviate by $G_{j}$ the isotropy group $G_{x(j)}$  acting on $V_{x(j)}$ by its natural representation.
We claim that $(\pi^{-1}(\pi (U_j')))_{j\in J}$ is a locally finite cover of $X$. Indeed, take $y\in X$. Then $y\in U_k'\subset V_{x(k)}$ for some $k\in J$. Since $(U_j')_{j\in J}$ is locally finite, there exists an open neighborhood $W_y$ of $y$ contained in $U'_k$ and intersecting only a finite number of the sets $U_j'$, say $U'_{j_1}, \ldots ,U'_{j_N}$. Hence $k=j_i$ for some $1\leq i\leq N$. Replacing $W_y$  by a smaller set,  we may assume that $W_y$ is $G_{k}$-invariant.  Assume that $z\in W_y\cap \pi^{-1}(\pi (U'_j))$ for some $j\neq j_1,\ldots , j_N$.  Then there is a morphism between some point $v\in V_j$ and $z=\varphi_g (v)$ for some $g\in G_j$. In view of the definition of $V_{x(k)}$,  we have $v\in V_{x(k)}$.
Hence there is $h\in G_{k}$ such that $v=\varphi_h (z)$, and since $z\in W_y$ and $W_y$ is $G_j$-invariant $v\in W_y$. Consequently, $W_y\cap U'_j\neq \emptyset$ and it follows that $W_y$ intersects only the sets $\pi^{-1}(\pi (U'_{j_1})), \ldots ,\pi^{-1}(\pi (U'_{j_N}))$.
For every $j\in J$, set  $U_j=\bigcup_{g\in G_j}\varphi_g (U'_j)$. Then $\pi^{-1}(\pi (U_j))=\pi^{-1}(\pi (U_j'))$ and since the isotropy groups $G_j$ are finite, it follows that $(U_j)_{j\in J}$ is a locally finite cover of $X$ such that $U_j\subset V_{x(j)}$.  Using paracompactness of $X$ again, we find a locally finite cover $(W'_{j})_{j\in J}$ such that $\ov{W_j'}\subset U_{j}$.
Define  $W_{j}=\bigcup_{g\in G_{j}}\varphi_g (W_j')$. Then  $W_{j}$ is  a $G_{j}$-invariant  open subset of $V_{j}$ such that  $\ov{W_j}\subset U_j$, and   the open cover  $(W_j )_{j\in J}$ is  locally finite.  This completes the  proof of the lemma..
\end{proof}

\begin{lem}\label{lemmascp1}
Let  $U=U(x_0)\subset X$ be an open neighborhood of $x_0$ with  the properties as listed in  Theorem \ref{lem7.9.3} and let  $\varphi:U\to K^{\mathcal S}$ be  a coordinate chart onto an open subset of the splicing core $K^{\mathcal S}$. Assume that  $W$ is a $G_{x_0}$-invariant open subset of $U$ such that $\ov{W}\subset U$. Then there exists an sc-functor $f:X\to [0,1]$ satisfying $f=1$ on $\ov{W}$ and  $\supp f\subset \pi^{-1}(\pi (U)).$
\end{lem}
\begin{proof}
We choose an open $G_{x_0}$-invariant set $V$ such that  $\ov{W}\subset V\subset \ov{V}\subset U$.  With the help of  Lemma \ref{l2} and the chart $\varphi:U\to K^{\mathcal S }$, we find an sc-smooth function $f_0: X\to [0,1]$  satisfying $\supp f_0\subset V$ and $f_0=1$ on $\ov{W}$.  Define
 the function $f_1$ on $U$  by
 $$f_1 (x)=\dfrac{1}{\sharp G_{x_0}}\sum_{g\in G_{x_0}}f_0 (\varphi_g (x)), \quad x\in U.$$
Then   $f_1$ is sc-smooth as  a  finite sum of sc-smooth functions, $0\leq  f_1 \leq 1$, and  $f_1 =1$  on $x\in \ov{W}$ since $\ov{W}$ is $G_{x_0}$-invariant.   Since $f_0 =0$ on $U\setminus \ov{V}$ and $U\setminus \ov{V}$ is $G_{x_0}$-invariant,  it follows that also $f_1=0$ on $U\setminus \ov{V}$.  In particular, $f_1=0$ on $\partial V$.

We extend $f_1$ to the function  $f:X\to [0,1]$ as follows.  If $x\in U$,  then $f(x):=f_1 (x).$ If  there exists a morphism between $x$ and some point $y\in U$, then  we set $f(x):=f_1  (y).$ Finally, if there is no morphism between $x$ and a point in $U$, then  we set  $f(x):=0.$

Clearly, $f(x)=f(y)$ if there  is a morphism $h: x\to y$. Note also that $f (x)=0$ for $x\in \partial U$. Indeed, if there are no morphisms  between $x$ and points of $U$, then $f(x)=0$  by the definition of $f$. If there exists a morphism between $x$ and $y\in U$, then by Theorem  \ref{lem7.9.3}, the point $x$ belongs to $U\setminus \ov{V}$,  so that again $f (x)=f_1(y)=0$.

We already know that $f$ is sc-smooth on $U$. To show that it is sc-smooth on $X$, we
take   $x\in X\setminus U$ and consider the following cases.  If  there is no morphism between $x$ and a point in  $\ov{U}$, in particular, there is no morphism between $x$ and a point in $U$, then $f(x)=0$. By   part (b) of Theorem  \ref{lem7.9.3},  $f=0$ on some  open neighborhood $U_x$ of $x$  and so $f$ is sc-smooth on $U_x$.

Next assume that  there exists a morphism $h: x\to y$  between the point $x$ and a point $y\in U$. According to the definition of $f$,  $f (x)=f_1 (y)$. We find two  open neighborhoods $U_x$ and $U_y$  of $x$ and $y$ such that  $U_y\subset U$  and   $t\circ s^{-1}:U_x\to U_y$ is an sc-diffeomorphism.  Then, $f =f_1 \circ t\circ s^{-1}$ on $U_x$ and since the right hand side is an sc-smooth function,  the function $f$ is sc-smooth on $U_x$.

In the last case, assume that  there is no morphism between $x$ and points of $U$ but there is a morphism between $x$ and some  point $y\in \partial U$. Then again we find open neighborhoods $U_x$ and $U_y$ of points $x$ and $y$ such that $t\circ s^{-1}:U_x\to U_y$ is an sc-diffemorphism. By  Theorem   \ref{lem7.9.3}, we may  take  these neighborhoods so small that the following holds. If there exists a morphism between a point $y'\in U_y$ and a point $z\in U$, then necessarily $z\in U\setminus W$.  At $y'$ we have $f(y')=0$ since $f_1=0$ on $U\setminus W$. If there are no morphism between $y'\in U_y$ and points in $U$, then $f (y')=0$. Hence $f$ is equal to $0$ on $U_y$ and  since $f\vert U_x=(f\vert U_y)\circ (t\circ s^{-1})$,  we conclude that $f$ is equal to $0$ on $U_x$.  So we  proved that the function $f$ is sc-smooth on $X$.

It remains to prove  that $\supp f\subset \pi^{-1}(\pi (U))$. At every point $x\not \in \pi^{-1}(\pi (U))$, $f(x)=0$.   Hence  it  is  enough to show that for every $x\in \partial \pi^{-1}(\pi (U))$ there exists a neighborhood $U_x$ of $x$ such that $f=0$ on $U_x$.  To see this, we prove that there is an  open neighborhood $U_x$ of $x$ such that if there is a morphism between $x'\in U_x$ and a point in $y\in U$,  then $y\in U\setminus W$. Otherwise, we find a sequence  $(x_n)$ converging to $x$, a sequence $(y_n)\subset W$, and a sequence  $(h_n)$ of morphisms $h_n:y_n\to x_n$.  Since the map $t:s^{-1}(\ov{U(x_0}))\to X$ is proper, there is  subsequence of the morphisms $(h_n)$ converging to the  morphism $h$. This implies that  the subsequence of $(y_n)$ converges to the point $y\in W$ and that $h:y\to x$,  contradicting the fact that there are no morphisms between points in $\partial \pi^{-1}(\pi (U))$ and points in  $U$.   Hence $f =0$ on $U_x$ and this proves that $\supp f \subset \pi^{-1}(\pi (U)).$
\end{proof}

\begin{proof}[Proof of Theorem \ref{scpounity}]
Let ${(O_{\alpha})}_{\alpha \in A}$ be an open cover of $\abs{X}$.  In view of  Lemma \ref{invarinatcover}, there are open covers $(W_j)_{j\in J}$ and  $(U_j)_{j\in J}$ subordinate to $(\pi^{-1}(O_{\alpha}))$ such that $\ov{W_j}\subset U_j$. Moreover, the sets
$W_j$ and $U_j$ are invariant with respect to the natural representation of $U_j$ and the cover $(\pi^{-1}(\pi (U_j)))_{j\in J}$ is locally finite. By Lemma \ref{lemmascp1}, for every
$j\in J$, there is an sc-smooth functor $f_j':X\to [0,1]$ which is equal to $1$ on $\ov{W_j}$ and $\supp f_j'\subset \pi^{-1}(\pi (U_j))$. Set $f'=\sum_{j\in J}f_j'$. In view of the local finiteness of   $(\pi^{-1}(\pi (U_j)))_{j\in J}$, the sum  has only a finitely many nonzero terms in a neighborhood of each point and thus defines an sc-smooth function. Because $f_j=1$ on $W_j$ and every point of $X$ is in some $W_j$, the sum is also positive. Now define
$f_j=\tfrac{f'_j}{f'}$. Then each $f_j$ is an sc-smooth functor such that $\supp f_j\subset \pi^{-1}(\pi (U_j))$.   Finally, we may reindex our functions $f_j$ so that they are indexed by the indices in the set $A$.  Since  the cover $(U_j)_{j\in J}$ is a refinement of  $(\pi^{-1}(O_{\alpha}))$,  we choose for each $j$ an index $\alpha (j)$ such that  $U_j\subset \pi^{-1}(O_{\alpha (j)})$. Then for each $\alpha\in A$,  we define $g_{\alpha}=\sum_{j, \,  \alpha (j)=\alpha}f_{j}$. If there is no  index $j$ satisfying  $\alpha (j)=\alpha$,  we set $g_{\alpha}=0$.  Every $g_{\alpha}$ is smooth and invariant under  the morphisms, satisfies $0\leq g_{\alpha}\leq 1$ and $\supp g_{\alpha}\subset \pi^{-1}(O_{\alpha})$.  In addition, $\sum_{\alpha\in A}g_{\alpha}=\sum_{j\in J}f_j=1$.  Consequently, $(g_\alpha )_{\alpha \in A}$ is a desired sc-smooth partition of unity.
\end{proof}

\subsection{Submanifolds of M-Polyfolds}\label{sub}

In \cite{HWZ2}, Definition 3.19,  we have introduced  the concept  of a strong finite dimensional  submanifold of an
M-polyfold. It  carries  the structure
of a manifold in a natural way.
In \cite{HWZ3} we have introduced the more general notion of a finite dimensional submanifold which we recall here for the convenience of the reader. Again, submanifolds according to the new
definition will have natural manifold structures and,  moreover,  strong
finite dimensional submanifolds  are also submanifolds according to the new definition.
The manifold structures induced in both cases are the same.
\begin{defn}
Let $X$ be an M-polyfold and $M\subset X$ a subset  equipped with the
induced topology. The subset  $M$ is called a  {\bf finite dimensional submanifold} of $X$
provided the following holds.
\begin{itemize}
\item[$\bullet$]  The subset $M$ lies in $X_\infty$.
\item[$\bullet$] At  every point $m\in M$  there exists an  M-polyfold chart
$$(U, \varphi, (\pi,E,V))
$$
where  $m\in U\subset X$ and where $\varphi:U\rightarrow
O$ is a homeomorphism satisfying $\varphi (m)=0$, onto the  open neighborhood $O$ of $0$ in
the splicing core $K$ associated with the sc-smooth splicing $(\pi, E, V)$. Here $V$ is an open neighborhood of $0$ in a
partial quadrant $C$ of the sc-Banach space $ W$. Moreover,  there
exists a finite-dimensional smooth linear subspace $N\subset W\oplus
E$ in good position to $C$ and a  corresponding sc-complement
$N^\perp$, an open neighborhood $Q$ of $0\in C\cap N$ and an
$\ssc$-smooth map $A:Q\rightarrow N^\perp$ satisfying  $A(0)=0$, $DA(0)=0$
so that the map
$$
\Gamma:Q\rightarrow W\oplus E:q\rightarrow q+A(q)
$$
has its image in $O$ and the image of the composition
$\Phi:=\varphi^{-1}\circ\Gamma:Q\rightarrow U$ is equal to  $M\cap U$.
\item[$\bullet$] The map $\Phi:Q\rightarrow M\cap U$ is a homeomorphism.
\end{itemize}
We call the map $\Phi:Q\rightarrow U$ a {\bf good parametrization}
of a neighborhood of $m\in M$ in $M$.
\end{defn}
In  other words,   a  subset $M\subset X$   of an  M-polyfold $X$  consisting of smooth points
is a submanifold if for every $m\in M$ there is  a good
parametrization of an open neighborhood of $m$ in $M$. The following
proposition shows that the transition maps  $\Phi\circ \Psi^{-1}$ defined  by two good
parameterizations $\Phi$ and $\Psi$  are smooth, so the inverses of the good parametrizations define an atlas of smoothly compatible charts.  Consequently, a finite dimensional submanifold is in a natural way a manifold with boundary with corners.

\begin{prop} Any two parametrizations of a finite dimensional submanifold $M$ of the M-polyfold $X$ are smoothly compatible.
\end{prop}
\begin{proof}
Assume that $m_0:=\varphi^{-1}(q_0+A(q_0))=\psi^{-1}(p_0+B(p_0))$
for two good parameterizations. Since both good parameterizations
are local homeomorphisms onto an open neighborhood of $m_0$ in $M$,
we obtain a local homeomorpism $O(p_0)\rightarrow O(q_0)$, $p\mapsto
q(p)$, where the domain and codomain are relatively open
neighborhoods in partial quadrants. We have
$$
q(p)+A(q(p))=\varphi\circ\psi^{-1}(p+B(p)).
$$
Recall that $q(p)\in N$ and $A(q(p)\in N^\perp$, where $N\oplus
N^\perp=W\oplus E$ is an sc-splitting. If  $P:W\oplus E\rightarrow
N$ is an  sc-projection along $N^{\perp}$, then
$$
q(p)= P(\varphi\circ\psi^{-1}(p+B(p)).$$
The map $p\rightarrow q(p)$ is sc-smooth as a  composition of sc-smooth
maps. However,  since the domain and codomain
lie in finite dimensional smooth subspaces, the map is smooth in the usual sense.
\end{proof}

\subsection{Orientations and Determinants}\label{oranddet}

We begin by recalling determinants of linear Fredholm
operators. More details and proofs  can be found in \cite{DK} and   \cite{FH} and,  particularly,  relevant in our context, in \cite{HWZ3} and  \cite{HWZ8}.

The {\bf determinant} of a linear Fredholm operator $T:E\to E$ between two Banach spaces is  the one-dimensional real vector space
$$
\det(T)=(\wedge^{max}\ker(T))\otimes
(\wedge^{\max}\text{cokern}(T))^\ast
$$
where  the star $\ast$ refers to the dual space. The orientation of the Fredholm operator $T$ is, by definition, the orientation of the vector space $\det (T)$.
If  $\Gamma:X\rightarrow {\mathcal F}(E,F)$  is a continuous family of Fredholm operators $T:E\to F$, then the bundle
$$
\det(\Gamma)=\bigcup_{x\in X}\{x\}\times\det(\Gamma (x))\to X
$$
carries the structure of a topological  line-bundle. The
same result holds true if  the domains and targets of the operators vary
in vector bundles.

We shall consider now a Fredholm section $f$ of the  strong polyfold bundle
$p:W\rightarrow Z$ and let  $P:E\rightarrow X$ be a model for $p$ over the ep-groupoid $X$ and let  the proper Fredholm section
$F:X\to E$ be the corresponding representative of the section  $f$. As shown in \cite{HWZ2}
and \cite{HWZ3} there is a well-defined notion of a linearisation of
$F$ at a smooth point $x\in X$. It is a class of sc-Fredholm operator differing by $\ssc^+$-operators and defined as follows.  If $x$ is the smooth point, we take a germ of $\ssc^+$-sections $s$ defined near $x$ and satisfying
$$F(x)=s(x).$$

Then $(F-s)(x)=0$ so that  the linearisation $(F-s)'(x):T_xX\rightarrow E_x$ is a well-defined.  It is a classical Fredholm operator and hence possesses
the  determinant $\det((F-s)'(x))$.  It depends clearly on the choice of the $\ssc^+$-section $s$. If  $s_1$ is another such $\ssc^+$-section, then the  linearisations $(F-s)'(x)$ and $(F-s_1)'(x)$ differ only  by an $\ssc^+$-operator and therefore  have, in particular, the same Fredholm index.
The space of all these  linearisations at the point $x$ is a contractible convex space, and we denote the collection of the associated determinants by $\text{DET}(F, x)$.  Consequently,  if the determinant of one of these Fredholm operators is oriented, then,
by continuation,  every  other linearisation becomes canonically oriented as well.

In order to  describe the  behavior of the determinants under the morphismsms, we look at the  morphism $\varphi:x\rightarrow y$ between two smooth points in $X$.
We choose open neighborhoods ${\bf U}(\varphi )\subset {\bf X}$ of the morphismsm and $U(x)$ and $U(y)$ of the  points $x, y\in X$ (as small as necessary) such that the source and the target maps $s:{\bf U}(\varphi )\to U(x)$ and $t:{\bf U}(\varphi )\to U(y)$ are sc-smooth diffeomorphisms.  Then  associated with the morphism $\varphi$ is the sc-diffeomorphism $\sigma:U(x)\to U(y)$ defined by  $\sigma:=t\circ s^{-1}$  and satisfying $\sigma (x)=y$.  With an $\ssc^+$-germ $s_0:U(x)\to E$ satisfying $s_0(x)=F(x)$ we associate the $\ssc^+$-germ $s_1:U(y)\to E$ defined by
$$s_1(\sigma (z)):=\mu (s^{-1}(z), s_0(z))$$
for $z\in U(x)$. Since $F:X\to E$ is a functor, it satisfies the identity
$$F(\sigma (z))=\mu (s^{-1}(z), F(z)).$$
It follows that $s_1(y)=F(y)$, and for for the linearizations we deduce
$$(F-s_1)(y)\circ T\sigma (x)\cdot h=\mu (\varphi, (F-s_0)'(x)\cdot h.$$
Abbreviating the linear isomorphism
$$\Phi:=\mu (\varphi , \cdot ):E_x\to E_y$$
and recalling  that, by definition, $T\varphi=T\sigma (x)$ one obtains
$$
\Phi^{-1}\circ (F-s_1)'(y)\circ T\varphi = (F-s_0)'(x).
$$
This  formula shows that the morphism $\varphi:x\to y$ determines  the natural isomorphism
$$
\varphi_\ast:\det((F-s_0)'(x))\rightarrow \det((F-s_1)'(y))
$$
between the determinants of the linearizations at the source and the target of $\varphi.$

In order to define the continuation of the orientations along an sc-smooth curve, we consider the sc-smooth path  $\Theta:[0,1]\rightarrow X$ connecting the smooth point $\Theta (0)$ with the smooth point $\Theta (1)$. It is easy to construct an sc-smooth family of germs
of $\ssc^+$-sections $s_t$ satisfying
$$s_t(\Theta(t))=F(\Theta(t)),$$ where
$s_t$ is defined on an open neighborhood of the point  $\Theta(t)$. Using the Fredholm property of the section $F$ one proves  in \cite{HWZ8}, that the bundle
$$
\bigcup_{t\in [0,1]} \{t\}\times \det((F-s_t)'(\Theta(t)))\rightarrow [0,1]
$$
has in a natural way the structure of a real topological  line-bundle over $[0,1]$. (We should point out  that this  is not entirely trivial because  the linearizations do not depend continuously as bounded operators on the points $\Theta (t)$ at which the section is  linearized.) As consequence we can define the continuation of the  orientations  of the linerisations  along an sc-smooth path in $X$.

\begin{defn}
An {\bf orientation for the sc-Fredholm section  $F$ of the strong bundle $P:E\to X$ }
 consists of an orientation for the linearisation of $F$ at every smooth point $x\in X$,  which is invariant under the morphisms
and stable under continuation along sc-smooth paths in $X$.
\end{defn}

Two  oriented  sc-Fredholm sections $F$ of $P:E\to X$ and $F'$ of $P':E'\to X'$  are called  {\bf equivalent},
if there is a common bundle refinement pulling back the sections to the same oriented section.  Since bundle equivalences induce isomorphisms between the determinants, we can therefore define an orientation for a Fredholm section $f$ of a strong polyfold bundle $p:W\to Z$.

If $F$ is an oriented Fredholm section and  $(F,\Lambda)$ is a transversal pair and $\Lambda(F(x))>0$,
then the linearisation  at the smooth solution  $x$ is
an intrinsic finite collection of sc-Fredholm operators
which all differ by  $\ssc^+$-operators.  Therefore, we obtain an
an orientation for the manifolds of the local solution structure. In view of the  compatibility with morphisms and the  stability under continuation, the solution set inherits the  structure of an oriented  branched suborbifold. The details and the proofs will be carried out in \cite{HWZ8}.

\end{document}